\documentclass[a4paper, intlimits, reqno]{amsart}

\usepackage[english]{babel}
\usepackage[T1]{fontenc} 
\usepackage[utf8]{inputenc}

\usepackage{amsmath}
\usepackage{amssymb}
\usepackage{MnSymbol}
\usepackage{amsthm}
\allowdisplaybreaks
\usepackage{amsfonts}
\usepackage{mathrsfs} 
\usepackage{mathtools}
\usepackage{nicefrac}
\usepackage{enumitem}
\usepackage{multicol}
\usepackage{url}
\usepackage{dsfont}
\usepackage[numbers,sort&compress]{natbib}
\usepackage{caption}
\usepackage{tikz}
\usetikzlibrary{patterns,calc,arrows}
\tikzstyle{arrow}=[draw, -latex]
\usepackage{undertilde}
\usepackage{prettyref}
\usepackage{doi}
\usepackage{appendix}

\newrefformat{def}{Definition \ref{#1}}
\newrefformat{rem}{Remark \ref{#1}}
\newrefformat{sect}{Section \ref{#1}}
\newrefformat{prop}{Proposition \ref{#1}}
\newrefformat{thm}{Theorem \ref{#1}}
\newrefformat{cor}{Corollary \ref{#1}}
\newrefformat{ex}{Example \ref{#1}}
\newrefformat{fig}{Figure \ref{#1}}
\newrefformat{app}{Appendix \ref{#1}}

\swapnumbers
\newtheoremstyle{dotless}{}{}{\itshape}{}{\bfseries}{}{}{}
\theoremstyle{dotless}
\theoremstyle{plain}
\newtheorem{thm}{Theorem}[section]
\newtheorem{lem}[thm]{Lemma}
\newtheorem{prop}[thm]{Proposition}
\newtheorem{cor}[thm]{Corollary}
\theoremstyle{definition}
\newtheorem{defn}[thm]{Definition}
\newtheorem{rem}[thm]{Remark}
\newtheorem{exa}[thm]{Example}

\newcommand{\N} {\mathbb{N}}
\newcommand{\R} {\mathbb{R}}
\newcommand{\C} {\mathbb{C}}
\newcommand{\h} {\mathbb{H}}

\DeclareMathOperator{\re}{Re}
\DeclareMathOperator{\im}{Im}
\providecommand{\differential}{\mathrm{d}}
\renewcommand{\d}{\differential}

\newcommand{\vertiii}[1]{{\left\vert\kern-0.25ex\left\vert\kern-0.25ex\left\vert #1 
    \right\vert\kern-0.25ex\right\vert\kern-0.25ex\right\vert}}  

\begin{document}

\title[Asymptotic Fourier and Laplace transforms]{Asymptotic Fourier and Laplace transforms 
for vector-valued hyperfunctions}
\author[K.~Kruse]{Karsten Kruse}
\thanks{K.~Kruse acknowledges the support by the Deutsche Forschungsgemeinschaft (DFG) within the Research Training
 Group GRK 2583 \glqq{}Modeling, Simulation and Optimization of Fluid Dynamic Applications\grqq{}.}
\address{Hamburg University of Technology\\ Institute of Mathematics \\
Am Schwarzenberg-Campus~3 \\
21073 Hamburg \\
Germany}
\email{karsten.kruse@tuhh.de}



\date{\today}
\begin{abstract}
We study Fourier and Laplace transforms for Fourier hyperfunctions with values in a 
complex locally convex Hausdorff space. Since any hyperfunction with values in a wide class of 
locally convex Hausdorff spaces can be extended to a Fourier hyperfunction, this gives simple notions 
of asymptotic Fourier and Laplace transforms for vector-valued hyperfunctions, 
which improves the existing models of Komatsu, B\"aumer, Lumer and Neubrander and Langenbruch.  
\end{abstract}

\maketitle

\section{Introduction}
The Fourier and Laplace transforms are frequently used in engineering, physics and analysis, in particular, 
in the theory of partial differential and convolution equations and in the theory of semigroups, see the works of 
Paley and Wiener \cite{paleywiener1934}, Titchmarsh \cite{titchmarsh1948}, 
Schwartz \cite{schwartz1947_48,schwartz1952}, Bracewell \cite{bracewell2000}, 
H\"ormander \cite{H1}, Doetsch \cite{doetsch1950,doetsch1972,doetsch1973}, Widder \cite{widder1946}, 
Sebasti\~{a}o e Silva \cite{silva1958}, Ku\v{c}era \cite{kucera1969}, \={O}uchi \cite{Ouchi2}, 
\v{Z}harinov \cite{zharinov1980}, Kunstmann \cite{kunstmann2001}, Arendt, Batty, Hieber and Neubrander \cite{arendt2011} 
and the references given in Deakin \cite{deakin1992}.

The main obstacle to apply the Fourier or Laplace transform is that they require exponential bounds, for instance, 
for the Laplace transform  
\[
L(f)(\zeta):=\int_{0}^{\infty}f(t)e^{-t\zeta } \d t
\]
of a locally integrable function $f\in L_{\operatorname{loc}}^{1}([0,\infty[)$ to converge 
(absolutely on the right halfplane), it is needed that $|f(t)|\leq Ce^{Ht}$ on $[0,\infty[$ for some constants $C,H>0$. 
In order to circumvent this difficulty and to use these tools with no restrictions on the growth of the functions, 
several different approaches have been made to extend the transforms to asymptotic versions without losing the 
properties that are necessary in applications.  

Vignaux \cite{vignaux1939} introduced in 1939 the concept of an asymptotic Laplace transform which was further
investigated by Ditkin \cite{ditkin1958}, Berg \cite{berg1962}, Lyubich \cite{lyubich1966} and Deakin \cite{deakin1994}. 
Based on these works, B\"aumer \cite{baeumer1997,baeumer2001,baeumer2003}, Lumer and Neubrander
\cite{lumer1999,lumer2001} and Mihai \cite{mihai2004} developed and investigated an asymptotic Laplace transform 
$\mathfrak{L}_{\operatorname{LN}}$ on the space $L_{\operatorname{loc}}^{1}([0,\infty[,E)$ 
of locally Bochner integrable functions on $[0,\infty[$ with values in a complex Banach space $(E,\|\cdot\|_{E})$ 
by studying the asymptotic behaviour of the local Laplace transforms
\begin{equation}\label{eq:asymptotic_behav_laplace}
L_{j}(f)(\zeta):=\int_{0}^{j}f(t)e^{-t\zeta } \d t,\quad j\to\infty.
\end{equation}

Another approach is due to Komatsu \cite{Kom3,Kom4,komatsu1988,Kom5,komatsu1994,Kom6}. Let $\mathcal{O}(\Omega,E)$ 
be the space of holomorphic functions on an open set $\Omega\subset\C$ with values in a complex locally convex 
Hausdorff space $E$.
Denoting by $\mathcal{D}:=\C\cup\{\infty e^{i\varphi}\;|\;|\varphi|\leq\pi\}$ 
the radial compactification of $\C$, he defined for $a\in\R$ and a complex Banach space $E$ 
a Laplace transform on the space $\mathcal{H}^{\operatorname{exp}}(\mathcal{D}\setminus [a,\infty],E)$ 
of $E$-valued holomorphic functions $F$ on $\C\setminus[a,\infty[$ which are of exponential type on 
each closed sector $\Sigma\subset\C\setminus[a,\infty[$, i.e.\ 
for each $\Sigma$ there are some constants $C,H>0$ such that $\|F(z)\|_{E}\leq Ce^{H|z|}$ for all $z\in\Sigma$.
This Laplace transform is given by 
\begin{equation}\label{eq:laplace_kom}
\mathfrak{L}_{\operatorname{Kom},a}(F)(\zeta):=\int_{\Gamma_{a}}F(z)e^{-z\zeta} \d z
\end{equation}
and converges near the half-circle $S_{\infty}:=\{\infty e^{i\varphi}\;|\;|\varphi|<\tfrac{\pi}{2}\}$ at $\infty$ 
for $F\in\mathcal{H}^{\operatorname{exp}}(\mathcal{D}\setminus [a,\infty],E)$ where $\Gamma_{a}$ is a path composed 
of the ray from $\infty e^{i\alpha}$, $-\tfrac{\pi}{2}<\alpha<0$, to a point $c<a$ and the ray from $c$ to 
$\infty e^{i\beta}$, $0<\beta<\tfrac{\pi}{2}$ (see \cite[p.\ 216]{Kom5}).
\begin{center}
\begin{minipage}{\linewidth}
\centering
\begin{tikzpicture}
\draw[->,thick,blue] (6,-2) -- (2,-6/7); 
\draw[thick,blue] (2,-6/7) -- (-1,0);
\draw[->,thick,blue] (-1,0) -- (2,6/7);
\draw[thick,blue] (2,6/7) -- (6,2);
\draw (1,0.57) arc(15.94:-15.94:2.08);
\node[anchor=north east] (A) at (5.5,1.75) {\color{blue}$\Gamma_{a}$};
\node[anchor=north east] (B) at (1,0.525) {$\beta$};
\node[anchor=north east] (C) at (1,-0.05) {$\alpha$};
\draw(-1 cm,1pt)--(-1cm,-1pt)node[anchor=north] {$c$};
\draw(2 cm,1pt)--(2cm,-1pt)node[anchor=north] {$a$};
\draw[->] (-2,0) -- (6,0) node[right] {$\re(z)$} coordinate (x axis);
\draw[red,thick] (2,0)--(6,0);
\draw[->] (0,-2) -- (0,2) node[above] {$\im(z)$} coordinate (y axis);
\end{tikzpicture}
\end{minipage}
\captionsetup{type=figure}
\caption{Path $\Gamma_{a}$ for $a\in\R$}
\end{center}

Due to Cauchy's integral theorem the kernel of this map is the space $\mathcal{H}^{\operatorname{exp}}(\mathcal{D},E)$
of entire $E$-valued functions which are of exponential type on each closed sector $\Sigma\subset\C$, and denoting by 
$\mathcal{B}_{[a,\infty]}^{\operatorname{exp}}(E):=\mathcal{H}^{\operatorname{exp}}(\mathcal{D}\setminus [a,\infty],E)/
\mathcal{H}^{\operatorname{exp}}(\mathcal{D},E)$ the space of Komatsu's Laplace hyperfunctions with support 
in $[a,\infty]$ the map
\[
\mathfrak{L}_{\operatorname{Kom},a}\colon \mathcal{B}_{[a,\infty]}^{\operatorname{exp}}(E)\to 
\mathfrak{L}_{\operatorname{Kom}}\mathcal{B}^{\operatorname{exp}}_{[a,\infty]}(E),\;
\mathfrak{L}_{\operatorname{Kom},a}([F]):=\mathfrak{L}_{\operatorname{Kom},a}(F),
\]
becomes a linear isomorphism where the range 
$\mathfrak{L}_{\operatorname{Kom}}\mathcal{B}^{\operatorname{exp}}_{[a,\infty]}(E)$ 
is the space of germs of holomorphic functions of exponential type $0$ near the half-circle $S_{\infty}$.
Setting $\mathcal{H}^{\operatorname{exp}}(\mathcal{D}\setminus\{\infty\},E):=
\mathcal{H}^{\operatorname{exp}}(\mathcal{D}\setminus [0,\infty],E)\cap\mathcal{O}(\C,E)$ and the space 
$\mathcal{B}_{\{\infty\}}^{\operatorname{exp}}(E):=\mathcal{H}^{\operatorname{exp}}(\mathcal{D}\setminus\{\infty\},E)/
\mathcal{H}^{\operatorname{exp}}(\mathcal{D},E)$ of Laplace hyperfunctions with support at $\infty$,
Komatsu proved that every Banach-valued hyperfunction with support in $[0,\infty[$ can be extended to a 
Laplace hyperfunction with support in $[0,\infty]$ and that such an extension is unique only up to 
$\mathcal{B}_{\{\infty\}}^{\operatorname{exp}}(E)$, i.e.\ the canonical (restriction) map
\[
\rho_{0}\colon \mathcal{B}_{[0,\infty]}^{\operatorname{exp}}(E)/\mathcal{B}_{\{\infty\}}^{\operatorname{exp}}(E)
\to \mathcal{B}([0,\infty[,E),\; [F]\mapsto [F],
\]
is a linear isomorphism where $\mathcal{B}([0,\infty[,E):=\mathcal{O}(\C\setminus [0,\infty[,E)/
\mathcal{O}(\C,E)$ is the space of $E$-valued hyperfunctions with support in $[0,\infty[$. 
Therefore his asymptotic Laplace transform 
\begin{gather*}
\mathfrak{L}_{\operatorname{Kom}}\colon \mathcal{B}([0,\infty[,E)\to 
\mathfrak{L}_{\operatorname{Kom}}\mathcal{B}^{\operatorname{exp}}_{[0,\infty]}(E)/
\mathfrak{L}_{\operatorname{Kom}}\mathcal{B}^{\operatorname{exp}}_{\{\infty\}}(E),\\
\mathfrak{L}_{\operatorname{Kom}}(f):=[(\mathfrak{L}_{\operatorname{Kom},0}\circ \rho_{0}^{-1})(f)],
\end{gather*}
is a linear isomorphism where 
$\mathfrak{L}_{\operatorname{Kom}}\mathcal{B}^{\operatorname{exp}}_{\{\infty\}}(E)$ 
is the space of exponentially decreasing holomorphic germs near $S_{\infty}$ 
(see \prettyref{sect:Asymptotic_Laplace_transform_Lan} for the precise definitions of the spaces 
$\mathfrak{L}_{\operatorname{Kom}}\mathcal{B}^{\operatorname{exp}}_{[0,\infty]}(E)$ and 
$\mathfrak{L}_{\operatorname{Kom}}\mathcal{B}^{\operatorname{exp}}_{\{\infty\}}(E)$). 
By \cite[Eq.\ (15), p.\ 157]{lumer2001} the two asymptotic Laplace transforms $\mathfrak{L}_{\operatorname{LN}}$ and 
$\mathfrak{L}_{\operatorname{Kom}}$ coincide on $L_{\operatorname{loc}}^{1}([0,\infty[,E)$.
The advantage of Komatsu's asymptotic Laplace transform is that it can be applied to any Banach-valued hyperfunction 
with support in $[0,\infty[$ and thus to a wider class of (generalised) 
functions, e.g.\ to distributions $\mathcal{D}'([0,\infty[,E)$ with support in $[0,\infty[$ too. 

A third approach is due to Langenbruch \cite{langenbruch2011}. Distilling from the discussions in 
\cite[Section 4]{komatsu1988}, \cite[Section 2]{lumer2001} (we may add \cite[p.\ 80-81]{baeumer1997}) 
and \cite[p.\ 429]{kunstmann2008} on the appropriate way to define an asymptotic Laplace transform, 
he states in \cite[p.\ 42]{langenbruch2011} that a satisfactory theory should meet the following conditions:
\begin{enumerate}
\item[(I)] The model contains a wide class of generalized functions and is based on an elementary version of the 
Laplace transform defined on a space of generalized functions which has a simple topological structure.
\item[(II)] For (generalized) functions with compact support, the Laplace transform should coincide with the 
Fourier-Laplace transform. Moreover, the Laplace transform should be compatible with convolution and multiplication 
by (a large class of) functions.
\item[(III)] The Laplace transform should be \emph{asymptotic}, i.e.\ in applications calculations should be needed 
near $S_{\infty}$ only.
\end{enumerate}

One of the reasons to include condition (III) is \={O}uchi's characterisation of the existence of 
hyperfunction fundamental solutions for abstract Cauchy problems in locally convex spaces 
\cite[Theorem 2, p.\ 139]{Ouchi1}
(see Komatsu's \cite[Theorem 4, p.\ 221]{Kom5} and Doma\'nski's and Langenbruch's 
\cite[Theorem 4.6, p.\ 262]{domanskilangenbruch2012} as well).
Lumer and Neubrander's asymptotic Laplace transform $\mathfrak{L}_{\operatorname{LN}}$ does not fulfil condition (I)
since it is restricted to $L_{\operatorname{loc}}^{1}([0,\infty[,E)$ even though it is sketched in 
\cite[Section 4]{lumer2001} how to extend $\mathfrak{L}_{\operatorname{LN}}$ to $\mathcal{B}([0,\infty[,E)$ by 
first interptreting \eqref{eq:asymptotic_behav_laplace} for hyperfunctions with compact support and then using 
the flabbiness of the sheaf of Banach-valued hyperfunctions. 
Komatsu's asymptotic Laplace transform $\mathfrak{L}_{\operatorname{Kom}}$ fulfils condition (II) by 
\cite[Theorem 3.3, p.\ 818]{Kom3} (compact support), \cite[Section 3]{Kom4}, 
\cite[Section 4, 5]{Kom5} (convolution), \cite[Section 2, 3]{komatsu1994} (multiplication) and (III) as discussed above. 
But it only fulfils (I) partially because the spaces $\mathcal{H}^{\operatorname{exp}}(\cdot,E)$ 
have no simple topological structure (we even did not specify a topology on them). 

Langenbruch's approach fulfils all three conditions. We only describe it briefly and go into the details in 
\prettyref{sect:Asymptotic_Laplace_transform_Lan}. Instead of exponentially increasing holomorphic 
functions Langenbruch uses exponentially decreasing functions of type $-\infty$ as key components. 
Namely, he considers the space 
\[
\mathcal{G}([0,\infty]):=\mathcal{H}_{-\infty}(\overline{\C}\setminus [0,\infty])/
\mathcal{H}_{-\infty}(\overline{\C}).
\]
of exponentially decreasing hyperfunctions of type $-\infty$ supported in
$[0,\infty]$ where for $K=[0,\infty]$ or $K=\emptyset$
\[
\mathcal{H}_{-\infty}(\overline{\C}\setminus K):= \{f\in\mathcal{O}(\C\setminus K)
\; | \; \forall\;n\in\N:\;\sup_{z\in W_{n}(K)}|f(z)|e^{n|\re(z)|} < \infty\}
\]
with $\overline{\C}:= [-\infty,\infty]+i\R$ and 
\[
W_{n}(K):=\{z\in\C\;|\;|\im(z)|\leq n,\,\operatorname{dist}(z,K\cap\C)>\tfrac{1}{n}\}.
\]
The Laplace transform $\mathfrak{L}$ of an element 
$f=[F]\in\mathcal{G}([0,\infty])$ is then defined by formula \eqref{eq:laplace_kom} 
where the path $\Gamma_{0}$ is replaced by the path $\gamma_{[0,\infty]}$ depicted below.
\begin{center}
\begin{minipage}{\linewidth}
\centering
\begin{tikzpicture}
\draw[blue,very thick]  (0,1.5) arc (90:270:15mm);
\draw[blue,very thick] (2,1.5) -- (4,1.5) node [above=0.1pt] {\color{blue}$\gamma_{[0,\infty]}$};
\draw[->,blue,very thick] (0,1.5) -- (2.2,1.5);
\draw[blue,very thick] (2,-1.5) -- (0,-1.5);
\draw[->,blue,very thick] (4,-1.5) -- (1.8,-1.5);
\draw(-1.5 cm,1pt)--(-1.5cm,-1pt)node[anchor=north east] {$-c$};
\draw(1pt,1.5 cm)--(-1pt,1.5 cm)node[anchor=south east] {$c$};
\draw(1pt,-1.5 cm)--(-1pt,-1.5 cm)node[anchor=north east] {$-c$};
\draw[->] (-4,0) -- (4,0) node[right] {$\re(z)$} coordinate (x axis);
\draw[red,thick] (0,0)--(4,0);
\draw[->] (0,-2) -- (0,2) node[above] {$\im(z)$} coordinate (y axis);
\end{tikzpicture}
\end{minipage}
\captionsetup{type=figure}
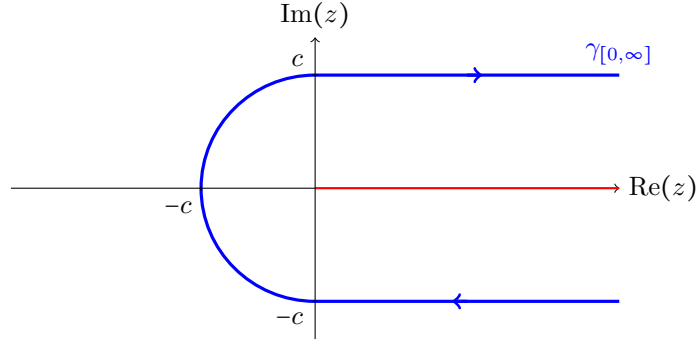
\caption{Path $\gamma_{[0,\infty]}$ with $c>0$}
\end{center}
The Laplace transform
\[
\mathfrak{L}\colon \mathcal{G}([0,\infty])\to \mathfrak{L}\mathcal{G}_{[0,\infty]}
\]
is a topological isomorphism by \cite[Theorem 4.1, p.\ 53]{langenbruch2011} where 
the Laplace range $\mathfrak{L}\mathcal{G}_{[0,\infty]}$ is given by 
\[
\mathfrak{L}\mathcal{G}_{[0,\infty]}:=\{f\in\mathcal{O}(\C)\;|\;\forall\;k\in\N:\; 
\sup_{\re(z)\geq -k}|f(z)|e^{-\frac{1}{k}|z|}<\infty\}.
\]
Every hyperfunction with support in $[0,\infty[$ can be extended to an exponentially decreasing hyperfunction 
of type $-\infty$ supported in $[0,\infty]$ and such an extension is unique only up to $\mathcal{G}(\{\infty\})$ 
by \cite[Theorem 5.1, p.\ 54]{langenbruch2011}, i.e.\ the canonical (restriction) map 
\[
R_{+}\colon \mathcal{H}_{-\infty}(\overline{\C}\setminus [0,\infty])/
\mathcal{H}_{-\infty}(\overline{\C}\setminus\{\infty\})
\to \mathcal{B}([0,\infty[),\;[F]\mapsto [F],
\]
is a linear isomorphism where 
$\mathcal{H}_{-\infty}(\overline{\C}\setminus\{\infty\}):= 
\mathcal{H}_{-\infty}(\overline{\C}\setminus [0,\infty])\cap \mathcal{O}(\C)$ and 
$\mathcal{G}(\{\infty\}):=\mathcal{H}_{-\infty}(\overline{\C}\setminus\{\infty\})/\mathcal{H}_{-\infty}(\overline{\C})$ 
is the space of exponentially decreasing hyperfunctions of type $-\infty$ supported at $\infty$.
The Laplace transform 
\[
\mathfrak{L}\colon \mathcal{G}(\{\infty\})\to \mathfrak{L}\mathcal{G}_{\{\infty\}}
\]
is also a topological isomorphism by \cite[Proposition 5.2, p.\ 55]{langenbruch2011} where
\[
\mathfrak{L}\mathcal{G}_{\{\infty\}}:=\{f\in\mathcal{O}(\C)\;|\;\forall\;k\in\N:\; 
\sup_{\re(z)\geq -k}|f(z)|e^{k|\re(z)|-\frac{1}{k}|z|}<\infty\}.
\]
Combining these results, Langenbruch's asymptotic Laplace transform
\[
\mathfrak{L}_{\mathcal{B}}\colon \mathcal{B}([0,\infty[,\C)\to 
\mathfrak{L}\mathcal{G}_{[0,\infty]}/\mathfrak{L}\mathcal{G}_{\{\infty\}},\;
\mathfrak{L}_{\mathcal{B}}(f):=[(\mathfrak{L}\circ R_{+}^{-1})(f)],
\]
is a linear isomorphism by \cite[Theorem 5.3, p.\ 55]{langenbruch2011}.

This asymptotic Laplace transform fulfils condition (I). In particular, the spaces 
$\mathcal{H}_{-\infty}(\overline{\C}\setminus K)$ are nuclear Fr\'echet spaces for $K=\emptyset$, $\{\infty\}$ 
and $[0,\infty]$. It fulfils condition (II) by \cite[Section 5]{langenbruch2011} and also condition (III) because 
it coincides up to a linear isomorphism with Komatsu's asymptotic Laplace transform $\mathfrak{L}_{\operatorname{Kom}}$
for $E=\C$ by \cite[Theorem 6.3, p.\ 59]{langenbruch2011}. 

By using the $\pi$-tensor product for nuclear Fr\'echet spaces, Langenbruch's asymptotic Laplace transform
may be extended to the space $\mathcal{B}([0,\infty[,E)$ of $E$-valued hyperfunctions with support in $[0,\infty[$ 
where $E$ is a complex Fr\'echet space (see \cite[p.\ 61]{langenbruch2011}). 
This is needed for applications to the abstract Cauchy problem in Fr\'echet spaces 
(see \cite[Section 7]{langenbruch2011}). 
Since abstract Cauchy problems are of interest in more general locally convex spaces $E$, like distributions, as well 
(see \cite{albanese2016,domanskilangenbruch2012,jacob2015_1,jacob2015_2,komatsu1964,komura1968,
lobanov1994,Ouchi1,Ouchi2,teichmann2002} and the references therein), 
we may add a fourth condition that a satisfactory theory of Laplace transforms should meet:
\begin{enumerate}
\item[(IV)] The Laplace transform should be applicable to (generalized) functions with values in a wide class 
of locally convex Hausdorff spaces, containing most of the common spaces in analysis. 
\end{enumerate}
This fourth condition is only partially fulfilled by the asymptotic Laplace transforms 
$\mathfrak{L}_{\operatorname{LN}}$, $\mathfrak{L}_{\operatorname{Kom}}$ and $\mathfrak{L}_{\mathcal{B}}$ 
which are restricted to Banach and Fr\'echet spaces, respectively.

It is indicated in \cite[p.\ 42]{langenbruch2011} that instead of Komatsu's Laplace hyperfunctions 
(or exponentially decreasing hyperfunctions of type $-\infty$) other hyperfunctions may be used to develop 
a satisfactory theory, like Kawai's Fourier hyperfunctions \cite{Kawai}, 
Saburi's modified Fourier hyperfunctions \cite{saburi1985},
Kunstmann's bounded hyperfunctions \cite{kunstmann2008}
or, we may add, Park's and Morimoto's Fourier ultra-hyperfunctions \cite{park1973}. 
Kawai's Fourier hyperfunctions are the route we take in the present paper.  
They were introduced by Kawai \cite{Kawai} in the complex-valued case and then extended 
by Ito and Nagamachi to Fourier hyperfunctions with values in a separable Hilbert space 
\cite{ItoNag1975_1,ItoNag1975_2} (in a general Hilbert space \cite{Ito1992_2}), by Junker to Fourier hyperfunctions
with values in a Fr\'echet space \cite{J} and then beyond Fr\'echet spaces by us in \cite{ich,kruse2019_5}.
Following in Langenbruch's footsteps \cite{langenbruch2011}, we develop a theory that satisfies all four conditions. 
In \prettyref{sect:Notation} we recall the necessary notation and results from the theory of vector-valued 
Fourier hyperfunctions that are needed. In \prettyref{sect:Fourier_half_line} and \prettyref{sect:Fourier_real_comp} 
we study the Fourier transform of vector-valued Fourier hyperfunctions supported in extended half-lines and real compact 
sets, respectively. In \prettyref{sect:Asymptotic_Fourier_transform} we introduce our asymptotic Fourier transform 
for vector-valued hyperfunctions and show that it coincides (up to a linear isomorphism) with 
Langenbruch's asymptotic Fourier transform in the case $E=\C$ in \prettyref{sect:Asymptotic_Fourier_transform_Lan}.
We use our Fourier transform to define and develop our Laplace and asymptotic Laplace transforms 
$\mathcal{L}^{\mathcal{B}}$ in \prettyref{sect:Laplace_and_Asymptotic_Laplace}. 
In \prettyref{sect:Asymptotic_Laplace_transform_Lan} we clarify the relation between the asymptotic 
Laplace transforms considered so far. Namely, we show that 
$\mathfrak{L}_{\mathcal{B}}$ and $\mathcal{L}^{\mathcal{B}}$ coincide (up to a linear isomorphism) 
for complex Fr\'echet spaces $E$, that both coincide with $\mathfrak{L}_{\operatorname{Kom}}$ for complex Banach 
spaces $E$ and thus all three coincide on $L_{\operatorname{loc}}^{1}([0,\infty[,E)$ with 
$\mathfrak{L}_{\operatorname{LN}}$. 
\section{Notation and Preliminaries}
\label{sect:Notation}
In the following $E$ is always a locally convex Hausdorff space over $\C$ equipped with a directed 
system of seminorms $(p_{\alpha})_{\alpha\in\mathfrak{A}}$, in short, $E$ is a $\C$-lcHs.
We set $(p_{\alpha})_{\alpha\in \mathfrak{A}}:=\{|\cdot|\}$ if $E=\C$ and $|\cdot|$ the Euclidean norm on $\C$.  
We recall that the space $E$ is called locally 
complete if every closed disk $D\subset E$ is a Banach disk (see \cite[10.2.1 Proposition, p.\ 197]{Jarchow}). 
In particular, every sequentially complete $\C$-lcHs is locally complete.
Further, we denote by $L(F,E)$ the space of continuous linear maps from 
a $\C$-lcHs $F$ to $E$ and we sometimes use the notion $\langle T, f\rangle:=T(f)$, $f\in F$,  
for $T\in L(F,E)$. If $E=\C$, we write $F':=L(F,\C)$ for the dual space of $F$. 
If $F$ and $E$ are (linearly) topologically isomorphic, we write $F\cong E$.
We denote by $L_{t}(F,E)$ the space $L(F,E)$ equipped with the locally convex topology of uniform convergence 
on the absolutely convex compact subsets of $F$ if $t=\kappa$, and on the bounded subsets of $F$ if $t=b$. 
We recall that the $\varepsilon$-product of Schwartz is defined by 
$F\varepsilon E:=L_{e}(F_{\kappa}',E)$ where $L(F_{\kappa}',E)$ is equipped with the topology of uniform convergence 
on the equicontinuous subsets of $F'$.

If $F(\Omega,E)$ is a space of (equivalence classes) of functions from a set $\Omega$ to $E$, 
we use the convention $F(\Omega):=F(\Omega,\C)$. 
We denote by $\mathcal{O}(\Omega,E)$ the space of $E$-valued holomorphic functions on an open 
set $\Omega\subset\C$ and by $\mathcal{C}^{\infty}(\Omega,E)$ the space of $E$-valued infinitely continuously 
partially differentiable functions on an open set $\Omega\subset\R^{2}=\C$. 
We denote by $\partial^{\beta}f$ the partial derivative of $f\in \mathcal{C}^{\infty}(\Omega, E)$ for a multiindex 
$\beta\in\N_{0}^{2}$. If $E$ is locally complete, then a function $f\in\mathcal{O}(\Omega,E)$ is infinitely complex 
differentiable on $\Omega$ and we write $\partial^{n}f:=f^{(n)}$ for its $n$th complex derivative where $n\in\N_{0}$. 
If we want to emphasize that we consider derivatives w.r.t.\ a complex variable, 
we write $\partial_{\C}^{n}f$ instead of just $\partial^{n}f$.
We use the symbol $\mathcal{C}(\Omega,E)$ for the space of $E$-valued continuous functions on a set $\Omega\subset\C$. 
We denote by $\overline{\R}:=\R\cup\{\pm \infty\}$ the two-point compactification of $\R$ and set 
$\overline{\C}:=\overline{\R}+i\R$. We define the distance of $z\in\C$ to a set $M\subset\C$ w.r.t.\ $|\cdot|$ 
via $\d(z,M):=\inf_{w\in M}|z-w|$ if $M\neq\emptyset$, and $\d(z,M):=\infty$ if $M=\emptyset$.
For a compact set $K\subset \overline{\R}$ and $c>0$ we define the sets
\begin{align*}
  U_{\frac{1}{c}}(K) 
&:=\phantom{\cup} \{ z \in \C \:| \: \d(z,K\cap\C)< c\} \\
&\phantom{:=}\cup
\begin{cases}
\emptyset &,\; K \subset \R, \\
]\nicefrac{1}{c},\infty[+i ]-c, c[ &,\; \infty \in K, \, -\infty \notin K, \\
]-\infty, -\nicefrac{1}{c}[+i ]-c, c[ &, \; \infty \notin K, \,  -\infty \in K,\\ 
\left(]-\infty, -\nicefrac{1}{c}[\cup ]\nicefrac{1}{c},\infty[\right)+i ]-c, c[ &,\; \infty \in K, \, -\infty \in K,
\end{cases}
\end{align*}
and for $n\in\N$ 
\[
S_{n}(K):= \left(\C\setminus\overline{U_{n}(K)}\right) \cap \{z\in\C\;|\;|\im(z)|<n\}. 
\]
We remark that $S_{n}(K)$ coincides with $W_{n}(K)$ from the introduction for $K=\varnothing$, $[0,\infty]$. 

\begin{defn}
Let $K\subset\overline{{\R}}$ be compact. The space of rapidly decreasing holomorphic germs near $K$ is defined as 
the inductive limit
\[
\mathcal{P}_{\ast}(K) := \lim\limits_{\substack{\longrightarrow\\n\in \N}}\,\mathcal{O}_{n}(U_{n}(K))
\]
with restrictions as linking and spectral maps where
\[
\mathcal{O}_{n}(U_{n}(K)) :=
\{ f \in \mathcal{O}(U_{n}(K))\cap\mathcal{C}(\overline{U_{n}(K)}) 
\; | \; \|f\|_{n,K}:=\sup_{z \in \overline{U_{n}(K)}}|f(z)|e^{\frac{1}{n}|\re(z)|} < \infty \}
\]
if $K\neq\emptyset$. Further, we set $\mathcal{P}_{\ast}(\emptyset):=0$.
\end{defn}

$\mathcal{P}_{\ast}(K)$ is a DFS-space by \cite[p.\ 469]{Kawai} resp.\ \cite[1.11 Satz, p.\ 11]{J}. 
Another common symbol for $\mathcal{P}_{\ast}(K)$ 
is $\utilde{\mathcal{O}}(K)$ and for the special case $\mathcal{P}_{\ast}(\overline{\R})$ 
the symbol $\mathcal{P}_{\ast}$ is used as well (see \cite[Definition 1.1.3, p.\ 468-469]{Kawai}). 
We remark that $\mathcal{P}_{\ast}(K)\cong\mathcal{A}(K)$ if $K\subset\R$ where $\mathcal{A}(K)$ 
is the space of germs of real analytic functions on $K$ with its inductive limit topology.

\begin{defn}[{\cite[3.2 Definition, p.\ 12-13]{ich}}]
Let $E$ be a $\C$-lcHs and $K\subset\overline{{\R}}$ be compact.
\begin{enumerate}
\item [a)] The space of vector-valued slowly increasing infinitely continuously partially differentiable 
functions outside $K$ is defined as
\[
\mathcal{E}^{exp}(\overline{\C}\setminus K, E):= \{f\in\mathcal{C}^{\infty}(\C\setminus K,E)
\; | \; \forall\;n\in\N,\,m\in\N_{0},\,\alpha\in\mathfrak{A}:\;\vertiii{f}_{n,m,\alpha,K} < \infty\}
\]
where
\[
\vertiii{f}_{n,m, \alpha,K}:=\sup_{\substack{z\in S_{n}(K)\\ \beta\in\N_{0}^{2}, |\beta|\leq m}}
p_{\alpha}(\partial^{\beta}f(z))
e^{-\frac{1}{n}|\re(z)|}.
\]
 \item [b)] The space of vector-valued slowly increasing holomorphic functions outside $K$ is defined as
\[
\mathcal{O}^{exp}(\overline{\C}\setminus K,E):= \{f\in\mathcal{O}(\C\setminus K,E)
\; | \; \forall\;n\in\N,\,\alpha\in\mathfrak{A}:\; \vertiii{f}_{n,\alpha,K} < \infty \}
\]
where
\[
\vertiii{f}_{n, \alpha,K}:=\sup_{z \in S_{n}(K)}p_{\alpha}(f(z))e^{-\frac{1}{n}|\re(z)|}.
\]
Furthermore, we denote by $\vertiii{f}_{n, \alpha,K}^{\wedge}$ the usual seminorms on the quotient space 
\[
bv_{K}(E):=\mathcal{O}^{exp}(\overline{\C}\setminus K, E)/\mathcal{O}^{exp}(\overline{\C}, E).
\]
\end{enumerate}
In the Banach-valued, particularly, scalar-valued, case the subscript $\alpha$ in the notation of the seminorms 
is omitted.
\end{defn}

We note that $S_{1}(\overline{\R})=\emptyset$ and 
$\vertiii{f}_{1,m, \alpha,\overline{\R}}=-\infty=\vertiii{f}_{1, \alpha,\overline{\R}}$ 
for any $f\colon\C\setminus\R\to E$, $m\in\N_{0}$ and $\alpha\in\mathfrak{A}$.
Other common symbols for the spaces $\mathcal{E}^{exp}(\overline{\C}\setminus K,E)$ resp.\ 
$\mathcal{O}^{exp}(\overline{\C}\setminus K,E)$ are $\widetilde{\mathcal{E}}(\overline{\C}\setminus K,E)$ resp.\ 
$\widetilde{\mathcal{O}}(\overline{\C}\setminus K,E)$ (see \cite[1.2 Definition, p.\ 5]{J}). 

\begin{defn}[{(strictly) admissible, \cite[p.\ 55]{ich}}]
A $\C$-lcHs $E$ is called \emph{admissible}, if the Cauchy-Riemann operator
\[
\overline{\partial}\colon \mathcal{E}^{exp}(\overline{\C}\setminus K,E)\to \mathcal{E}^{exp}(\overline{\C}\setminus K,E)
\]
is surjective for any compact set $K\subset\overline{\R}$. $E$ is called \emph{strictly admissible} 
if $E$ is admissible and if, in addition,
\[
\overline{\partial}\colon \mathcal{C}^{\infty}(\Omega,E)\to\mathcal{C}^{\infty}(\Omega,E)
\]
is surjective for any open set $\Omega\subset\C$.
\end{defn}

If $E$ is strictly admissible and sequentially complete, then the sheaf of $E$-valued Fourier 
hyperfunctions is flabby and can be represented on the one hand by boundary values of exponentially slowly increasing 
holomorphic functions and on the other by equivalence classes of $E$-valued $\mathcal{P}_{\ast}$-functionals 
(see \cite[Theorem 5.9, p.\ 33]{kruse2019_5}). In particular, its subsheaf of $E$-valued hyperfunctions 
is flabby under this condition as well. Moreover, we may regard $bv_{K}(E)$ as the space of 
$E$-valued Fourier hyperfunctions with support in $K\subset\overline{\R}$ under this condition 
by \cite[5.11 Lemma, p.\ 44]{kruse2019_5}.

\begin{thm}[{\cite[5.25 Theorem, p.\ 98]{ich}}]\label{thm:examples_strictly_admiss}
If
\begin{enumerate}
\item [a)] $E$ is a $\C$-Fr\'echet space, or if
\item [b)] $E:=F_{b}'$ where $F$ is a $\C$-Fr\'echet space satisfying $(DN)$, or if
\item [c)] $E$ is a complex ultrabornological PLS-space satisfying $(PA)$, 
\end{enumerate}
then $E$ is strictly admissible.
\end{thm}

The definitions of the topological invariants $(DN)$ and $(PA)$ are given in 
\cite[Chap.\ 29, Definition, p.\ 359]{meisevogt1997} and \cite[Section 4, Eq.\ (24), p.\ 577]{Dom1}, respectively. 
Besides every $\C$-Fr\'echet space, the theorem above covers the space $E=\mathcal{S}(\R^{d})'$ 
of tempered distributions, the space $\mathcal{D}(V)'$ of distributions and the space $\mathcal{D}_{(\omega)}(V)'$ 
of ultradistributions of Beurling type and many more spaces given in \cite{Dom1}, 
\cite[Corollary 4.8, p.\ 1116]{D/L}, \cite[Example 4.4, p.\ 14-15]{kruse2019_5} and \cite{vogt1983}.
Due to \cite[Theorem 6.9, p.\ 1125]{D/L} another sufficient condition for the flabbiness of the sheaf 
of $E$-valued hyperfunctions is that $E$ is complete and $2$\emph{-admissible} which means that the 
$2$-dimensional Laplace operator 
\[
\Delta\colon \mathcal{C}^{\infty}(\Omega,E)\to\mathcal{C}^{\infty}(\Omega,E)
\]
is surjective for any open set $\Omega\subset\R^{2}$ (see \cite[p.\ 1112]{D/L}). 
If $E$ is an ultrabornological PLS-space, then the three conditions strictly admissible, $2$-admissible and $(PA)$
are equivalent and actually necessary for a reasonable theory of $E$-valued (Fourier) hyperfunctions 
by \cite[Theorem 8.9, p.\ 1139]{D/L} and \cite[Theorem 5.12, p.\ 45-46]{kruse2019_5}.

The introduced spaces $bv_{K}(E)$ and $\mathcal{P}_{\ast}(K)$ 
are connected by duality in the following way, which is a generalisation
of \cite[4.1 Theorem, p.\ 41]{ich} where $E$ is complete, and the idea of its proof comes 
from \cite[Theorem 3.3, p.\ 85-86]{Mori2} where $K=[a,\infty]$, $a\in\R$, and $E=\C$.

\begin{thm}[{\cite[3.15 Corollary (i), (iii), p.\ 21]{kruse2019_2}}]\label{thm:duality}
Let $E$ be a $\C$-lcHs and $K\subset\overline{\R}$ a non-empty compact set. If 
\begin{enumerate}
\item[(i)] $K\subset\R$ and $E$ is locally complete, or if 
\item[(ii)] $E$ is sequentially complete,
\end{enumerate}
then the map
\begin{gather*}
\mathscr{H}_{K}\colon bv_{K}(E)\to L_{b}(\mathcal{P}_{\ast}(K), E),\\
\mathscr{H}_{K}([F])(\varphi):=\int_{\gamma_{K,n}}{F(\zeta)\varphi(\zeta)\d\zeta},
\end{gather*}
for $[F]\in bv_{K}(E)$ and $\varphi\in\mathcal{O}_{n}(U_{n}(K))$, $n\in\N$, is 
a topological isomorphism where $\gamma_{K,n}$ is a suitable path 
around $K$ in $U_{n}(K)$ 
and the integral is a Pettis integral. If $K$ is an interval, we may choose 
$\gamma_{K,n}$ as the boundary of $U_{\nicefrac{1}{c}}(K)$, $0<c<\tfrac{1}{n}$, with clockwise orientation. 
Its inverse 
\[
 \mathscr{H}_{K}^{-1}\colon  L_{b}(\mathcal{P}_{\ast}(K),E)\to bv_{K}(E)
\]
is given by  
\[
 \mathscr{H}_{K}^{-1}(T)=\Bigl[\C\setminus K \ni z\longmapsto 
 \frac{i}{2\pi}\bigl\langle T,\frac{e^{-(z-\cdot)^2}}{z-\cdot}\bigr\rangle\Bigr],
 \quad T\in L_{b}(\mathcal{P}_{\ast}(K),E).
\] 
\end{thm}

We note that in \cite[3.15 Corollary (i), (iii), p.\ 21]{kruse2019_2} the orientation of $\gamma_{K,n}$ 
is chosen to be counterclockwise but that only changes the sign. 
For our purpose the clockwise orientation is more suitable. 
Let
\[
\mathcal{B}(A,E):=\mathcal{O}(\C\setminus A,E)/\mathcal{O}(\C,E)
\]
be the space of hyperfunctions with values in a $\C$-lcHs E and support in a closed set $A\subset\R$. 
If $K\subset\R$ is compact, then we equip $\mathcal{O}(\C\setminus K,E)$ resp.\ $\mathcal{O}(\C,E)$ 
with the topology of uniform convergence 
on compact subsets of $\C\setminus K$ resp.\ $\C$ and $\mathcal{B}(K,E)$ with the induced quotient topology.
We note the following well-known corollary of the theorem above for the space $\mathcal{B}(K,E)$
whose $E$-valued version is unfortunatelly not contained in literature (to the best of our knowledge). 
The $\C$-valued version can be found in \cite[Theorem 2.1.3, p.\ 25]{Mori1}.

\begin{cor}\label{cor:duality_real_comp}
Let $E$ be a locally complete $\C$-lcHs and $K\subset\R$ a non-empty compact set. Then
$\mathscr{H}_{K}\colon \mathcal{B}(K,E)\to L_{b}(\mathcal{A}(K),E)$ is a topological isomorphism, 
with inverse $\mathscr{H}_{K}^{-1}$ as above. 
\end{cor}
\begin{proof}
We only sketch the proof. The continuity and injectivity of $\mathscr{H}_{K}$ are easily checked 
(like in the proof of \prettyref{thm:duality} or \cite[Theorem 2.1.3, p.\ 25]{Mori1}), the surjectivity 
follows directly from \prettyref{thm:duality} and $\mathcal{A}(K)\cong\mathcal{P}_{\ast}(K)$ for $K\subset\R$  
and actually gives that every $f\in\mathcal{B}(K,E)$ has a representing function in
$\mathcal{O}^{exp}(\overline{\C}\setminus K,E)$. The continuity of $\mathscr{H}_{K}^{-1}$ follows like in the proof 
of \prettyref{thm:duality}. 
\end{proof}

In the proof of \cite[Theorem 2.1.3, p.\ 25]{Mori1} the inverse of $\mathscr{H}_{K}$ is given 
in \cite[Eq.\ (1.6), p.\ 26]{Mori1} as
\[
 \mathscr{H}_{K}^{-1}(T)=\Bigl[\C\setminus K \ni z\longmapsto 
 \frac{i}{2\pi}\bigl\langle T,\frac{1}{z-\cdot}\bigr\rangle\Bigr],
 \quad T\in L_{b}(\mathcal{A}(K),E).
\] 
This is no contradiction to our inverse because the difference 
\[
\C\setminus K \ni z\longmapsto \frac{i}{2\pi}\bigl\langle T,\frac{e^{-(z-\cdot)^2}}{z-\cdot}\bigr\rangle
-\frac{i}{2\pi}\bigl\langle T,\frac{1}{z-\cdot}\bigr\rangle
\]
extends to a function in $\mathcal{O}(\C,E)$ (cf.\ \cite[Eq.\ (2.9), p.\ 47]{langenbruch2011}).

Next, we define the Fourier transform on $L_{b}(\mathcal{P}_{\ast}(\overline{\R}),E)$ 
which goes back to Kawai \cite{Kawai} in the case $E=\C$.
By \cite[Proposition 3.2.4, p.\ 483]{Kawai} (cf.\ \cite[Proposition 8.2.2, p.\ 376]{Kan}) 
the Fourier transform 
$\mathcal{F}\colon \mathcal{P}_{\ast}(\overline{\R})\to\mathcal{P}_{\ast}(\overline{\R})$ defined by
\[
\mathcal{F}(\varphi)(\zeta):=\widehat{\varphi}(\zeta):=\int_{\R}{\varphi(x)e^{-ix\zeta}\d x},
\quad\varphi\in\mathcal{O}_{n}(U_{n}(\overline{\R})),\,\zeta\in U_{k}(\overline{\R}),\, k>n,
\]
is a topological isomorphism and $\widehat{\varphi}\in\mathcal{O}_{k}(U_{k}(\overline{\R}))$ with $k>n$ for 
$\varphi\in\mathcal{O}_{n}(U_{n}(\overline{\R}))$. Here we follow the convention from 
\cite[Proposition 8.2.2, p.\ 376]{Kan} and use $e^{-ix\zeta}$ in the definition of $\mathcal{F}$ instead 
of $e^{ix\zeta}$ as in \cite[Proposition 3.2.4, p.\ 483]{Kawai}. 
The Fourier transform on $L_{b}(\mathcal{P}_{\ast}(\overline{\R}),E)$ is now defined by transposition and we obtain 
(cf.\ \cite[3.14 Folgerung, p.\ 45]{J}, \cite[Definition 3.2.5, p.\ 483]{Kawai}, \cite[4.6 Theorem, p.\ 53]{ich} 
and \cite[3.10 Corollary, p.\ 13]{kruse2019_5}):

\begin{thm}\label{thm:fourier_isom_extended_reals}
Let $E$ be a $\C$-lcHs. The Fourier transform
$\mathcal{F}_{\star}\colon L_{b}(\mathcal{P}_{\ast}(\overline{\R}),E)\to L_{b}(\mathcal{P}_{\ast}(\overline{\R}),E)$ 
defined by
\[
\mathcal{F}_{\star}(T)(\varphi):=\langle T,\mathcal{F}(\varphi)\rangle,
\quad T\in L(\mathcal{P}_{\ast}(\overline{\R}),E),\;\varphi\in\mathcal{P}_{\ast}(\overline{\R}),
\]
is a topological isomorphism with inverse given by
\[
\mathcal{F}_{\star}^{-1}(T)(\varphi)=\langle T,\mathcal{F}^{-1}(\varphi)\rangle,
\quad T\in L(\mathcal{P}_{\ast}(\overline{\R}),E),\;\varphi\in\mathcal{P}_{\ast}(\overline{\R}).
\]
\end{thm}

Thus we have the following reasonable definition of the Fourier transform on 
$bv_{\overline{\R}}(E)$.

\begin{cor}\label{cor:fourier_isom_extended_reals}
Let $E$ be a sequentially complete $\C$-lcHs. The Fourier transform 
$\mathcal{F}_{\overline{\R}}\colon bv_{\overline{\R}}(E)\to bv_{\overline{\R}}(E)$ 
defined by 
$\mathcal{F}_{\overline{\R}}:=\mathscr{H}_{\overline{\R}}^{-1}\circ\mathcal{F}_{\star}\circ \mathscr{H}_{\overline{\R}}$
is a topological isomorphism with inverse given by 
$\mathcal{F}_{\overline{\R}}^{-1}
=\mathscr{H}_{\overline{\R}}^{-1}\circ\mathcal{F}_{\star}^{-1}\circ \mathscr{H}_{\overline{\R}}$.
\end{cor}
\section{Fourier transform of Fourier hyperfunctions with support in an extended half-line}
\label{sect:Fourier_half_line}
In this section we study the Fourier transform $\mathcal{F}_{K}$ of vector-valued Fourier hyperunctions 
with suppport in the extended half-lines $K=[a,\infty]$ or $K=[-\infty,a]$ with $a\in\overline{\R}$ 
but $K\neq \overline{\R}$. We will use the duality given by the map $\mathscr{H}_{K}$ and Kawai's Fourier transform 
$\mathcal{F}_{\overline{\R}}$ on the whole extended real line $\overline{\R}$. 
To characterise the range of $\mathcal{F}_{K}$ we need the supporting function of $K$ and some of its properties, 
whose simple proof we omit. 

\begin{prop}\label{prop:supporting_function}
For a compact interval $K\subsetneq\overline{\R}$ such that $K\neq\pm\{\infty\}$ and $\eta\in\R$ we define the 
\emph{supporting function}
\[
 H_{K}(\eta):=\sup\limits_{x\in K} x\eta  .
\]
Let $a,b,c,d\in\R$. Then we have
\begin{enumerate}
 \item [a)] $H_{[a,b]}(\eta)=\max(a\eta,b\eta)$ for $\eta\in\R$,
 \item [b)] $H_{[a,\infty]}(\eta)=a\eta=-a|\eta|$ for $\eta<0$,
 \item [c)] $H_{[-\infty,a]}(\eta)=a\eta=a|\eta|$ for $\eta>0$,
 \item [d)] $H_{[a,b]+[c,d]}(\eta)=H_{[a,b]}(\eta)+H_{[c,d]}(\eta)$ for $\eta\in\R$. 
\end{enumerate}
\end{prop}

\begin{defn} 
 Let $E$ be a $\C$-lcHs and $\h:=\{z\in\C\;|\;\im(z)>0\}$.
 \begin{enumerate}
 \item [a)] For $-\infty<a\leq\infty$ we define the space
 \[
  \mathcal{FO}_{[a,\infty]}(E):=\{f\in\mathcal{O}(-\h,E)\;|\;\forall\;k\in\N,\,\alpha\in\mathfrak{A}:\;
   |f|_{k,\alpha,[a,\infty]}<\infty\}
 \]
 where
 \[
  |f|_{k,\alpha,[a,\infty]}:=\sup_{\im(z)\leq -\frac{1}{k}}p_{\alpha}(f(z))
  e^{-\frac{1}{k}|z|-w^{+}_{a}(\im(z))}
 \]
and for $x<0$
 \[
   w^{+}_{a}(x):=
   \begin{cases}
    H_{[a,\infty]}(x)=-a|x|&,\; a\neq \infty,\\
    H_{[k,\infty]}(x)=-k|x|&,\; a=\infty.
   \end{cases}
 \]
 \item [b)] For $-\infty\leq a <\infty$ we define the space
 \[
  \mathcal{FO}_{[-\infty,a]}(E):=\{f\in\mathcal{O}(\h,E)\;|\;\forall\;k\in\N,\,\alpha\in\mathfrak{A}:\;
   |f|_{k,\alpha,[-\infty,a]}<\infty\},
 \]
 where
 \[
  |f|_{k,\alpha,[-\infty,a]}:=\sup_{\im(z)\geq \frac{1}{k}}p_{\alpha}(f(z))
  e^{-\frac{1}{k}|z|-w^{-}_{a}(\im(z))}
 \]
and for $x>0$
 \[
   w^{-}_{a}(x):=
   \begin{cases}
    H_{[-\infty,a]}(x)=a|x|&,\; a\neq -\infty,\\
    H_{[-\infty,-k]}(x)=-k|x|&,\; a=-\infty.
   \end{cases}
 \]
 \end{enumerate}
\end{defn}

We may consider the elements of $\mathcal{FO}_{[a,\infty]}(E)$ resp.\ $\mathcal{FO}_{[-\infty,a]}(E)$ 
as elements of $\mathcal{O}^{exp}(\overline{\C}\setminus\overline{\R},E)$ by trivial extension.

\begin{prop}\label{prop:upper_lower_extension}
Let $E$ be a $\C$-lcHs.
\begin{enumerate}
 \item [a)] Let $a\in\R\cup\{\infty\}$, $f\in\mathcal{FO}_{[a,\infty]}(E)$ and set $f(z):=0$ for $\im(z)>0$. 
 Then $f\in\mathcal{O}^{exp}(\overline{\C}\setminus\overline{\R},E)$.
 \item [b)] Let $a\in\R\cup\{-\infty\}$, $f\in\mathcal{FO}_{[-\infty,a]}(E)$ and set $f(z):=0$ for $\im(z)<0$. 
 Then $f\in\mathcal{O}^{exp}(\overline{\C}\setminus\overline{\R},E)$.
\end{enumerate}
\end{prop}
\begin{proof}
 Let $K=[a,\infty]$ or $K=[-\infty,a]$ but $K\neq\overline{\R}.$ In both parts the extended $f$ is holomorphic 
 on $\C\setminus\R$ and 
 \[
 \vertiii{f}_{k,\alpha,\overline{\R}}\leq
 \begin{cases}
  e^{|a|k}|f|_{k,\alpha,K}&,\;a\in\R,\\
  e^{k^{2}}|f|_{k,\alpha,K}&,\;a=\pm\infty,
  \end{cases}
 \]
 for $k\in\N$ and $\alpha\in\mathfrak{A}$.
\end{proof}

\begin{lem}\label{lem:fourier_unbounded_interval}
Let $E$ be a sequentially complete $\C$-lcHs, $a\in\overline{\R}$ and $K:=[a,\infty]$ or $K:=[-\infty,a]$ 
but $K\neq\overline{\R}$. Then the map 
\begin{gather*}
 \mathcal{F}\colon bv_{K}(E) \to \mathcal{FO}_{K}(E),\\
 \mathcal{F}([F])(\zeta):=\langle \mathscr{H}_{K}([F]), e^{-i(\cdot)\zeta}\rangle
 =\int_{\gamma_{K}}{F(z)e^{-iz\zeta}\d z},
\end{gather*}
where $\gamma_{K}$ is the path along the boundary of $U_{\nicefrac{1}{c}}(K)$ with clockwise orientation, 
does not depend on the choice of $c>0$, is well-defined and continuous. 
Further, the equations
\begin{equation}\label{P.0.1.1}
 \mathscr{H}_{\overline{\R}}([\mathcal{F}([F])])
 =-\mathcal{F}_{\star}(\mathscr{H}_{[a,\infty]}([F]))\;\text{and}\;
  \mathscr{H}_{\overline{\R}}([\mathcal{F}([F])])
 =\mathcal{F}_{\star}(\mathscr{H}_{[-\infty,a]}([F]))
\end{equation}
hold on $\mathcal{P}_{\ast}(\overline{\R})$ where $\mathcal{F}([F])$ 
is considered as an element of $\mathcal{O}^{exp}(\overline{\C}\setminus\overline{\R},E)$ 
in these equations according to \prettyref{prop:upper_lower_extension}.
\end{lem}
\begin{proof}
Let $K:=[a,\infty]$, $F\in\mathcal{O}^{exp}(\overline{\C}\setminus[a,\infty],E)$ and $\alpha\in\mathfrak{A}$.
Let $a\in\R$ and define $a_{+}:=\max(0,a)$. For $k\in\N$ let $\zeta\in\C$ such that 
$\im(\zeta)\leq -\frac{1}{k}$. 
We choose $c:=\tfrac{1}{2k}$ and get $\tfrac{1}{2(k+1)}<c<2(k+1)$ as well as
\begin{equation}\label{P.0.0.0}
a_{+}\im(\zeta)\leq -a |\im(\zeta)|\quad\text{and}\quad
 \sup_{z\in\widetilde{\gamma}_{K}}e^{\re(z)\im(\zeta)}
=e^{-\inf_{z\in\widetilde{\gamma}_{K}}{\re(z)}|\im(\zeta)|}=e^{-(a-c)|\im(\zeta)|},
\end{equation}
since $\im(\zeta)<0$, where $\widetilde{\gamma}_{K}$ is the part of $\gamma_{K}$ as depicted in \prettyref{fig:path_1}. 
\begin{center}
\begin{minipage}{\linewidth}
\centering
\begin{tikzpicture}
\draw[dashed,blue,very thick]  (-1,1.5) arc (90:270:15mm);
\path  (-1,1.5) arc (90:145:15mm) node [above=5pt] (A) {\color{blue}$\widetilde{\gamma}_{K}$};
\draw[dashed,blue,very thick] (0,1.5) -- (-1,1.5);
\draw[dashed,blue,very thick] (0,-1.5) -- (-1,-1.5);
\draw[blue,very thick] (2,1.5) -- (4,1.5) node [above=0.1pt] {\color{blue}$\gamma_{K}$};
\draw[->,blue,very thick] (0,1.5) -- (2.2,1.5);
\draw[blue,very thick] (2,-1.5) -- (0,-1.5);
\draw[->,blue,very thick] (4,-1.5) -- (1.8,-1.5);
\draw(-1 cm,1pt)--(-1cm,-1pt)node[anchor=north] {$a$};
\draw(-2.5 cm,1pt)--(-2.5cm,-1pt)node[anchor=north east] {$a-c$};
\coordinate (0,0) node[below=1pt] {$a_{+}$};
\draw(1pt,1.5 cm)--(-1pt,1.5 cm)node[anchor=south east] {$c$};
\draw(1pt,-1.5 cm)--(-1pt,-1.5 cm)node[anchor=north east] {$-c$};
\draw[->] (-4,0) -- (4,0) node[right] {$\re(z)$} coordinate (x axis);
\draw[red,thick] (-1,0)--(4,0);
\draw[->] (0,-2) -- (0,2) node[above] {$\im(z)$} coordinate (y axis);
\end{tikzpicture}
\end{minipage}
\captionsetup{type=figure}
\caption{Path $\gamma_{K}$ for $K=[a,\infty]$, $a\in\R$}
\label{fig:path_1}
\end{center}
Furthermore, we have 
 \begin{equation}\label{P.0.0.1}
 \frac{1}{2(k+1)}+\im(\zeta)\leq \frac{1}{2(k+1)}-\frac{1}{k}=-\frac{k+2}{2k(k+1)}<0.
 \end{equation}
 Moreover, we obtain
 \begin{flalign}\label{eq:cont_fourier_unbounded_interval}
  &\hspace{0.35cm} p_{\alpha}(\mathcal{F}([F])(\zeta))\notag\\
  &=p_{\alpha}\bigl(\int_{\gamma_{K}}{F(z)e^{-iz\zeta}\d z}\bigr)\notag\\
  &\leq (\pi c+2(a_{+}-a))\sup_{z\in\widetilde{\gamma}_{K}}p_{\alpha}(F(z))e^{-\re(iz\zeta)}
  +\int_{a_{+}}^{\infty}{p_{\alpha}(F(t+ic))e^{-\re(i(t+ic)\zeta)}\d t}\notag\\
  &\phantom{\leq}+\int_{a_{+}}^{\infty}{p_{\alpha}(F(t-ic))e^{-\re(i(t-ic)\zeta)}\d t}\notag\\
  &\leq C_{0}\vertiii{F}_{2(k+1),\alpha,[a,\infty]}\sup_{z\in\widetilde{\gamma}}e^{\frac{1}{2(k+1)}|\re(z)|}
   \sup_{z\in\widetilde{\gamma}}e^{\re(z)\im(\zeta)+\im(z)\re(\zeta)}\notag\\
  &\phantom{\leq}+\vertiii{F}_{2(k+1),\alpha,[a,\infty]}
   \int_{a_{+}}^{\infty}{e^{\frac{1}{2(k+1)}|t|+t\im(\zeta)+c\re(\zeta)}\d t}\notag\\
  &\phantom{\leq}+\vertiii{F}_{2(k+1),\alpha,[a,\infty]}
   \int_{a_{+}}^{\infty}{e^{\frac{1}{2(k+1)}|t|+t\im(\zeta)-c\re(\zeta)}\d t}\notag\\
  &\;\;\mathclap{\underset{\eqref{P.0.0.0}}{\leq}}\;\;\vertiii{F}_{2(k+1),\alpha,[a,\infty]}C_{0}e^{\frac{1}{2(k+1)}
   \max{(|a-c|,a_{+})}}e^{(c-a)|\im(\zeta)|}e^{c|\re(\zeta)|}\notag\\
  &\phantom{\leq}+2\vertiii{F}_{2(k+1),\alpha,[a,\infty]}e^{c|\re(\zeta)|}
   \int_{a_{+}}^{\infty}{e^{\left(\frac{1}{2(k+1)}+\im(\zeta)\right)t}\d t}\notag\\
  &=\vertiii{F}_{2(k+1),\alpha,[a,\infty]}e^{c|\re(\zeta)|}\Bigl(C_{1}e^{(c-a)|\im(\zeta)|}
   +\frac{2}{\bigl|\frac{1}{2(k+1)}+\im(\zeta)\bigr|}e^{\left(\frac{1}{2(k+1)}+\im(\zeta)\right)a_{+}}\Bigr)\notag\\
  &\;\;\mathclap{\underset{\eqref{P.0.0.0}}{\leq}}\;\;\vertiii{F}_{2(k+1),\alpha,[a,\infty]}\Bigl(C_{1}
   +\frac{2e^{\frac{1}{2(k+1)}a_{+}}}{\bigl|\frac{1}{2(k+1)}+\im(\zeta)\bigr|}\Bigr)e^{c(|\re(\zeta)|+|\im(\zeta)|)
   -a|\im(\zeta)|}\notag\\
  &\;\;\mathclap{\underset{\eqref{P.0.0.1}}{\leq}}\;\;\vertiii{F}_{2(k+1),\alpha,[a,\infty]}\Bigl(C_{1}+
   \frac{4k(k+1)e^{\frac{1}{2(k+1)}a_{+}}}{k+2}\Bigr)e^{2c|\zeta|-a|\im(\zeta)|}\notag\\
  &=C_{2}\vertiii{F}_{2(k+1),\alpha,[a,\infty]}e^{\frac{1}{k}|\zeta|-a|\im(\zeta)|}
 \end{flalign}
and therefore
\[
|\mathcal{F}([F])|_{k,\alpha,[a,\infty]}\leq C_{2}\vertiii{[F]}_{2(k+1),\alpha,[a,\infty]}^{\wedge}.
\]
Now, we consider the case $a:=\infty$. We replace \prettyref{fig:path_1} by:
\begin{center}
\begin{minipage}{\linewidth}
\centering
\begin{tikzpicture}
\draw[dashed,blue,very thick]  (2,1.5) -- (2,0.7) node [left] (A) {\color{blue}$\widetilde{\gamma}_{K}$};
\draw[->,dashed,blue,very thick]  (2,-1.5) -- (2,0.7);
\draw[blue,very thick] (7,1.5) -- (4.5,1.5) node [above=0.1pt] {\color{blue}$\gamma_{K}$};
\draw[->,blue,very thick] (2,1.5) -- (4.7,1.5);
\draw[blue,very thick] (2,-1.5)-- (4.5,-1.5);
\draw[->,blue,very thick] (7,-1.5) -- (4.3,-1.5);
\node[anchor=north east] (A) at (2,0) {$\frac{1}{c}$};
\draw(1pt,1.5 cm)--(-1pt,1.5 cm)node[anchor=south east] {$c$};
\draw(1pt,-1.5 cm)--(-1pt,-1.5 cm)node[anchor=north east] {$-c$};
\draw[->] (-1,0) -- (7,0) node[right] {$\re(z)$} coordinate (x axis);
\draw[->] (0,-2) -- (0,2) node[above] {$\im(z)$} coordinate (y axis);
\end{tikzpicture}
\end{minipage}
\captionsetup{type=figure}
\caption{Path $\gamma_{K}$ for $K=\{\infty\}$}
\end{center}
If we choose $k$ like before, we get
\[
\sup_{z\in\widetilde{\gamma}_{K}}e^{\re(z)\im(\zeta)}=e^{-\frac{1}{c}|\im(\zeta)|}
\]
and 
\begin{align*}
 p_{\alpha}(\mathcal{F}([F])(\zeta))
&\leq\vertiii{F}_{2(k+1),\alpha,\{\infty\}}
 \left(\frac{2}{c}e^{\frac{1}{2(k+1)c}}+\frac{4k(k+1)e^{\frac{1}{2(k+1)c}}}{k+2}\right)
 e^{c|\re(\zeta)|-\frac{1}{c}|\im(\zeta)|}\\
&\leq C_{2}\vertiii{F}_{2(k+1),\alpha,\{\infty\}}e^{\frac{1}{k}|\zeta|-k|\im(\zeta)|}.
\end{align*}
Moreover, the definition of the Fourier transform does not depend on $c>0$ or the choice of the representative 
by virtue of Cauchy's integral theorem. Hence the Fourier transform is well-defined and continuous.

Let us turn to \eqref{P.0.1.1}. We only consider the case $K=[a,\infty]$ with $a\in\R$. 
By \prettyref{prop:upper_lower_extension} we regard $\mathcal{F}([F])$ as an element 
of $\mathcal{O}^{exp}(\overline{\C}\setminus\overline{\R},E)$. 
Let $n\in\N$ and $\varphi\in\mathcal{O}_{n}(U_{n}(\overline{\R}))$. We choose $\delta>0$ and $k\in\N$ such that 
$\tfrac{1}{k}\leq \delta< \tfrac{1}{n}$. 
Due to the estimate \eqref{eq:cont_fourier_unbounded_interval} we have for every $e'\in E'$
\begin{flalign*}
&\hspace{0.35cm}\int_{-\infty}^{\infty}\Bigl(\int_{\pi/2}^{3\pi/2}|(-1)^2cie^{ict}e'(F(a+ce^{it}))
  e^{-i(a+ce^{it})(s-i\delta)}\varphi(s-i\delta)|\d t\\
&\phantom{\int_{-\infty}^{\infty}\Bigl(}
 +\int_{a}^{\infty}|(-1)e'(F(t+ic))e^{-i(t+ic)(s-i\delta)}\varphi(s-i\delta)|\d t\\
&\phantom{\int_{-\infty}^{\infty}\Bigl(}
 +\int_{a}^{\infty}|(-1)^2e'(F(t-ic))e^{-i(t-ic)(s-i\delta)}\varphi(s-i\delta)|\d t\Bigr)\d s\\
&\leq C_{2}\vertiii{e'\circ F}_{2(k+1),[a,\infty]}
 \int_{-\infty}^{\infty}e^{\frac{1}{k}|s-i\delta|-a|\im(s-i\delta)|}|\varphi(s-i\delta)|\d s\\
&\leq 2C_{2}e^{\frac{\delta}{k}-a\delta}\vertiii{e'\circ F}_{2(k+1),[a,\infty]}\|\varphi\|_{n,\overline{\R}}
 \int_{0}^{\infty} e^{(\frac{1}{k}-\frac{1}{n})s}\d s\\
&= 2C_{2}e^{\frac{\delta}{k}-a\delta}\vertiii{e'\circ F}_{2(k+1),[a,\infty]}\|\varphi\|_{n,\overline{\R}}
 \frac{1}{\frac{1}{n}-\frac{1}{k}}<\infty.
\end{flalign*}
It follows from the Fubini-Tonelli theorem that we may change the order of integration and obtain
\begin{align*}
  \langle e', \mathscr{H}_{\overline{\R}}([\mathcal{F}([F])])(\varphi)\rangle 
&=\mathscr{H}_{\overline{\R}}([\mathcal{F}([e'\circ F])])(\varphi)\\
&=-\int_{\im(\zeta)=-\delta}\int_{\gamma_{[a,\infty]}}(e'\circ F)(z)e^{-iz\zeta}\varphi(\zeta)\d z\;\d\zeta\\
&=-\int_{\gamma_{[a,\infty]}}(e'\circ F)(z)\int_{\im(\zeta)=-\delta}\varphi(\zeta)e^{-iz\zeta}\d\zeta\;\d z\\
&=-\int_{\gamma_{[a,\infty]}}(e'\circ F)(z)\int_{\R}\varphi(\zeta)e^{-iz\zeta}\d\zeta\;\d z\\
&=-\mathscr{H}_{[a,\infty]}([e'\circ F])(\widehat{\varphi})\\
&= \langle e', -\mathcal{F}_{\star}(\mathscr{H}_{[a,\infty]}([F]))(\varphi)\rangle
\end{align*}
for all $e'\in E'$, where we used Cauchy's integral theorem in the fourth equation. 
Thus the Hahn-Banach theorem implies the first equation in \eqref{P.0.1.1}.
The proof in the case $K=[-\infty,a]$ is analogous.
\end{proof}

If we want to emphasize that $\mathcal{F}$ depends on $K$, we write $\mathcal{F}_{K}$ instead.

\begin{thm}\label{thm:fourier_unbounded_interval} 
Let $E$ be a sequentially complete $\C$-lcHs, $a\in\overline{\R}$ and $K:=[a,\infty]$ or $K:=[-\infty,a]$ 
but $K\neq\overline{\R}$. Then the Fourier transform $\mathcal{F}$ in \prettyref{lem:fourier_unbounded_interval} 
is a topological isomorphism.
\end{thm}
\begin{proof}
Let $F\in\mathcal{O}^{exp}(\overline{\C}\setminus K,E)$ and $\mathcal{F}([F])=0$.
By \prettyref{prop:upper_lower_extension} we consider $\mathcal{F}([F])$ as an element of 
$\mathcal{O}^{exp}(\overline{\C}\setminus\overline{\R},E)$.
By virtue of \eqref{P.0.1.1} we obtain
\[
0=\mathscr{H}_{\overline{\R}}([\mathcal{F}([F])])(\varphi)
 =\langle \mathcal{F}_{\star}(\mathscr{H}_{K}([F])),\varphi\rangle,\;\varphi\in\mathcal{P}_{\ast}(\overline{\R}),
\]
implying $[F]=0$ since $\mathcal{F}_{\star}$ and $\mathscr{H}_{K}$ are isomorphisms. 
Therefore the Fourier transform $\mathcal{F}$ is injective.

The idea how to prove the surjectivity of $\mathcal{F}$ can be found in the proof of 
\cite[Theorem 3.3.2, p.\ 485-487]{Kawai} and we adapt the notation that is used there. 
Let $f\in\mathcal{FO}_{K}(E)$ and again we consider $f$ as an element of
$\mathcal{O}^{exp}(\overline{\C}\setminus\overline{\R},E)$ by \prettyref{prop:upper_lower_extension}. 
We define 
$\nu :=\mathcal{F}_{\star}^{-1}(\mathscr{H}_{\overline{\R}}([f]))\in L(\mathcal{P}_{\ast}(\overline{\R}),E)$ 
and next we show the following:
\begin{equation}\label{P.0.0.1.1}
\forall\;\alpha\in\mathfrak{A},\,n\in\N\;\exists\; k\in\N,\,C>0\;\forall\;
\varphi\in\mathcal{O}_{n}(U_{n}(\overline{\R})):\;
 p_{\alpha}(\nu(\varphi))\leq C|f|_{k,\alpha,K}\|\varphi\|_{n,K}
\end{equation}
Since $\mathcal{P}_{\ast}(K)$ is dense in $\mathcal{P}_{\ast}(\overline{\R})$ 
by \cite[Theorem 2.2.1, p.\ 474]{Kawai} (see the correction of its proof in \cite[Remark, p.\ 247-248]{saburi1985}), 
this implies that there is an unique, well-defined 
extension $\widetilde{\nu}\in L(\mathcal{P}_{\ast}(K),E)$ of $\nu$ given by 
$\widetilde{\nu}(\varphi):=\lim_{m\in\mathcal{N}}\nu(\varphi_{m})$ 
where $(\varphi_{m})_{m\in\mathcal{N}}$ is a net in $\mathcal{P}_{\ast}(\overline{\R})$ 
converging to $\varphi\in\mathcal{P}_{\ast}(K)$ due to \cite[3.4.2 Theorem, p.\ 61]{Jarchow}.

We choose $-\tfrac{1}{n}<\widetilde{c}<\tfrac{1}{n}$ and have for $\varphi\in\mathcal{O}_{n}(U_{n}(\overline{\R}))$  
\[
 \mathcal{F}^{-1}\varphi(\zeta)=\frac{1}{2\pi}\int_{\im(z)=\widetilde{c}}{\varphi(z)e^{iz\zeta}\d z}.
\]
Due to Cauchy's integral theorem and the growth of $\varphi$ the definition of $\mathcal{F}^{-1}\varphi(\zeta)$ 
does not depend on the choice of $\widetilde{c}$. Let $K:=[a,\infty]$.
By virtue of \prettyref{prop:upper_lower_extension} we consider $f$ as an element of 
$\mathcal{O}^{exp}(\overline{\C}\setminus\overline{\R},E)$ and we have for $\delta >0$ 
by \prettyref{thm:fourier_isom_extended_reals}
\begin{align*}
  \nu(\varphi)
&=\mathcal{F}_{\star}^{-1}(\mathscr{H}_{\overline{\R}}([f]))(\varphi)
 =\mathscr{H}_{\overline{\R}}([f])(\mathcal{F}^{-1}\varphi)
 =\int_{\gamma_{\overline{\R}}}{f(\zeta)\mathcal{F}^{-1}\varphi(\zeta)\d\zeta}\\
&=-\frac{1}{2\pi}\int_{\im(\zeta)=-\delta}{f(\zeta)\int_{\im(z)=\widetilde{c}}{\varphi(z)e^{iz\zeta}\d z}\,\d\zeta}.
\end{align*}
First, we consider the case $a\in\R$. We choose $\delta,\delta'>0$ such that
\begin{equation}\label{P.0.0.2}
  (\delta ')^{2}+\delta^{2}<\frac{1}{n^{2}},\quad\text{implying}\quad 0<\delta,\delta'<\frac{1}{n},
\end{equation}
and define
\[
I_{+}:=\int_{0-i\delta}^{\infty-i\delta}{f\left(\zeta\right)\int_{\im(z)=\delta}{\varphi(z)e^{iz\zeta}\d z}\,\d\zeta}
\quad\text{and}\quad
I_{-}:=\int_{-\infty-i\delta}^{0-i\delta}{f\left(\zeta\right)\int_{\im(z)=-\delta}{\varphi(z)e^{iz\zeta}dz}\,\d\zeta}.
\]
Furthermore, we set 
\[
 J_{++}:=\int_{0-i\delta}^{\infty-i\delta}{f(\zeta)\int_{a-\delta'+i\delta}^{\infty+i\delta}
 {\varphi(z)e^{iz\zeta}dz}\:d\zeta}
  \quad\text{and}\quad
 J_{+-}:=\int_{0-i\delta}^{\infty-i\delta}{f(\zeta)\int_{-\infty+i\delta}^{a-\delta'+i\delta}
 {\varphi(z)e^{iz\zeta}dz}\,\d\zeta}
\]
as well as
\[
 J_{-+}:=\int_{-\infty-i\delta}^{0-i\delta}{f\left(\zeta\right)\int_{a-\delta'-i\delta}^{\infty-i\delta}
 {\varphi(z)e^{iz\zeta}\d z}\,\d\zeta}
  \;\;\,\text{and}\;\;\,
 J_{--}:=\int_{-\infty-i\delta}^{0-i\delta}{f(\zeta)\int_{-\infty-i\delta}^{a-\delta'-i\delta}
 {\varphi(z)e^{iz\zeta}\d z}\,\d\zeta}.
\]
Then we have $-2\pi\nu(\varphi)=I_{+}+I_{-}$ and $I_{+}=J_{++}+J_{+-}$ as well as $I_{-}=J_{-+}+J_{--}$. 
Now, we set $a_{0}:=\max(0,a-\delta')$ and choose $k\in\N$ such that
\begin{equation}\label{P.0.0.3}
\frac{1}{k}<\delta.
\end{equation}
We obtain the following estimates for $\alpha\in\mathfrak{A}$
\begin{flalign}\label{P.0.0.4}
&\hspace{0.35cm}p_{\alpha}(J_{++})\notag\\
&\leq\int_{0}^{\infty}{p_{\alpha}(f(t-i\delta))
\int_{a-\delta'}^{\infty}{|\varphi(s+i\delta)|e^{\re(i(s+i\delta)(t-i\delta))}\d s}\,\d t}\notag\\
&\;\;\mathclap{\underset{\eqref{P.0.0.2},\eqref{P.0.0.3}}{\leq}}\quad\int_{0}^{\infty}{|f|_{k,\alpha,K}
 e^{\frac{1}{k}|t-i\delta|-a\delta}
 \int_{a-\delta'}^{\infty}{\|\varphi\|_{n,K}e^{-\frac{1}{n}|s|}e^{s\delta-t\delta}\d s}\,\d t}\notag\\
&\leq e^{(\frac{1}{k}-a)\delta} |f|_{k,\alpha,K}\|\varphi\|_{n,K}
 \int_{0}^{\infty}{e^{(\frac{1}{k}-\delta)t}\d t}
 \int_{a-\delta'}^{\infty}{e^{-\frac{1}{n}|s|+\delta s}\d s}\notag\\
&\;\;\mathclap{\underset{\eqref{P.0.0.3}}{=}} \;\; e^{(\frac{1}{k}-a)\delta}|f|_{k,\alpha,K}\|\varphi\|_{n,K}
 \frac{1}{\delta-\frac{1}{k}}\Bigl(\int_{a-\delta'}^{a_{0}}{e^{(\frac{1}{n}+\delta)s}\d s}
 +\int_{a_{0}}^{\infty}{e^{(\delta-\frac{1}{n})s}\d s}\Bigr)\notag\\
&\;\;\mathclap{\underset{\eqref{P.0.0.2}}{=}} \;\; \frac{1}{\delta-\frac{1}{k}}
 \Bigl(\frac{1}{\frac{1}{n}+\delta} \bigl(e^{(\frac{1}{n}+\delta)a_{0}}-e^{(\frac{1}{n}+\delta)(a-\delta')}\bigr)
 +\frac{1}{\frac{1}{n}-\delta}e^{(\delta-\frac{1}{n})a_{0}}\Bigr)e^{(\frac{1}{k}-a)\delta}
 |f|_{k,\alpha,K}\|\varphi\|_{n,K}
\end{flalign}
and analogously
\begin{flalign}\label{P.0.0.5}
&\hspace{0.35cm}p_{\alpha}(J_{-+})\notag\\
&\leq\int_{-\infty}^{0}{ |f|_{k,\alpha,K}e^{\frac{1}{k}|t-i\delta|-a\delta}
 \int_{a-\delta'}^{\infty}{\|\varphi\|_{n,K}e^{-\frac{1}{n}|s|}e^{s\delta+t\delta}\d s}\;\d t}\notag\\
&\;\;\mathclap{\underset{\eqref{P.0.0.2},\eqref{P.0.0.3}}{\leq}}\quad e^{(\frac{1}{k}-a)\delta}
 |f|_{k,\alpha,K}\|\varphi\|_{n,K}
 \int_{-\infty}^{0}{ e^{\frac{1}{k}|t|+\delta t}\d t}
 \int_{a-\delta'}^{\infty}{e^{-\frac{1}{n}|s|+\delta s}\d s}\notag\\
&=e^{(\frac{1}{k}-a)\delta}|f|_{k,\alpha,K}\|\varphi\|_{n,K}
 \int_{-\infty}^{0}{ e^{(\delta-\frac{1}{k}) t}\d t}
 \int_{a-\delta'}^{\infty}{e^{-\frac{1}{n}|s|+\delta s}\d s}\notag\\
&\;\;\mathclap{\underset{\eqref{P.0.0.2},\eqref{P.0.0.3}}{=}}\quad \frac{1}{\delta-\frac{1}{k}}
 \Bigl(\frac{1}{\frac{1}{n}+\delta}\bigl(e^{(\frac{1}{n}+\delta)a_{0}}-e^{(\frac{1}{n}+\delta)(a-\delta')}\bigr)
 +\frac{1}{\frac{1}{n}-\delta}e^{(\delta-\frac{1}{n})a_{0}}\Bigr)e^{(\frac{1}{k}-a)\delta} 
 |f|_{k,\alpha,K}\|\varphi\|_{n,K}.
\end{flalign}
Let $\eta >0$ and set
\[
 J(-\eta;I):=\int_{0-i\eta}^{\infty-i\eta}{f(\zeta)J_{I}(\zeta)\d\zeta}
 \quad\text{and}\quad
 J(-\eta;II):=\int_{-\infty-i\eta}^{0-i\eta}{f(\zeta)J_{II}(\zeta)\d\zeta}
\]
where
\[
 J_{I}(\zeta):=\int_{-\infty+i\delta}^{a-\delta'+i\delta}{\varphi(z)e^{iz\zeta}\d z}
 \quad\text{and}\quad
 J_{II}(\zeta):=\int_{-\infty-i\delta}^{a-\delta'-i\delta}{\varphi(z)e^{iz\zeta}\d z}
\]
as well as
\[
 J(\downarrow ;I):=\int_{-i\delta}^{-i\eta}{f(\zeta)J_{I}(\zeta)\d\zeta}
\quad\text{and}\quad
 J(\uparrow ; II):=\int_{-i\eta}^{-i\delta}{f(\zeta)J_{II}(\zeta)\d\zeta}.
\]
We remark that $J_{+-}=J(-\delta;I)$ and $J_{--}=J(-\delta;II)$.
\begin{center}
\begin{minipage}{\linewidth}
\centering
\begin{tikzpicture}
\draw[blue,very thick] (1.8,-1) -- (4,-1) node [above=0.1pt] {\color{blue}$(-\delta;I)$};
\draw[->,blue,very thick] (0,-1) -- (2,-1);
\draw[->,green!50!black,very thick] (0,-2) -- (2,-2);
\draw[green!50!black,very thick] (1.8,-2) -- (4,-2) node [below=0.1pt] {\color{green!50!black}$(-\eta;I)$};
\draw[->,dashed,very thick] (3,-2) -- (3,-1.5) node [right] {$\gamma_{0}$};
\draw[dashed,very thick] (3,-1.5) -- (3,-1);
\draw[->,cyan,very thick] (-4,-1)node [above=0.1pt] {\color{cyan}$(-\delta;II)$} -- (-2,-1) ;
\draw[cyan,very thick] (-2.2,-1) -- (0,-1);
\draw[->,teal!50!white, very thick] (-0.03,-2) -- (-0.03,-1.5) node [left] {\color{teal!50!white}$\uparrow$};
\draw[teal!50!white, very thick] (-0.03,-1.7) -- (-0.03,-1);
\draw[->,green!50!black, very thick] (0.03,-1) -- (0.03,-1.5) node [right] {\color{green!50!black}$\downarrow$};
\draw[green!50!black, very thick] (0.03,-1.3) -- (0.03,-2);
\draw[->,teal!50!white,very thick] (-4,-2)node [below=0.1pt] {\color{teal!50!white}$(-\eta;II)$} -- (-2,-2);
\draw[teal!50!white,very thick] (-2.2,-2) -- (0,-2) ;
\draw[->,dashed,very thick] (-3,-1) -- (-3,-1.5) node [left] {$\gamma_{1}$};
\draw[dashed,very thick] (-3,-1.5) -- (-3,-2);
\draw(-3 cm,1pt)--(-3cm,-1pt)node[anchor=north] {$-r$};
\draw(3cm,1pt)--(3cm,-1pt)node[anchor=north] {$r$};
\draw(1pt,-1cm)--(-1pt,-1cm)node[anchor=south east] {$-\delta$};
\draw(1pt,-2cm)--(-1pt,-2cm)node[anchor=north east] {$-\eta$};
\draw[->] (-4,0) -- (4,0) node[right] {$\re(\zeta)$} coordinate (x axis);
\draw[->] (0,-3) -- (0,1) node[above] {$\im(\zeta)$} coordinate (y axis);
\end{tikzpicture}
\end{minipage}
\captionsetup{type=figure}
\caption{Paths of integration in $\zeta$-plane for $K=[a,\infty]$ (cf.\ \cite[p.\ 487]{Kawai})}
\label{fig:paths_zeta_plane}
\end{center} 
\begin{center}
\begin{minipage}{\linewidth}
\centering
\begin{tikzpicture}
\def\mypath{
 (2,2) -- (-1,2) arc (90:270:20mm) -- (2,-2)}
\fill[fill=black!10,draw=black,very thick] \mypath;
\node (A) at (1.5,1.5) {$U_{n}(K)$};
\draw[->,dotted,very thick] (-1.5,1) -- (1,1);
\draw[dotted,very thick] (1,1) -- (2,1) node[anchor=north east]{$++$};
\draw[->,dotted,very thick] (-1.5,-1) -- (1,-1);
\draw[dotted,very thick] (1,-1) -- (2,-1) node[anchor=south east]{$-+$};
\draw[->,blue,very thick] (-4.5,1)node [above=0.1pt] {\color{blue}$I$} -- (-2.9,1);
\draw[blue,very thick] (-3,1) -- (-1.5,1);
\draw[->,green!50!black, very thick] (-1.5,-1) -- (-1.5,-0.5);
\draw[green!50!black, very thick] (-1.5,-0.5) -- (-1.5,1) node [anchor=north east] {\color{green!50!black}$III$};
\draw[->,teal!50!white,very thick] (-4.5,-1)node [below=0.1pt] {\color{teal!50!white}$II$} -- (-2.9,-1);
\draw[teal!50!white,very thick] (-3,-1) -- (-1.5,-1) ;
\draw[->,dashed,very thick] (-4,1) -- (-4,0.5) node [left] {$\gamma_{2}$};
\draw[dashed,very thick] (-4,0.5) -- (-4,-1);
\draw(-4 cm,1pt)--(-4cm,-1pt)node[anchor=north east] {$-r$};
\draw(-1.5cm,1pt)--(-1.5cm,-1pt)node[anchor=north east] {$a-\delta'$};
\draw(1pt,-1cm)--(-1pt,-1cm)node[anchor=north east] {$-\delta$};
\draw(1pt,1cm)--(-1pt,1cm)node[anchor=south east] {$\delta$};
\draw(-1 cm,1pt)--(-1cm,-1pt)node[anchor=north] {$a$};
\draw(1pt,2 cm)--(-1pt,2 cm)node[anchor=south east] {$\nicefrac{1}{n}$};
\draw(1pt,-2 cm)--(-1pt,-2 cm)node[anchor=north east] {$-\nicefrac{1}{n}$};
\draw[->] (-4.5,0) -- (2,0) node[right] {$\re(z)$} coordinate (x axis);
\draw[red,thick] (-1,0)--(2,0);
\draw[->] (0,-2.5) -- (0,2.5) node[above] {$\im(z)$} coordinate (y axis);
\end{tikzpicture}
\end{minipage}
\captionsetup{type=figure}
\caption{Paths of integration in $z$-plane for $K=[a,\infty]$, $a\in\R$ (cf.\ \cite[p.\ 487]{Kawai})}
\label{fig:paths_z_plane}
\end{center}
Next, we show that: 
\begin{enumerate}
 \item [(i)] $\operatorname{lim}_{\eta\to\infty}J\left(-\eta;I\right)
 =\operatorname{lim}_{\eta\to\infty}J\left(-\eta;II\right)=0$
 \item [(ii)] $J\left(-\eta;I\right)+J\left(\downarrow;I\right)=J_{+-}$ and 
 $J\left(-\eta;II\right)+J\left( \uparrow ;I\right)=J_{-+}$ for $\eta>\delta$
\end{enumerate}
(i) We choose $k\in\N$ such that 
\begin{equation}\label{P.0.0.6}
 \frac{1}{k}<\delta'.
\end{equation}
Further, we observe that for $b_{0}\geq 0$
\begin{equation}\label{P.0.0.7}
 -b_{0}|x|\leq b_{0}x,\;x\in\R.
\end{equation}
Then we get the following estimates for $\eta>\delta>\tfrac{1}{k}$
\begin{align*}
 p_{\alpha}(J(-\eta; I))
&\;\;\;\mathclap{\underset{\eqref{P.0.0.2},\eqref{P.0.0.3}}{\leq}}\quad 
 \int_{0}^{\infty}{|f|_{k,\alpha,K}e^{\frac{1}{k}|t-i\eta|-a\eta}
 \int_{-\infty}^{a-\delta'}{\|\varphi\|_{n,\overline{\R}}e^{-\frac{1}{n}|s|}e^{s\eta-t\delta}\d s}\,\d t}\\
&\;\;\;\mathclap{\underset{\eqref{P.0.0.7}}{\leq}} \;\; e^{(\frac{1}{k}-a)\eta}
 |f|_{k,\alpha,K}\|\varphi\|_{n,\overline{\R}}
 \int_{0}^{\infty}{e^{(\frac{1}{k}-\delta)t}\d t}
 \int_{-\infty}^{a-\delta'}{e^{(\frac{1}{n}+\eta) s}\d s}\\
&\;\;\;\mathclap{\underset{\eqref{P.0.0.3}}{=}} \;\;e^{\frac{1}{n}(a-\delta')}
 |f|_{k,\alpha,K}\|\varphi\|_{n,\overline{\R}}\frac{1}{(\delta-\frac{1}{k})(\frac{1}{n}+\eta)}
 e^{(\frac{1}{k}-\delta')\eta}\\
&\;\;\;\mathclap{\underset{\eqref{P.0.0.6}}{\to}}\quad 0,\;\eta\to \infty,
\end{align*}
and
\begin{align*}
 p_{\alpha}(J(-\eta; II))
&\;\;\mathclap{\underset{\eqref{P.0.0.2},\eqref{P.0.0.3}}{\leq}}\quad\int_{-\infty}^{0}{|f|_{k,\alpha,K}
 e^{\frac{1}{k}|t-i\eta|-a\eta}\int_{-\infty}^{a-\delta'}{\|\varphi\|_{n,\overline{\R}}
 e^{-\frac{1}{n}|s|}e^{s\eta+t\delta}\d s}\,\d t}\\
&\;\;\mathclap{\underset{\eqref{P.0.0.7}}{\leq}}\;\;e^{(\frac{1}{k}-a)\eta}|f|_{k,\alpha,K}\|\varphi\|_{n,\overline{\R}}
 \int_{-\infty}^{0}{e^{(\delta-\frac{1}{k})t}\d t}
 \int_{-\infty}^{a-\delta'}{e^{(\frac{1}{n}+\eta) s}\d s}\\
&\;\;\mathclap{\underset{\eqref{P.0.0.3}}{=}}\;\; e^{\frac{1}{n}(a-\delta')}
 |f|_{k,\alpha,K}\|\varphi\|_{n,\overline{\R}}\frac{1}{(\delta-\frac{1}{k})(\frac{1}{n}+\eta)}
 e^{(\frac{1}{k}-\delta')\eta}\\
&\;\;\mathclap{\underset{\eqref{P.0.0.6}}{\to}}\quad 0,\;\eta\to \infty,
\end{align*}
for every $\alpha\in\mathfrak{A}$, implying statement (i).
 
(ii) The function defined by $G_{0}(\zeta):=f(\zeta)J_{I}(\zeta)$ is holomorphic on the lower halfplane and 
with the path $\gamma_{0}$ from \prettyref{fig:paths_zeta_plane} we have
\begin{align*}
 p_{\alpha}\bigl(\int_{\gamma_{0}}{G_{0}(\zeta)\d\zeta}\bigr)
&=p_{\alpha}\bigl(-\int_{\delta}^{\eta}{G_{0}(r-it)\cdot(-i)\d t}\bigr)\\
&\;\;\mathclap{\underset{\eqref{P.0.0.2},\eqref{P.0.0.3}}{\leq}}\quad e^{\frac{1}{k}r}
 |f|_{k,\alpha,K}\|\varphi\|_{n,\overline{\R}}
 \int_{\delta}^{\eta}{e^{(\frac{1}{k}-a)t}
 \int_{-\infty}^{a-\delta'}{e^{-\frac{1}{n}|s|}e^{st-\delta r}\d s}\,\d t}\\
&\;\;\mathclap{\underset{\eqref{P.0.0.7}}{\leq}}\;\;e^{(\frac{1}{k}-\delta) r}
 |f|_{k,\alpha,K}\|\varphi\|_{n,\overline{\R}}
 \int_{\delta}^{\eta}{e^{(\frac{1}{k}-a)t}\frac{1}{\frac{1}{n}+t}e^{(\frac{1}{n}+t)(a-\delta')}\d t}\\
&\leq (\eta-\delta)\frac{1}{\frac{1}{n}+\delta}e^{\frac{1}{k}\eta+\frac{1}{n}a}
 |f|_{k,\alpha,K}\|\varphi\|_{n,\overline{\R}}e^{(\frac{1}{k}-\delta)r}\\
&\;\;\mathclap{\underset{\eqref{P.0.0.3}}{\to}}\quad 0,\;r\to \infty,
\end{align*}
for all $\alpha\in\mathfrak{A}$. 
The function defined by $G_{1}(\zeta):=f(\zeta)J_{II}(\zeta)$ is also holomorphic on the lower halfplane 
and like above we have 
\begin{align*}
 p_{\alpha}\bigl(\int_{\gamma_{1}}{G_{1}(\zeta)\d\zeta}\bigr)
&=p_{\alpha}\bigl(\int_{\delta}^{\eta}{G_{1}(-r-it)\cdot(-i)\d t}\bigr)\\
&\leq (\eta-\delta)\frac{1}{\frac{1}{n}+\delta}e^{\frac{1}{k}\eta+\frac{1}{n}a}
 |f|_{k,\alpha,K}\|\varphi\|_{n,\overline{\R}}e^{(\frac{1}{k}-\delta )r}\\
&\;\;\mathclap{\underset{\eqref{P.0.0.3}}{\to}}\quad 0,\;r\to \infty.
\end{align*}
Hence Cauchy's integral theorem proves claim (ii).

Further, we have
\begin{flalign*}
&\hspace{0.35cm} J_{+-}+J_{--}-(J(-\eta;I)+J(-\eta;I))\\
&\;\;\mathclap{\underset{(ii)}{=}}\;\;J(\downarrow;I)+J(\uparrow;I)
 =\int_{-i\delta}^{-i\eta}{G_{0}(\zeta)\d\zeta}+\int_{-i\eta}^{-i\delta}{G_{1}(\zeta)\d\zeta}
 =\int_{-i\delta}^{-i\eta}{G_{0}(\zeta)-G_{1}(\zeta)\d\zeta}\\
& =\int_{-i\delta}^{-i\eta}{f(\zeta)(J_{I}(\zeta)-J_{II}(\zeta))\d\zeta}.
\end{flalign*}
Now, we claim that
\begin{equation}\label{P.0.0.7a}
 J_{I}(\zeta)-J_{II}(\zeta)
=\int_{a-\delta'-i\delta}^{a-\delta'+i\delta}{\varphi(z)e^{iz\zeta}\d z}=:J_{III}(\zeta),\;\zeta\in[-i\eta,-i\delta].
\end{equation}
Since for $-r<a-\delta'$ and the path $\gamma_{2}$ from \prettyref{fig:paths_z_plane} it holds that
\begin{align*}
 \bigl|\int_{\gamma_{2}}{\varphi(z)e^{iz\zeta}\d z}\bigr|
&=\bigl|-\int_{-\delta}^{\delta}{\varphi(-r+it)e^{i(-r+it)\zeta}i\d z}\bigr|\\
&\;\;\mathclap{\underset{\eqref{P.0.0.3}}{\leq}} \;\; \int_{-\delta}^{\delta}{\|\varphi\|_{n,\overline{\R}}
 e^{-\frac{1}{n}|-r|}e^{r\im(\zeta)-t\re(\zeta)}\d t}\\
&\;\;\mathclap{\underset{\re\left(\zeta\right)=0}{=}} \quad  e^{(\im(\zeta)-\frac{1}{n})r}
 \|\varphi\|_{n,\overline{\R}}\int_{-\delta}^{\delta}{1\d t}\\
&\;\;\mathclap{\underset{\im\left(\zeta\right)<0}{=}} \quad 2\delta\|\varphi\|_{n,\overline{\R}}
 e^{-(|\im(\zeta)|+\frac{1}{n})r}\\
&\to 0,\;r\to\infty,
\end{align*}
the claim follows again by Cauchy's theorem. Thus we get
\[
 J(\downarrow;I)+J(\uparrow;II)
 =\int_{-i\delta}^{-i\eta}{f(\zeta)J_{III}(\zeta)\d\zeta}
 =:J(-\eta;III),
\]
yielding
\[
 J_{+-}+J_{--}-(J(-\eta;I)+J(-\eta;II))\underset{(ii)}{=}J(\downarrow;I)+J(\uparrow;II)=J(-\eta;III)
\]
and
\begin{equation}\label{P.0.0.8}
 J_{+-}+J_{--}\underset{(i)}{=}\lim_{\eta\to\infty}(J_{+-}+J_{--}-(J(-\eta;I)+J(-\eta;II)))
=\lim_{\eta\to\infty}J(-\eta;III).
\end{equation}
Now, we estimate the right-hand side of this equation and obtain for all $\alpha\in\mathfrak{A}$
\begin{align}\label{P.0.0.9}
 p_{\alpha}(J(-\eta;III))
&=p_{\alpha}\bigl(\int_{\delta}^{\eta}{f(-it)\cdot(-i)\int_{-\delta}^{\delta}{\varphi(a-\delta'+is)
 e^{i(a-\delta'+is)(-it)}i\d s}\,\d t}\bigr)\notag\\
&\;\;\mathclap{\underset{\eqref{P.0.0.2},\eqref{P.0.0.3}}{\leq}}\quad \; \int_{\delta}^{\eta}{
 |f|_{k,\alpha,K}e^{\frac{1}{k}t-at}
 \int_{-\delta}^{\delta}{\|\varphi\|_{n,K}e^{-\frac{1}{n}|a-\delta'|}e^{(a-\delta')t}\d s}\,\d t}\notag\\
&\leq 2\delta|f|_{k,\alpha,K}\|\varphi\|_{n,K} \int_{\delta}^{\eta}{e^{(\frac{1}{k}-\delta') t}\d t}\notag\\
&= 2\delta|f|_{k,\alpha,K}\|\varphi\|_{n,K}\frac{1}{\frac{1}{k}-\delta'}
 (e^{(\frac{1}{k}-\delta')\eta}-e^{(\frac{1}{k}-\delta')\delta})\notag\\
&\;\;\mathclap{\underset{\eqref{P.0.0.6}}{\to}}\quad\frac{2\delta}{\delta'-\frac{1}{k}}
 e^{(\frac{1}{k}-\delta')\delta}|f|_{k,\alpha,K}\|\varphi\|_{n,K},\;\eta\to \infty.
\end{align}
Therefore we get for all $\alpha\in\mathfrak{A}$
\begin{align*}
 2\pi p_{\alpha}(\nu(\varphi))
&=p_{\alpha}(J_{++}+J_{+-}+J_{-+}+J_{--})
 \leq p_{\alpha}(J_{++})+p_{\alpha}(J_{-+})+p_{\alpha}(J_{+-}+J_{--})\\
&\;\;\mathclap{\underset{\eqref{P.0.0.8}}{=}}\;\;\; p_{\alpha}(J_{++})+p_{\alpha}(J_{-+})
 +p_{\alpha}(\lim_{\eta\to\infty}J(-\eta;III))\\
&\;\;\mathclap{\underset{\eqref{P.0.0.4},\eqref{P.0.0.5},\eqref{P.0.0.9}}{\leq}}\qquad\;
 \frac{2}{\delta-\frac{1}{k}}\Bigl(\frac{1}{\frac{1}{n}+\delta}\bigl(e^{(\frac{1}{n}+\delta)a_{0}}
 -e^{(\frac{1}{n}+\delta)(a-\delta')}\bigr)
 +\frac{1}{\frac{1}{n}-\delta}e^{(\delta-\frac{1}{n})a_{0}}\Bigr)e^{(\frac{1}{k}-a)\delta}\\
&\phantom{\leq}\cdot |f|_{k,\alpha,K}\|\varphi\|_{n,K} 
 +\frac{2\delta}{\delta'-\frac{1}{k}}e^{(\frac{1}{k}-\delta')\delta}|f|_{k,\alpha,K}\|\varphi\|_{n,K}.
\end{align*}
Now, we consider the case $a=\infty$. The proof is quite similar if we replace \prettyref{fig:paths_z_plane} by
\begin{center}
\begin{minipage}{\linewidth}
\centering
\begin{tikzpicture}
\def\mypath{
 (3.5,2) -- (1,2) -- (1,-2) -- (3.5,-2)}
\fill[fill=black!10,draw=black,very thick] \mypath;
\node (A) at (3,1.5) {$U_{n}(K)$};
\draw[->,dotted,very thick] (2,1) -- (2.75,1);
\draw[dotted,very thick] (2.75,1) -- (3.5,1) node[anchor=north east]{$++$};
\draw[->,dotted,very thick] (2,-1) -- (2.75,-1);
\draw[dotted,very thick] (2.75,-1) -- (3.5,-1) node[anchor=south east]{$-+$};
\draw[->,blue,very thick] (-4.5,1)node [above=0.1pt] {\color{blue}$I$} -- (-1,1);
\draw[blue,very thick] (-1.1,1) -- (2,1);
\draw[->,green!50!black, very thick] (2,-1) -- (2,-0.5);
\draw[green!50!black, very thick] (2,-0.5) -- (2,1) node [anchor=north east] {\color{green!50!black}$III$};
\draw[->,teal!50!white,very thick] (-4.5,-1)node [below=0.1pt] {\color{teal!50!white}$II$} -- (-1,-1);
\draw[teal!50!white,very thick] (-1.1,-1) -- (2,-1) ;
\draw[->,dashed,very thick] (-4,1) -- (-4,0.5) node [left] {$\gamma_{2}$};
\draw[dashed,very thick] (-4,0.5) -- (-4,-1);
\draw(-4 cm,1pt)--(-4cm,-1pt)node[anchor=north east] {$-r$};
\draw(1cm,1pt)--(1cm,-1pt)node[anchor=north west] {$n$};
\draw(2cm,1pt)--(2cm,-1pt)node[anchor=north west] {$k-\delta'$};
\draw(1pt,-1cm)--(-1pt,-1cm)node[anchor= north east] {$-\delta$};
\draw(1pt,1cm)--(-1pt,1cm)node[anchor=south east] {$\delta$};
\draw(1pt,2 cm)--(-1pt,2 cm)node[anchor=south east] {$\nicefrac{1}{n}$};
\draw(1pt,-2 cm)--(-1pt,-2 cm)node[anchor=north east] {$-\nicefrac{1}{n}$};
\draw[->] (-4.5,0) -- (3.5,0) node[right] {$\re(z)$} coordinate (x axis);
\draw[->] (0,-2.5) -- (0,2.5) node[above] {$\im(z)$} coordinate (y axis);
\end{tikzpicture}
\end{minipage}
\captionsetup{type=figure}
\caption{Paths of integration in $z$-plane for $K=\{\infty\}$}
\end{center}
and thus replace in the definition of the $J$-integrals $a-\delta'$ by $k-\delta'$. 
We obtain for the choice $\tfrac{1}{k}<\delta,\delta'<\tfrac{1}{n}$ and $\alpha\in\mathfrak{A}$ the estimate
\[
 p_{\alpha}(J_{++})
\leq \frac{1}{(\delta-\frac{1}{k})(\frac{1}{n}-\delta)}e^{(\delta-\frac{1}{n})(k-\delta')}e^{(\frac{1}{k}-k)\delta}
 |f|_{k,\alpha,K}\|\varphi\|_{n,K}
\]
and the same for $p_{\alpha}(J_{-+})$. 
Furthermore, for $\eta>\delta$ and $\alpha\in\mathfrak{A}$
\begin{align*}
  p_{\alpha}(J(-\eta;I))
&\leq e^{\frac{1}{n}(k-\delta')}|f|_{k,\alpha,K}\|\varphi\|_{n,\overline{\R}}
 \frac{1}{(\delta-\frac{1}{k})(\frac{1}{n}+\eta)}e^{(\frac{1}{k}-\delta')\eta}\\
&\to 0,\;\eta\to \infty,
\end{align*}
and the same for $ p_{\alpha}(J(-\eta; II))$. Moreover, we have for $\alpha\in\mathfrak{A}$
\begin{align*}
 p_{\alpha}\bigl(\int_{\gamma_{0}}{G_{0}(\zeta)\d\zeta}\bigr)
&\leq (\eta-\delta)\frac{1}{\frac{1}{n}+\delta}e^{\frac{1}{k}\eta+\frac{1}{n}k}
 |f|_{k,\alpha,K}\|\varphi\|_{n,\overline{\R}}e^{(\frac{1}{k}-\delta )r}\\
&\to 0,\;r\to \infty,
\end{align*}
and the same for $p_{\alpha}(\int_{\gamma_{1}}{G_{1}(\zeta)\d\zeta})$. 
The proof of $J_{I}(\zeta)-J_{II}(\zeta)=J_{III}(\zeta)$ for all $\zeta\in[-i\eta,-i\delta]$ needs no adjustment 
and we still have 
\[
 \lim_{\eta\to\infty} p_{\alpha}(J(-\eta;III))
\leq \frac{2\delta}{\delta'-\frac{1}{k}}e^{(\frac{1}{k}-\delta')\delta}|f|_{k,\alpha,K}\|\varphi\|_{n,K}.
\]
Altogether we get for $\alpha\in\mathfrak{A}$
\begin{align*}
 2\pi p_{\alpha}(\nu(\varphi))
&\leq\frac{2}{(\delta-\frac{1}{k})(\frac{1}{n}-\delta)}e^{(\delta-\frac{1}{n})(k-\delta')}e^{(\frac{1}{k}-k)\delta}
 |f|_{k,\alpha,K}\|\varphi\|_{n,K}\\
&\phantom{\leq}+\frac{2\delta}{\delta'-\frac{1}{k}}e^{(\frac{1}{k}-\delta')\delta}|f|_{k,\alpha,K}\|\varphi\|_{n,K}
\end{align*}
and thus \eqref{P.0.0.1.1} in both cases.

Now, we consider $\mathcal{F}(\mathscr{H}^{-1}_{K}(\widetilde{\nu}))$ as an element of 
$\mathcal{O}^{exp}(\overline{\C}\setminus\overline{\R},E)$ by \prettyref{prop:upper_lower_extension}.
We get for every $\varphi\in\mathcal{P}_{\ast}(\overline{\R})$
\begin{align*}
 -\mathscr{H}_{\overline{\R}}([\mathcal{F}(\mathscr{H}^{-1}_{K}(\widetilde{\nu}))])(\varphi)
&\underset{\eqref{P.0.1.1}}{=}
 \langle (\mathcal{F}_{\star}\circ \mathscr{H}_{K} \circ \mathscr{H}^{-1}_{K})(\widetilde{\nu}),\varphi\rangle
 =\langle\mathcal{F}_{\star}(\widetilde{\nu}),\varphi\rangle
 =\langle\widetilde{\nu},\mathcal{F}\varphi\rangle\\
&\:=\langle\nu,\mathcal{F}\varphi\rangle
 =\langle\mathcal{F}^{-1}_{\star}(\mathscr{H}_{\overline{\R}}([f])),\mathcal{F}\varphi\rangle
 =\mathscr{H}_{\overline{\R}}([f])(\varphi),
\end{align*}
implying $\mathcal{F}(-\mathscr{H}^{-1}_{K}(\widetilde{\nu}))-f\in\mathcal{O}^{exp}(\overline{\C},E)$ 
because $\mathscr{H}_{\overline{\R}}$ is an isomorphism. 
Since $\mathcal{F}(-\mathscr{H}^{-1}_{K}(\widetilde{\nu}))-f$ is in particular an 
entire function and 
\begin{equation}\label{eq:surjectivity_unbounded}
\mathcal{F}(-\mathscr{H}^{-1}_{K}(\widetilde{\nu}))-f=0
\end{equation} 
on the upper halfplane, we obtain $f=\mathcal{F}(-\mathscr{H}^{-1}_{K}(\widetilde{\nu}))$ by the identity theorem. 
Thus the Fourier transform $\mathcal{F}$ is surjective.

Let $\beta\in\mathfrak{A}$ and $q\in\N$. Due to the continuity of $\mathscr{H}^{-1}_{K}$ there are 
a bounded set $M\subset\mathcal{P}_{\ast}(K)$, $\alpha\in\mathfrak{A}$ and $C_{0}>0$ such that
\[
 \vertiii{\mathcal{F}^{-1}(f)}_{q,\beta,K}^{\wedge}
=\vertiii{\mathscr{H}^{-1}_{K}(\widetilde{\nu})}_{q,\beta,K}^{\wedge}
\leq C_{0}\sup_{\varphi\in M}p_{\alpha}(\widetilde{\nu}(\varphi)).
\]
Since $\mathcal{P}_{\ast}(K)$ is a DFS-space, there are $n\in\N$ and $\lambda>0$ such that $M\subset\lambda B_{n}$ 
where $B_{n}$ is the closed unit ball of $\mathcal{O}_{n}(U_{n}(K))$ 
by \cite[Proposition 25.19 (2), p.\ 303]{meisevogt1997}. Let $(\varphi_{m})_{m\in\mathcal{N}}$ be a net in 
$\mathcal{P}_{\ast}(\overline{\R})$ converging to $\varphi\in M$. Then there are $k\in\N$ and $C>0$ such that
\[
 p_{\alpha}(\widetilde{\nu}(\varphi))
=\lim_{m\in\mathcal{N}}p_{\alpha}(\nu(\varphi_{m}))
\underset{\eqref{P.0.0.1.1}}{\leq}C|f|_{k,\alpha,K}\lim_{m\in\mathcal{N}}\|\varphi_{m}\|_{n,K}
=C|f|_{k,\alpha,K}\|\varphi\|_{n,K},
\]
implying
\[
 \vertiii{\mathcal{F}^{-1}(f)}_{q,\beta,K}^{\wedge}
\leq C_{0}C|f|_{k,\alpha,K}\sup_{\varphi\in M}\|\varphi\|_{n,K}
\leq C_{0}C\lambda |f|_{k,\alpha,K}.
\]
Thus $\mathcal{F}^{-1}$ is continuous and by \prettyref{lem:fourier_unbounded_interval} $\mathcal{F}$ as well. 
The proof for the remaining case $K=[-\infty,a]$ is analogous.
\end{proof}

\prettyref{thm:fourier_unbounded_interval} improves \cite[Theorem 3.3.1, 3.3.2, p.\ 485]{Kawai} (in one variable) 
since the latter theorem only covers the case $a=0$ and $E=\C$ and, more importantly, only shows that 
$\mathcal{F}$ is a linear isomorphism. 
We kept track of the estimates to show that $\mathcal{F}$ is actually a topological isomorphism. 

In \prettyref{sect:Asymptotic_Fourier_transform} we treat the asymptotic Fourier transform 
on the space of hyperfunctions $\mathcal{B}(\R,E)$ which also needs a characterisation of the Fourier transform 
of Fourier hyperfunctions with suppport in the union of two disjoint extended half-lines.

\begin{defn} 
 Let $E$ be a $\C$-lcHs, $-\infty\leq a<b\leq\infty$ and set $f+g:=f$ on $\h$ resp.\ $f+g:=g$ on $-\h$ 
 for $f\in\mathcal{FO}_{[-\infty,a]}(E)$ and $g\in\mathcal{FO}_{[b,\infty]}(E)$. We define the space 
 \[
   \mathcal{FO}_{[-\infty,a]\cup [b,\infty]}(E)
 :=\mathcal{FO}_{[-\infty,a]}(E)\oplus\mathcal{FO}_{[b,\infty]}(E).
 \]
\end{defn}

\begin{thm}\label{thm:fourier_union_unbounded_intervals} 
Let $E$ be a sequentially complete $\C$-lcHs and $-\infty\leq a<b\leq\infty$. 
Then the Fourier transform 
\begin{gather*}
 \mathcal{F}\colon bv_{[-\infty,a]\cup [b,\infty]}(E) \to \mathcal{FO}_{[-\infty,a]\cup [b,\infty]}(E),\\
 \mathcal{F}([F])(\zeta):=\mathcal{F}_{[-\infty,a]}([F])(\zeta)+\mathcal{F}_{[b,\infty]}([F])(\zeta),
 \;\zeta\in\C\setminus\R,
\end{gather*}
with $0<c<\tfrac{b-a}{2}$ is a topological isomorphism and
\begin{equation}\label{eq:fourier_union_unbounded_intervals}
 \mathscr{H}_{\overline{\R}}([\mathcal{F}([F])])
=\mathcal{F}_{\star}(\mathscr{H}_{[-\infty,a]}([F]))-\mathcal{F}_{\star}(\mathscr{H}_{[b,\infty]}([F])).
\end{equation}
\end{thm}
\begin{proof}
$\mathcal{F}$ being well-defined and its continuity follow from adjusting the proof of
\prettyref{lem:fourier_unbounded_interval}. Namely, 
if $a\neq-\infty$ and $b\neq\infty$, we have to adjust $c$ and choose $c:=\tfrac{1}{2mk}$ with $m\in\N$, 
$m\geq \tfrac{1}{(b-a)k}$, which guarantees that 
$\gamma_{[b,\infty]}$ does not intersect $[-\infty,a]$ and $\gamma_{[-\infty,a]}$ does not intersect $[b,\infty]$. 
Then we have 
\[
\mathcal{F}_{[b,\infty]}([F])(\zeta)\leq 
C_{2}\vertiii{F}_{2m(k+1),\alpha,[-\infty,a]\cup [b,\infty]}e^{\frac{1}{k}|\zeta|-b|\im(\zeta)|},\;\im(\zeta)<0,
\]
and
\[
\mathcal{F}_{[-\infty,a]}([F])(\zeta)\leq 
C_{2}\vertiii{F}_{2m(k+1),\alpha,[-\infty,a]\cup [b,\infty]}e^{\frac{1}{k}|\zeta|+a|\im(\zeta)|},\;\im(\zeta)>0,
\]
for every $\alpha\in\mathfrak{A}$ even if $F\in\mathcal{O}^{exp}(\overline{\C}\setminus([-\infty,a]\cup [b,\infty]),E)$. 
If $a=-\infty$ or $b=\infty$, no adjustment in the choice of $c$ is needed in the proof 
of \prettyref{lem:fourier_unbounded_interval}. The proof of \prettyref{lem:fourier_unbounded_interval} shows that
the equations
\begin{align}\label{eq:compat_fourier_union}
 \mathscr{H}_{\overline{\R}}([\mathcal{F}_{[-\infty,a]}([F])])
 &=\mathcal{F}_{\star}(\mathscr{H}_{[-\infty,a]}([F])),\notag \\ 
 \mathscr{H}_{\overline{\R}}([\mathcal{F}_{[b,\infty]}([F])])
 &=-\mathcal{F}_{\star}(\mathscr{H}_{[b,\infty]}([F]))
\end{align}
even hold for $F\in\mathcal{O}^{exp}(\overline{\C}\setminus ([-\infty,a]\cup [b,\infty]),E)$. 
This implies that \eqref{eq:fourier_union_unbounded_intervals} is valid as 
$\mathcal{F}=\mathcal{F}_{[-\infty,a]}+\mathcal{F}_{[b,\infty]}$.

Let us turn to injectivity. Let $F\in\mathcal{O}^{exp}(\overline{\C}\setminus ([-\infty,a]\cup [b,\infty]),E)$ such 
that $\mathcal{F}([F])=0$. This implies that $\mathcal{F}_{[-\infty,a]}([F])=0$ and $\mathcal{F}_{[b,\infty]}([F])=0$ 
by definition of $\mathcal{F}$. It follows from the equations above that 
\[
0=\mathscr{H}_{\overline{\R}}(\mathcal{F}_{[-\infty,a]}([F]))(\varphi)
 =\langle \mathcal{F}_{\star}(\mathscr{H}_{[-\infty,a]}([F])),\varphi\rangle,
 \;\varphi\in\mathcal{P}_{\ast}(\overline{\R}),
\]
and 
\[
0=-\mathscr{H}_{\overline{\R}}(\mathcal{F}_{[b,\infty]}([F]))(\varphi)
 =\langle \mathcal{F}_{\star}(\mathscr{H}_{[b,\infty]}([F])),\varphi\rangle,
 \;\varphi\in\mathcal{P}_{\ast}(\overline{\R}). 
\]
Since $\mathcal{F}_{\star}$ is an isomorphism, we obtain 
\[
\langle \mathscr{H}_{[-\infty,a]}([F]),\varphi\rangle=0
\quad\text{and}\quad
\langle\mathscr{H}_{[b,\infty]}([F]),\varphi\rangle=0,\;\varphi\in\mathcal{P}_{\ast}(\overline{\R}). 
\]
Summing both equations, we have 
\[
0=\langle\mathscr{H}_{[-\infty,a]}([F]),\varphi\rangle+\langle\mathscr{H}_{[b,\infty]}([F]),\varphi\rangle
 =\langle\mathscr{H}_{[-\infty,a]\cup [b,\infty]}([F]),\varphi\rangle,\;\varphi\in\mathcal{P}_{\ast}(\overline{\R}).
\]
The map $\mathscr{H}_{[-\infty,a]\cup [b,\infty]}\colon bv_{[-\infty,a]\cup [b,\infty]}(E)\to 
L_{b}(\mathcal{P}_{\ast}([-\infty,a]\cup [b,\infty]),E)$ is a topological isomorphism by \prettyref{thm:duality} and 
$\mathcal{P}_{\ast}(\overline{\R})$ is dense in $\mathcal{P}_{\ast}([-\infty,a]\cup [b,\infty])$ 
by \cite[Theorem 2.2.1, p.\ 474]{Kawai}. Thus we obtain from the equation above 
that $F\in\mathcal{O}^{exp}(\overline{\C},E)$, which implies the injectivity of $\mathcal{F}$.

Next, we prove that $\mathcal{F}$ is surjective with continuous inverse. 
Let $f=f_{1}+f_{2}\in\mathcal{FO}_{[-\infty,a]}(E)\oplus\mathcal{FO}_{[b,\infty]}(E)$. Due to 
\prettyref{thm:fourier_unbounded_interval} we have 
$g_{1}:=\mathcal{F}_{[-\infty,a]}^{-1}(f_{1})\in bv_{[-\infty,a]}(E)$ and 
$g_{2}:=\mathcal{F}_{[b,\infty]}^{-1}(f_{2})\in bv_{[b,\infty]}(E)$. 
Then $g_{1}+g_{2}\in bv_{[-\infty,a]\cup [b,\infty]}(E)$ and 
\begin{align*}
 \mathcal{F}(g_{1}+g_{2})(\zeta)
&=\mathcal{F}_{[-\infty,a]}(g_{1}+g_{2})(\zeta)+\mathcal{F}_{[b,\infty]}(g_{1}+g_{2})(\zeta)\\
&=\langle \mathscr{H}_{[-\infty,a]}(g_{1}),e^{-i(\cdot)\zeta}\rangle 
 +\langle \mathscr{H}_{[-\infty,a]}(g_{2}),e^{-i(\cdot)\zeta}\rangle \\
&\phantom{=}+\langle \mathscr{H}_{[b,\infty]}(g_{2}),e^{-i(\cdot)\zeta}\rangle 
 +\langle \mathscr{H}_{[b,\infty]}(g_{2}),e^{-i(\cdot)\zeta}\rangle \\
&=\langle \mathscr{H}_{[-\infty,a]}(g_{1}),e^{-i(\cdot)\zeta}\rangle 
 +\langle \mathscr{H}_{[b,\infty]}(g_{2}),e^{-i(\cdot)\zeta}\rangle \\
&=\mathcal{F}_{[-\infty,a]}(g_{1})(\zeta)+\mathcal{F}_{[b,\infty]}(g_{2})(\zeta)
 =f_{1}(\zeta)+f_{2}(\zeta)
 =f(\zeta)
\end{align*}
for $\zeta\in\C\setminus\R$. 
Hence $\mathcal{F}^{-1}(f)=\mathcal{F}_{[-\infty,a]}^{-1}(f_{1})+\mathcal{F}_{[b,\infty]}^{-1}(f_{2})$, 
yielding the surjectivity of $\mathcal{F}$ and the continuity of $\mathcal{F}^{-1}$.
\end{proof}
\section{Fourier transform of Fourier hyperfunctions with real compact support}
\label{sect:Fourier_real_comp}
In this section we take a look at the Fourier transform $\mathcal{F}_{K}$ for a real compact set $K$. We want to 
characterise the range of the Fourier transform on $bv_{K}(E)$ for an lcHs $E$. 
The supporting function $H_{K}$ is defined for general real compact sets $K$ as well but it cannot distinguish 
a compact set $K_{1}\subset\R$ from a compact set $K_{2}\subset\R$ if their convex hulls coincide. 
Therefore we restrict our considerations to real compact convex sets, i.e.\ to intervals $K=[a,b]\subset\R$. 
Here a corresponding Paley-Wiener theorem \cite[Theorem 8.1.1, p.\ 368-369]{Kan} is already known 
for $E=\C$ but we augment it by a topological aspect, i.e.\ show that $\mathcal{F}_{K}$ is a topological isomorphism, 
and consider general locally complete $\C$-lcHs $E$. This is done by adapting the 
proof of \prettyref{thm:fourier_unbounded_interval}. First, we define the range space.

\begin{defn} 
Let $E$ be a $\C$-lcHs and $-\infty<a\leq b<\infty$. We define the space
\[
\mathcal{FO}_{[a,b]}(E):=\{f\in\mathcal{O}(\C,E)\;|\;\forall\;k\in\N,\,\alpha\in\mathfrak{A}:\;
|f|_{k,\alpha,[a,b]}<\infty\}
\]
where
\[
|f|_{k,\alpha,[a,b]}:=\sup_{z\in\C}p_{\alpha}(f(z))e^{-\frac{1}{k}|z|-H_{[a,b]}(\im(z))}.
\]
\end{defn}

\begin{lem}\label{lem:fourier_bounded_interval}
Let $E$ be a locally complete $\C$-lcHs, $-\infty<a\leq b<\infty$ and $K:=[a,b]$. Then the map
\begin{gather*}
 \mathcal{F}\colon bv_{K}(E)\to \mathcal{FO}_{K}(E),\\
 \mathcal{F}([F])(\zeta):=\langle\mathscr{H}_{K}([F]),e^{-i(\cdot)\zeta}\rangle=\int_{\gamma_{K}}{F(z)e^{-iz\zeta}\d z},
\end{gather*}
where $\gamma_{K}$ is the path along the boundary of $U_{\nicefrac{1}{c}}(K)$ with clockwise orientation, 
does not depend on the choice of $c>0$, is well-defined and continuous. 
Further, the equations
\begin{gather}\label{P.0.0.10a}
\begin{split}
 \mathcal{F}_{[a,b]}([F])(\zeta)=\mathcal{F}_{[a,\infty]}([F])(\zeta),\quad \im(\zeta)<0,\\
 \mathcal{F}_{[a,b]}([F])(\zeta)=\mathcal{F}_{[-\infty,b]}([F])(\zeta),\quad \im(\zeta)>0,
\end{split}
\end{gather}
hold for $F\in\mathcal{O}^{exp}(\overline{\C}\setminus K,E)$.
\end{lem}
\begin{proof}
Let $F\in\mathcal{O}^{exp}(\overline{\C}\setminus K,E)$, $k\in\N$, $\alpha\in\mathfrak{A}$ and $\zeta\in\C$. 
We choose $c:=\tfrac{1}{2k}$ and get $\tfrac{1}{2(k+1)}<c<2(k+1)$ as well as
\begin{align}\label{eq:support_func_estimate}
 \sup_{z\in\gamma_{K}}e^{\re(z)\im(\zeta)}
&=e^{\sup_{z\in\gamma_{K}}{\re(z)\im(\zeta)}}
 =e^{H_{[a-c,b+c]}(\im(\zeta))}=e^{H_{[a,b]}(\im(\zeta))+H_{[-c,c]}(\im(\zeta))}\notag\\
&=e^{H_{[a,b]}(\im(\zeta))+\max(\pm c\im(\zeta))}
 \leq e^{c|\im(\zeta)|}e^{H_{[a,b]}(\im(\zeta))}
\end{align}
by \prettyref{prop:supporting_function}.
\begin{center}
\begin{minipage}{\linewidth}
\centering
\begin{tikzpicture}
\draw[blue,very thick]  (-1,1.5) arc (90:270:15mm);
\draw[blue,very thick]  (2,-1.5) arc (270:450:15mm);
\draw[blue,very thick] (0.8,1.5) -- (2,1.5) node [above=0.1pt] {\color{blue}$\gamma_{K}$};
\draw[->,blue,very thick] (-1,1.5) -- (1,1.5);
\draw[blue,very thick] (1.2,-1.5) -- (-1,-1.5);
\draw[->,blue,very thick] (2,-1.5) -- (1,-1.5);
\draw(-1 cm,1pt)--(-1cm,-1pt)node[anchor=north] {$a$};
\draw(2 cm,1pt)--(2cm,-1pt)node[anchor=north] {$b$};
\draw(-2.5 cm,1pt)--(-2.5cm,-1pt)node[anchor=north east] {$a-c$};
\draw(3.5 cm,1pt)--(3.5cm,-1pt)node[anchor=north west] {$b+c$};
\draw(1pt,1.5 cm)--(-1pt,1.5 cm)node[anchor=south east] {$c$};
\draw(1pt,-1.5 cm)--(-1pt,-1.5 cm)node[anchor=north east] {$-c$};
\draw[->] (-4,0) -- (4,0) node[right] {$\re(z)$} coordinate (x axis);
\draw[red,thick] (-1,0)--(2,0);
\draw[->] (0,-2) -- (0,2) node[above] {$\im(z)$} coordinate (y axis);
\end{tikzpicture}
\end{minipage}
\captionsetup{type=figure}
\caption{Path $\gamma_{K}$ for $K=[a,b]$}
\end{center}
We obtain
\begin{align*}
 p_{\alpha}(\mathcal{F}([F])(\zeta))
&=p_{\alpha}\bigl(\int_{\gamma_{K}}{F(z)e^{-iz\zeta}\d z}\bigr)
 \leq 2(\pi c+b-a)\sup_{z\in\gamma_{K}}p_{\alpha}(F(z))e^{-\re(iz\zeta)}\\
&\leq C_{0}\vertiii{F}_{2(k+1),\alpha,K}\sup_{z\in\gamma_{K}}e^{\frac{1}{2(k+1)}|\re(z)|}
 \sup_{z\in\gamma_{K}}e^{\re(z)\im(\zeta)+\im(z)\re(\zeta)}\\
&\leq\vertiii{F}_{2(k+1),\alpha,K}C_{0}e^{\frac{1}{2(k+1)}\max{(|a-c|,|b+c|)}}
 e^{c|\re(\zeta)|}\sup_{z\in\gamma_{K}}e^{\re(z)\im(\zeta)}\\
&\underset{\mathclap{\eqref{eq:support_func_estimate}}}{\leq} C_{1}\vertiii{F}_{2(k+1),\alpha,K}e^{2c|\zeta|}e^{H_{K}(\im(\zeta))}
\end{align*}
and therefore
\[
|\mathcal{F}([F])|_{k,\alpha,K}\leq C_{2}\vertiii{[F]}_{2(k+1),\alpha,K}^{\wedge}.
\]
Moreover, the definition of the Fourier transform does not depend on $c>0$ or the choice of the representative 
by virtue of Cauchy's integral theorem. Hence the Fourier transform is well-defined and continuous.

Due to Cauchy's integral theorem we can deform the path of integration 
for $F\in \mathcal{O}^{exp}(\overline{\C}\setminus [a,b],E)$ and obtain the remaining equations \eqref{P.0.0.10a}.
\end{proof}

\begin{thm}\label{thm:fourier_bounded_interval}
Let $E$ be a locally complete $\C$-lcHs, $-\infty<a\leq b<\infty$ and $K:=[a,b]$. 
Then the Fourier transform $\mathcal{F}$ in \prettyref{lem:fourier_bounded_interval}
is a topological isomorphism.
\end{thm}
\begin{proof}
Let $F\in\mathcal{O}^{exp}(\overline{\C}\setminus [a,b],E)$ and $\mathcal{F}_{[a,b]}([F])=0$.
By equation \eqref{P.0.0.10a} we get $\mathcal{F}_{[a,\infty]}([F])=0$ and hence $\mathcal{F}_{[a,b]}$ 
is injective by \prettyref{thm:fourier_unbounded_interval}.

We prove the surjectivity of the Fourier transform like in \prettyref{thm:fourier_unbounded_interval}. 
Let $K:=[a,b]$, $f\in\mathcal{FO}_{K}(E)$ and set $f_{-}:=f$ on the lower and $f_{-}:=0$ on the upper halfplane. 
We define 
$\nu :=\mathcal{F}_{\star}^{-1}(\mathscr{H}_{\overline{\R}}([f_{-}]))\in L(\mathcal{P}_{\ast}(\overline{\R}),E)$ 
and next we show that \eqref{P.0.0.1.1} holds for $K=[a,b]$.
We choose $-\tfrac{1}{n}<\widetilde{c}<\tfrac{1}{n}$ and have for $\varphi\in\mathcal{O}_{n}(U_{n}(\overline{\R}))$ 
and $\delta >0$ 
\[
 \nu(\varphi)
=-\frac{1}{2\pi}\int_{\im(\zeta)=-\delta}{f(\zeta)\int_{\im(z)=\widetilde{c}}{\varphi(z)e^{iz\zeta}\d z}\,\d\zeta}.
\]
like in \prettyref{thm:fourier_unbounded_interval}. We choose $\delta,\delta'>0$ according to \eqref{P.0.0.2} 
and define $I_{+}$ and $I_{-}$ like before. 
Furthermore, we set 
\[
  J_{++}
:=\int_{0-i\delta}^{\infty-i\delta}{f(\zeta)\int_{a-\delta'+i\delta}^{b+\delta'+i\delta}
  {\varphi(z)e^{iz\zeta}\d z}\,\d\zeta}
\]
and 
\[
  J_{-+}
:=\int_{-\infty-i\delta}^{0-i\delta}{f(\zeta)\int_{a-\delta'-i\delta}^{b+\delta'-i\delta}
  {\varphi(z)e^{iz\zeta}\d z}\,\d\zeta}.
\]
as well as for $y\in\R$ and $x\in\{a,b\}$
\[
J(-y;I_{x}):=\int_{0-iy}^{\infty-iy}{f(\zeta)J_{I_{x}}(\zeta)\d\zeta}
 \quad\text{and}\quad
J(-y;II_{x}):=\int_{-\infty-iy}^{0-iy}{f(\zeta)J_{II_{x}}(\zeta)\d\zeta}
\]
and 
\[
J_{I_{a}}(\zeta):=\int_{-\infty+i\delta}^{a-\delta'+i\delta}{\varphi(z)e^{iz\zeta}\d z}
 \quad\text{and}\quad
J_{II_{a}}(\zeta):=\int_{-\infty-i\delta}^{a-\delta'-i\delta}{\varphi(z)e^{iz\zeta}\d z}
\]
plus
\[
J_{I_{b}}(\zeta):=\int_{b+\delta'+i\delta}^{\infty+i\delta}{\varphi(z)e^{iz\zeta}\d z}
 \quad\text{and}\quad
J_{II_{b}}(\zeta):=\int_{b+\delta'-i\delta}^{\infty-i\delta}{\varphi(z)e^{iz\zeta}\d z}.
\]
Then we have $-2\pi\nu(\varphi)=I_{+}+I_{-}$ and $I_{+}=J_{++}+J(-\delta;I_{b})+J(-\delta;I_{a})$ as well as 
$I_{-}=J_{-+}+J(-\delta;II_{b})+J(-\delta;II_{a})$. 
If we choose $k\in\N$ such that $\tfrac{1}{k}<\delta$, we obtain for $\alpha\in\mathfrak{A}$
\begin{align}\label{P.0.0.10b}
 p_{\alpha}(J_{++})
&\leq \int_{0}^{\infty}{|f|_{k,\alpha,K}e^{\frac{1}{k}|t-i\delta|+H_{K}(-\delta)}
 \int_{a-\delta'}^{b+\delta'}{\|\varphi\|_{n,K}e^{-\frac{1}{n}|s|}e^{s\delta-t\delta}\d s}\,\d t}\notag\\
&\leq \frac{1}{\delta-\frac{1}{k}}(b-a+2\delta')e^{(\frac{1}{n}+\delta)(\max{(|a|,|b|)}+\delta')}
 e^{(\frac{1}{k}-a)\delta}|f|_{k,\alpha,K}\|\varphi\|_{n,K}
\end{align}
and an analogous estimate for $J_{-+}$. 

Moreover, we set
\[
J(\uparrow):=\int_{-i\delta}^{i\delta}{f(\zeta)J_{I_{b}}(\zeta)\d\zeta}
 \quad\text{and}\quad
J(\downarrow):=\int_{i\delta}^{-i\delta}{f(\zeta)J_{II_{b}}(\zeta)\d\zeta}.
\]
\begin{center}
\begin{minipage}{\linewidth}
\centering
\begin{tikzpicture}
\draw[blue,very thick] (1.8,-1) -- (4,-1) node [below=0.1pt] {\color{blue}$(-\delta;I_{b})$};
\draw[->,blue,very thick] (0,-1) -- (2,-1);
\draw[->,green!50!black,very thick] (0,1) -- (2,1);
\draw[green!50!black,very thick] (1.8,1) -- (4,1) node [above=0.1pt] {\color{green!50!black}$(\delta;I_{b})$};
\draw[dashed,very thick] (3,-1) -- (3,-0.5) node [left] {$\widetilde{\gamma}_{0}$};
\draw[->,dashed,very thick] (3,1) -- (3,-0.5);
\draw[->,cyan,very thick] (-4,-1)node [below=0.1pt] {\color{cyan}$(-\delta;II_{b})$} -- (-2,-1) ;
\draw[cyan,very thick] (-2.2,-1) -- (0,-1);
\draw[teal!50!white, very thick] (-0.03,-1) -- (-0.03,-0.3) node [left] {\color{cyan}$\downarrow$};
\draw[->,teal!50!white, very thick] (-0.03,1) -- (-0.03,-0.5);
\draw[green!50!black, very thick] (0.03,1) -- (0.03,0.3) node [right] {\color{green!50!black}$\uparrow$};
\draw[->,green!50!black, very thick] (0.03,-1) -- (0.03,0.5);
\draw[->,teal!50!white,very thick] (-4,1)node [above=0.1pt] {\color{teal!50!white}$(\delta;II_{b})$} -- (-2,1);
\draw[teal!50!white,very thick] (-2.2,1) -- (0,1) ;
\draw[dashed,very thick] (-3,1) -- (-3,0.5) node [right] {$\widetilde{\gamma}_{1}$};
\draw[->,dashed,very thick] (-3,-1) -- (-3,0.5);
\draw(-3 cm,1pt)--(-3cm,-1pt)node[anchor=north east] {$-r$};
\draw(3cm,1pt)--(3cm,-1pt)node[anchor=north west] {$r$};
\draw(1pt,-1cm)--(-1pt,-1cm)node[anchor=north east] {$-\delta$};
\draw(1pt,1cm)--(-1pt,1cm)node[anchor=south east] {$\delta$};
\draw[->] (-4,0) -- (4,0) node[right] {$\re(\zeta)$} coordinate (x axis);
\draw[->] (0,-2) -- (0,2) node[above] {$\im(\zeta)$} coordinate (y axis);
\end{tikzpicture}
\end{minipage}
\captionsetup{type=figure}
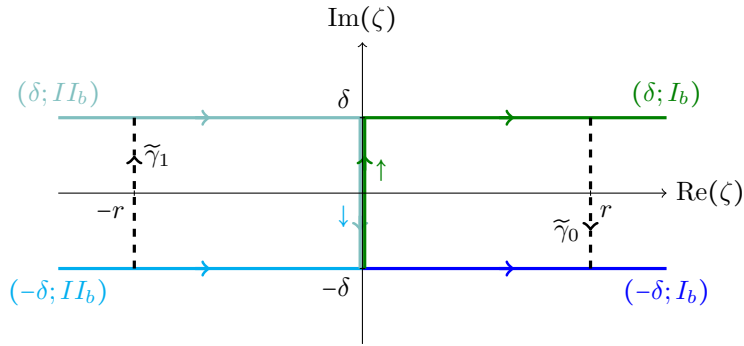
\caption{Some paths of integration in $\zeta$-plane for $K=[a,b]$ and $x=b$}
\label{fig:paths_zeta_plane_bounded}
\end{center}
\begin{center}
\begin{minipage}{\linewidth}
\centering
\begin{tikzpicture}
\def\mypath{
 (1,2) -- (-1,2) arc (90:270:20mm) (-1,-2) -- (1,-2) arc (270:450:20mm)}
\fill[fill=black!10,draw=black,very thick] \mypath;
\node (A) at (1,1.5) {$U_{n}(K)$};
\draw[->,dotted,very thick] (-1.5,1) -- (0.5,1);
\draw[dotted,very thick] (0.5,1) -- (1.5,1) node[anchor=north east]{$++$};
\draw[->,dotted,very thick] (-1.5,-1) -- (0.5,-1);
\draw[dotted,very thick] (0.5,-1) -- (1.5,-1) node[anchor=south east]{$-+$};
\draw[->,blue,very thick] (-4.5,1)node [above=0.1pt] {\color{blue}$I_{a}$} -- (-2.9,1);
\draw[blue,very thick] (-3,1) -- (-1.5,1);
\draw[->,green!50!black, very thick] (-1.5,-1) -- (-1.5,-0.5);
\draw[green!50!black, very thick] (-1.5,-0.5) -- (-1.5,1) node [anchor=north east] {\color{green!50!black}$III_{a}$};
\draw[->,teal!50!white,very thick] (-4.5,-1)node [below=0.1pt] {\color{teal!50!white}$II_{a}$} -- (-2.9,-1);
\draw[teal!50!white,very thick] (-3,-1) -- (-1.5,-1);
\draw[->,blue,very thick] (1.5,1) -- (3,1);
\draw[blue,very thick] (2.9,1) -- (4.5,1) node [above=0.1pt] {\color{blue}$I_{b}$};
\draw[->,green!50!black, very thick] (1.5,1) -- (1.5,0.5);
\draw[green!50!black, very thick] (1.5,0.5) node [right] {\color{green!50!black}$III_{b}$} -- (1.5,-1) ;
\draw[->,teal!50!white,very thick] (1.5,-1) -- (3,-1);
\draw[teal!50!white,very thick] (2.9,-1) -- (4.5,-1) node [below=0.1pt] {\color{teal!50!white}$II_{b}$};
\draw(-1.5cm,1pt)--(-1.5cm,-1pt)node[anchor=north east] {$a-\delta'$};
\draw(1.5cm,1pt)--(1.5cm,-1pt)node[anchor=north west] {$b+\delta'$};
\draw(1pt,-1cm)--(-1pt,-1cm)node[anchor=north east] {$-\delta$};
\draw(1pt,1cm)--(-1pt,1cm)node[anchor=south east] {$\delta$};
\draw(-1 cm,1pt)--(-1cm,-1pt)node[anchor=north] {$a$};
\draw(1 cm,1pt)--(1cm,-1pt)node[anchor=north] {$b$};
\draw(1pt,2 cm)--(-1pt,2 cm)node[anchor=south east] {$\nicefrac{1}{n}$};
\draw(1pt,-2 cm)--(-1pt,-2 cm)node[anchor=north east] {$-\nicefrac{1}{n}$};
\draw[->] (-4.5,0) -- (4.5,0) node[right] {$\re(z)$} coordinate (x axis);
\draw[red,thick] (-1,0)--(1,0);
\draw[->] (0,-2.5) -- (0,2.5) node[above] {$\im(z)$} coordinate (y axis);
\end{tikzpicture}
\end{minipage}
\captionsetup{type=figure}
\caption{Paths of integration in $z$-plane for $K=[a,b]$}
\end{center}
Since it holds for the path $\widetilde{\gamma}_{0}$ from \prettyref{fig:paths_zeta_plane_bounded} that
\begin{align*}
 p_{\alpha}\bigl(\int_{\widetilde{\gamma}_{0}}{f(\zeta)J_{I_{b}}(\zeta)}\bigr)
&\leq \int_{-\delta}^{\delta}{ |f|_{k,\alpha,K}e^{\frac{1}{k}|r+it|+H_{K}(t)}
 \int_{b+\delta'}^{\infty}{\|\varphi\|_{n,\overline{\R}}e^{-\frac{1}{n}|s|}e^{-r\delta+ts}\d s}\,\d t}\\
&\leq 2 |f|_{k,\alpha,K}\|\varphi\|_{n,\overline{\R}}e^{(\frac{1}{k}-\delta)r}
 \int_{-\delta}^{\delta}{e^{\frac{1}{k}|t|+\max{(at,bt)}}\d t}
 \int_{0}^{\infty}{e^{(\delta-\frac{1}{n})|s|}\d s}\\
&\leq \frac{4\delta}{\frac{1}{n}-\delta}|f|_{k,\alpha,K}\|\varphi\|_{n,\overline{\R}}
 e^{(\max{|a|,|b|)}+\frac{1}{k})\delta}e^{(\frac{1}{k}-\delta)r}\\
&\;\;\mathclap{\underset{\frac{1}{k}<\delta}{\to}}\quad 0,\;r\to \infty,
\end{align*}
we get that $J(-\delta;I_{b})=J(\delta;I_{b})+J(\uparrow)$ and analogously 
$J(-\delta;II_{b})=J(\delta;II_{b})+J(\downarrow)$ by Cauchy's integral theorem. 
Like in \eqref{P.0.0.7a} we have by Cauchy's integral theorem
\[
 J(\uparrow)+J(\downarrow)=\int_{-i \delta}^{i\delta}{f(\zeta)J_{III_{b}}(\zeta)\d\zeta}
\]
where
\[
 J_{III_{b}}(\zeta):=\int_{b+\delta'+i\delta}^{b+\delta'-i\delta}{\varphi(z)e^{iz\zeta}\d z}.
\]
Further, we estimate
\begin{align}\label{P.0.0.10c}
 p_{\alpha}(J(\uparrow)+J(\downarrow))
&\leq \int_{-\delta}^{\delta}{ |f|_{k,\alpha,K}e^{\frac{1}{k}|it|+H_{K}(t)}
 \int_{-\delta}^{\delta}{\|\varphi\|_{n,K}e^{-\frac{1}{n}|b+\delta'|}e^{-(b+\delta')t}\d s}\,\d t}\notag\\
&\leq 4\delta^{2}e^{(\frac{1}{k}+\max{|a|,|b|)}+|b+\delta'|)\delta}|f|_{k,\alpha,K}\|\varphi\|_{n,K}.
\end{align}
Altogether we obtain
\begin{flalign*}
&\hspace{0.35cm}-2\pi \nu(\varphi)\\
&=J_{++}+J(-\delta;I_{b})+J(-\delta;I_{a})+J_{-+}+J(-\delta;II_{b})+J(-\delta;II_{a})\\
&=J_{++}+J_{-+}+J(-\delta;I_{a})+J(-\delta;II_{a})+J(-\delta;I_{b})+J(-\delta;II_{b})\\
&=J_{++}+J_{-+}+J(-\delta;I_{a})+J(-\delta;II_{a})+J(\delta;I_{b})+J(\delta;II_{b})+J(\uparrow)+J(\downarrow)\\
&=J_{++}+J_{-+}+(J(-\delta;I_{a})+J(-\delta;II_{a}))+(J(\delta;I_{b})+J(\delta;II_{b}))+(J(\uparrow)+J(\downarrow)).
\end{flalign*}
By the same method as in \prettyref{thm:fourier_unbounded_interval} we get that 
\[
J(-\delta;I_{a})+J(-\delta;II_{a})=\lim_{\eta\to\infty}\int_{-i\delta}^{-i\eta}f(\zeta)
\int_{a-\delta'-i\delta}^{a-\delta'+i\delta}\varphi(z)e^{iz\zeta}\d z\;\d\zeta=:\lim_{\eta\to\infty}J(-\eta;III_{a})
\]
and 
\[
J(\delta;I_{b})+J(\delta;II_{b})=\lim_{\eta\to\infty}\int_{i\delta}^{i\eta}f(\zeta)J_{III_{b}}(\zeta)\d\zeta
=:\lim_{\eta\to\infty}J(\eta;III_{b}).
\]
Thus $p_{\alpha}(J(-\delta;I_{a})+J(-\delta;II_{a}))$ and 
$p_{\alpha}(J(\delta;I_{b})+J(\delta;II_{b}))$ for $\alpha\in\mathfrak{A}$ can be estimated 
by the right-hand side of \eqref{P.0.0.9} with $K=[a,b]$, using that $H_{K}(-t)=-at$ and $H_{K}(t)=bt$ for $t>0$.
Combining the estimates \eqref{P.0.0.10b}, \eqref{P.0.0.10c} and \eqref{P.0.0.9}, we get \eqref{P.0.0.1.1}.
Since $\mathscr{H}^{-1}_{[a,b]}(\widetilde{\nu})=\mathscr{H}^{-1}_{[a,\infty]}(\widetilde{\nu})$ holds 
by \cite[Eq.\ (11), p.\ 20]{kruse2019_2} for the unique extension $\widetilde{\nu}\in L(\mathcal{P}_{\ast}([a,b]),E)$ 
of $\nu=\mathcal{F}_{\star}^{-1}(\mathscr{H}_{\overline{\R}}([f_{-}]))$, 
it follows by \prettyref{thm:fourier_unbounded_interval} that
\[
 \mathcal{F}_{[a,b]}(-\mathscr{H}^{-1}_{[a,b]}(\widetilde{\nu}))
\underset{\eqref{P.0.0.10a}}{=}\mathcal{F}_{[a,\infty]}(-\mathscr{H}^{-1}_{[a,b]}(\widetilde{\nu}))
=\mathcal{F}_{[a,\infty]}(-\mathscr{H}^{-1}_{[a,\infty]}(\widetilde{\nu}))
\underset{\eqref{eq:surjectivity_unbounded}}{=}f_{-}
\]
holds on the lower halfplane. The left-hand side is an entire function and thus coincides with $f$ on $\C$ 
by the identity theorem, implying the surjectivity of $\mathcal{F}_{[a,b]}$. 
The continuity of the inverse $\mathcal{F}^{-1}_{[a,b]}$ is a consequence of \eqref{P.0.0.1.1} like before.
\end{proof}

\begin{rem}\label{P.0.0.10}
a) In short we express the inverse of $\mathcal{F}_{K}$ for a non-empty compact interval $K\neq\overline{\R}$ 
as
\[
 \mathcal{F}_{[a,\infty]}^{-1}=-\mathscr{H}_{[a,\infty]}^{-1}\circ \mathcal{F}_{\star}^{-1}\circ 
 \mathscr{H}_{\overline{\R}}
  \quad\text{and}\quad 
 \mathcal{F}_{[-\infty,b]}^{-1}=\mathscr{H}_{[-\infty,b]}^{-1}\circ \mathcal{F}_{\star}^{-1}\circ 
 \mathscr{H}_{\overline{\R}}
\]
for $-\infty<a\leq\infty$ and $-\infty\leq b<\infty$, resp.\ as
\[
 \mathcal{F}_{[a,b]}^{-1}(f)=-(\mathscr{H}_{[a,b]}^{-1}\circ \mathcal{F}_{\star}^{-1}\circ 
 \mathscr{H}_{\overline{\R}})([f_{-}])=(\mathscr{H}_{[a,b]}^{-1}\circ \mathcal{F}_{\star}^{-1}\circ 
 \mathscr{H}_{\overline{\R}})([f_{+}])
\]
for $-\infty<a\leq b<\infty$ where $f_{+}:=f$ on the upper halfplane and $f_{+}:=0$ on the lower halfplane, 
since there is a unique extension of $\mathcal{F}_{\star}^{-1}(\mathscr{H}_{\overline{\R}}([f]))$ resp.\ 
$\mathcal{F}_{\star}^{-1}(\mathscr{H}_{\overline{\R}}([f_{\pm}]))$ to $L(\mathcal{P}_{\ast}(K),E)$ 
for every $f\in\mathcal{FO}_{K}(E)$.

b) Let $E$ be sequentially complete. 
If $E$ is strictly admissible, then the sheaf of $E$-valued Fourier hyperfunctions is flabby by 
\cite[Theorem 5.9 b), p.\ 33]{kruse2019_5}. Hence for $f\in bv_{\overline{\R}}(E)$ and $a\in\R$ there are 
$f_{1}\in bv_{[-\infty,a]}(E)$ and $f_{2}\in bv_{[a,\infty]}(E)$ such that $f=f_{1}+f_{2}$ 
by \cite[Lemma 1.4.4, p.\ 36]{Kan}. Due to \cite[Eq.\ (6), p.\ 14]{kruse2019_2} we have
\[
\mathscr{H}_{\overline{\R}}(f_{1}) = \mathscr{H}_{[-\infty,a]}(f_{1}) 
 \quad\text{and}\quad  
\mathscr{H}_{\overline{\R}}(f_{2}) = \mathscr{H}_{[a,\infty]}(f_{2})
\]
on $\mathcal{P}_{\ast}(\overline{\R})$. Thus we get
\begin{align*}
  (\mathcal{F}_{\star}\circ \mathscr{H}_{\overline{\R}})(f)
&=(\mathcal{F}_{\star}\circ \mathscr{H}_{\overline{\R}})(f_{1})
 +(\mathcal{F}_{\star}\circ \mathscr{H}_{\overline{\R}})(f_{2})\\
&=(\mathcal{F}_{\star}\circ \mathscr{H}_{[-\infty,a]})(f_{1})
 +(\mathcal{F}_{\star}\circ \mathscr{H}_{[a,\infty]})(f_{2})\\
&\;\;\mathclap{\underset{\eqref{P.0.1.1}}{=}}\;\;\mathscr{H}_{\overline{\R}}([\mathcal{F}_{[-\infty,a]}(f_{1})])
 -\mathscr{H}_{\overline{\R}}([\mathcal{F}_{[a,\infty]}(f_{2})]).
\end{align*}
Therefore we may write the Fourier transform of $f$ by \prettyref{cor:fourier_isom_extended_reals} as
\[
 \mathcal{F}_{\overline{\R}}(f)
=(\mathscr{H}_{\overline{\R}}^{-1}\circ \mathcal{F}_{\star}\circ \mathscr{H}_{\overline{\R}})(f)
=[\mathcal{F}_{[-\infty,a]}(f_{1})-\mathcal{F}_{[a,\infty]}(f_{2})]
\]
keeping \prettyref{prop:upper_lower_extension} in mind 
(cf.\ \cite[Theorem 3.2.6, p.\ 483, Definition 3.2.7, p.\ 484]{Kawai} for $a=0$ and $E=\C$). 
If $f\in bv_{K}(E)$, such that $K\subset[a,b]\subset\R$ compact, we get by \eqref{P.0.0.10a} that 
\[
 \mathcal{F}_{\overline{\R}}(f)=[\mathcal{F}_{[-\infty,b]}(f)]=[\mathcal{F}_{[a,b]}(f)_{+}]
 =-[\mathcal{F}_{[a,\infty]}(f)]=-[\mathcal{F}_{[a,b]}(f)_{-}].
\]
\end{rem}

We close this section with a discussion of some standard operations and their connection to the Fourier transform 
$\mathcal{F}_{K}$ for compact intervals $\varnothing\neq K\varsubsetneq\overline{\R}$. 
We have the following counterpart of \cite[Example 2.5, p.\ 47]{langenbruch2011}, 
which follows from a simple calculation. 

\begin{prop}\label{prop:fourier_transform_of_shift}
Let $E$ be a $\C$-lcHs and $\varnothing\neq K\varsubsetneq\overline{\R}$ a compact interval. Let 
$\tau_{h}([F]):=[F(\cdot -h)]$, $h\in\R$, be the shift operator from $bv_{K}(E)$ to $bv_{h+K}(E)$. If 
\begin{enumerate}
\item[(i)] $K\subset\R$ and $E$ is locally complete, or if
\item[(ii)] $E$ is sequentially complete,
\end{enumerate}
then we have for $[F]\in bv_{K}(E)$
\[
\mathcal{F}_{h+K}(\tau_{h}([F]))=e^{-ih(\cdot)}\mathcal{F}_{K}([F]).
\]
\end{prop}

Next, we transfer \cite[Example 2.6, p.\ 48]{langenbruch2011} and \cite[Proposition 2.9, p.\ 49]{langenbruch2011} 
to our setting. 

\begin{prop}\label{prop:fourier_transform_of_derivative}
Let $E$ be a sequentially complete $\C$-lcHs, $\varnothing\neq K\varsubsetneq\overline{\R}$ a compact interval and 
$P(-i\partial):=\sum_{k=0}^{\infty}\tfrac{c_{k}}{k!}(-i\partial)^{k}$ where $(c_{k})\subset\C$ and 
$P$ is of \emph{exponential type $0$}, i.e.\ 
\[
\forall\;\varepsilon>0\;\exists\;C>0\;\forall\;k\in\N_{0}:\;|c_{k}|=|P^{(k)}(0)|\leq C\varepsilon^{k}.
\]
Then $P(-i\partial)$ resp.\ $P(-i\partial)([F]):= [P(-i\partial)F]$ 
and the multiplication operator $M_{P}(F):=PF$ resp.\ $M_{P}([F]):= [PF]$ are well-defined 
continuous linear operators on 
$\mathcal{O}^{exp}(\overline{\C}\setminus K,E)$ and
$\mathcal{O}^{exp}(\overline{\C},E)$ 
resp.\ $bv_{K}(E)$, and for $[F]\in bv_{K}(E)$ we have
\begin{equation}\label{eq:fourier_transform_of_derivative}
\mathcal{F}_{K}(P(-i\partial)[F])(\zeta)=P(\zeta)\mathcal{F}_{K}([F])(\zeta)
\end{equation}
and
\begin{equation}\label{eq:fourier_transform_of_multiplication}
\mathcal{F}_{K}(P[F])(\zeta)=P(i\partial)\mathcal{F}_{K}([F])(\zeta)
\end{equation}
for $\im(\zeta)<0$ if $K=[a,\infty]$, for $\im(\zeta)>0$ if $K=[-\infty,a]$, and for $\zeta\in\C$ if 
$K\subset\R$, respectively. 
\end{prop}
\begin{proof}
Let $F\in \mathcal{O}^{exp}(\overline{\C}\setminus K,E)$ and $n\in\N$. 
We choose $0<r<\tfrac{1}{2n}$ and remark that
\[
-\frac{1}{n}|\re(z)|\leq -\frac{1}{n}|\re(\zeta)|+\frac{1}{n}|\re(z)-\re(\zeta)|\leq -\frac{1}{2n}|\re(\zeta)|
+\frac{r}{n},
\quad z,\zeta\in\C,\;|\zeta-z|=r.
\]
By Cauchy's inequality \cite[Corollary 5.3 a), p.\ 263]{kruse2019_4} 
we have for $k\in\N_{0}$ and $\alpha\in\mathfrak{A}$
\[
 p_{\alpha}(\partial^{k}F(z))
\leq\frac{k!}{r^{k}}\sup_{\zeta\in\C,|\zeta-z|=r}p_{\alpha}(F(\zeta)),\quad z\in S_{n}(K),
\]
which implies 
\begin{flalign*}
&\hspace{0.35cm}\sup_{z\in S_{n}(K)}
 p_{\alpha}(\partial^{k}F(z))e^{-\frac{1}{n}|\re(z)|}
 \leq\sup_{z\in S_{n}(K)}
 \frac{k!}{r^{k}}\sup_{\zeta\in\C,|\zeta-z|=r}p_{\alpha}(F(\zeta))e^{-\frac{1}{n}|\re(z)|}\\
&\leq e^{\frac{r}{n}}\frac{k!}{r^{k}}\sup_{\zeta\in S_{2n}(K)}p_{\alpha}(F(\zeta))e^{-\frac{1}{2n}|\re(\zeta)|}
 = e^{\frac{r}{n}}\frac{k!}{r^{k}}\vertiii{F}_{2n,\alpha,K}.
\end{flalign*}
Since $P$ is of exponential type $0$, we have for $0<\varepsilon<r$ and $m,l\in\N_{0}$ with $m\geq l$
\begin{align*}
 \vertiii{\sum_{k=l}^{m}\frac{c_{k}}{k!}(-i\partial)^{k}F}_{n,\alpha,K}
&\leq \sup_{z\in S_{n}(K)}\sum_{k=l}^{m}\tfrac{|c_{k}|}{k!}p_{\alpha}(\partial^{k}F(z))e^{-\frac{1}{n}|\re(z)|}\\
&\leq C_{\varepsilon}e^{\frac{r}{n}}\sum_{k=l}^{m}\bigl(\frac{\varepsilon}{r}\bigr)^{k}\vertiii{F}_{2n,\alpha,K},
\end{align*}
which implies that the sum defining $P(-i\partial)$ converges on $\mathcal{O}^{exp}(\overline{\C}\setminus K,E)$ 
because this space is sequentially complete due to the sequential completeness of $E$ and 
$\mathcal{O}^{exp}(\overline{\C}\setminus K)$ combined with 
$\mathcal{O}^{exp}(\overline{\C}\setminus K,E)\cong \mathcal{O}^{exp}(\overline{\C}\setminus K)\varepsilon E$ 
by \cite[Remark 3.4 b), p.\ 8]{kruse2019_5}. Moreover, we deduce that
\[
 \vertiii{P(-i\partial)F}_{n,\alpha,K}
\leq C_{\varepsilon}e^{\frac{r}{n}}\frac{1}{1-\frac{\varepsilon}{r}}\vertiii{F}_{2n,\alpha,K},
\]
yielding that $P(-i\partial)$ is a continuous linear operator on $\mathcal{O}^{exp}(\overline{\C}\setminus K,E)$. 
Analogously, we obtain that $P(-i\partial)$ is a continuous linear operator on $\mathcal{O}^{exp}(\overline{\C},E)$.
Hence $P(-i\partial)$ is a well-defined continuous linear operator on $bv_{K}(E)$ as well. 

Let us turn to \eqref{eq:fourier_transform_of_derivative}. If $P(-i\partial)$ is of finite order, then 
\begin{align*}
 \langle e', \mathcal{F}_{K}(P(-i\partial)[F])(\zeta)\rangle
&= \mathcal{F}_{K}(P(-i\partial)[e'\circ F])(\zeta)
 =P(\zeta)\mathcal{F}_{K}([e'\circ F])(\zeta)\\
&=\langle e', P(\zeta)\mathcal{F}_{K}([F])(\zeta)\rangle,\quad e'\in E'
\end{align*}
holds by partial integration, which implies \eqref{eq:fourier_transform_of_derivative} by the Hahn-Banach theorem. 
If $P(-i\partial)$ is of infinite order, then the sum defining $P(-i\partial)$ converges on $bv_{K}(E)$. In combination 
with the continuity of $\mathcal{F}_{K}$ this proves \eqref{eq:fourier_transform_of_derivative}.

The proof in the case of the multiplication operator $M_{P}$ is quite similar. 
Since $P$ is of exponential type $0$, we have for $0<\varepsilon<\tfrac{1}{2n}$ and $m,l\in\N_{0}$ with $m\geq l$
\begin{align*}
 \vertiii{\sum_{k=l}^{m}\frac{c_{k}}{k!}(\cdot)^{k}F}_{n,\alpha,K}
&\leq C_{\varepsilon}\sup_{z\in S_{n}(K)}\sum_{k=l}^{m}\tfrac{\varepsilon^{k}}{k!}|z|^{k}e^{-\frac{1}{2n}|\re(z)|}
 p_{\alpha}(F(z))e^{-\frac{1}{2n}|\re(z)|}\\
&\leq C_{\varepsilon}e^{\frac{1}{2}}\sum_{k=l}^{m}\tfrac{\varepsilon^{k}}{k!}\sup_{z\in S_{n}(K)}|z|^{k}
 e^{-\frac{1}{2n}|z|}\vertiii{F}_{2n,\alpha,K}\\
&\leq C_{\varepsilon}e^{\frac{1}{2}}\sum_{k=l}^{m}\tfrac{\varepsilon^{k}}{k!}k!(2n)^{k}\vertiii{F}_{2n,\alpha,K}
 = C_{\varepsilon}e^{\frac{1}{2}}\sum_{k=l}^{m}(2\varepsilon n)^{k}\vertiii{F}_{2n,\alpha,K},
\end{align*}
which yields that the sum defining $P$ converges on $\mathcal{O}^{exp}(\overline{\C}\setminus K,E)$ as above. 
Further, we derive that for $\varepsilon<\tfrac{1}{2n}$ 
\[
\vertiii{PF}_{n,\alpha,K}\leq C_{\varepsilon}e^{\frac{1}{2}}\frac{1}{1-2\varepsilon n}\vertiii{F}_{2n,\alpha,K},
\]
implying that $M_{P}$ is a continuous linear operator on $\mathcal{O}^{exp}(\overline{\C}\setminus K,E)$. 
Analogously, we obtain that $M_{P}$ is a continuous linear operator on $\mathcal{O}^{exp}(\overline{\C},E)$
and hence a well-defined continuous linear operator on $bv_{K}(E)$ as well.  

Let us turn to \eqref{eq:fourier_transform_of_multiplication}. If $P$ is a polynomial, then 
\begin{align*}
 \langle e', \mathcal{F}_{K}(P[F])(\zeta)\rangle
&= \mathcal{F}_{K}(P[e'\circ F])(\zeta)
 =P(i\partial)\mathcal{F}_{K}([e'\circ F])(\zeta)\\
&=\langle e', P(i\partial)\mathcal{F}_{K}([F])(\zeta)\rangle,\quad e'\in E'
\end{align*}
holds by differentiation w.r.t.\ the parameter $\zeta$, 
which implies \eqref{eq:fourier_transform_of_multiplication} by the Hahn-Banach theorem. 
If $P$ is not a polynomial, then we have for its Taylor series $(P_{j})$ and $g:=\mathcal{F}_{K}([F])$ by 
Cauchy's inequality with $0<\varepsilon<r<\tfrac{1}{2n}$
\[
 p_{\alpha}\bigl(\sum_{k=l}^{m}\frac{c_{k}}{k!}(i\partial)^{k}g(\zeta)\bigr)
\leq C_{\varepsilon}\sum_{k=l}^{m}\bigl(\frac{\varepsilon}{r}\bigr)^{k}\sup_{z\in\C,|z-\zeta|=r}p_{\alpha}(g(z)),
\]
which implies that $P_{j}(i\partial)g$ converges to $P(i\partial)g$ as $j\to\infty$ 
in the sequentially complete space $\mathcal{O}(-\h,E)$ if $K=[a,\infty]$, 
in $\mathcal{O}(\h,E)$ if $K=[-\infty,a]$ and in $\mathcal{O}(\C,E)$ if $K\subset\R$ 
with respect to the topology of uniform convergence on compact subsets, respectively. 
In combination with the continuity of $\mathcal{F}_{K}$ this proves \eqref{eq:fourier_transform_of_multiplication}.
\end{proof}

In \cite[Proposition 2.9, p.\ 49]{langenbruch2011} it is allowed that the $P$
in \prettyref{eq:fourier_transform_of_multiplication} is an entire function of exponential type because there the 
test functions (see \prettyref{def:test_functions_langenbruch} or the introduction) are exponentially decreasing of type $-\infty$ 
and not exponentially increasing like ours. At the end of this section we treat convolutions.

\begin{thm}\label{thm:convolution}
Let $(E,(p_{\alpha})_{\alpha\in\mathfrak{A}})$, $(E_{1},(p_{\beta})_{\beta\in\mathfrak{B}})$ 
and $(E_{2},(p_{\omega})_{\omega\in\Omega})$ be sequentially complete $\C$-lcHs such that 
a canonical bilinear map $\cdot\:\colon E_{1}\times E_{2}\to E$ is defined with the property
\begin{equation}\label{P.0.0.11}
\forall\;\alpha\in\mathfrak{A}\;\exists\;\beta\in\mathfrak{B},\,\omega\in\Omega,\,D>0\;\forall\;x\in E_{1},\,y\in E_{2}:
\;p_{\alpha}(x\cdot y)\leq D p_{\beta}(x)p_{\omega}(y).
\end{equation}
Let $a,b,c,d\in\overline{\R}$ and $a,c\neq -\infty$ or $b,d\neq \infty$.
\begin{enumerate}
\item[a)] If $f\in\mathcal{FO}_{[a,b]}(E_{1})$ and $g\in\mathcal{FO}_{[c,d]}(E_{2})$, 
then $fg\in\mathcal{FO}_{[a+c,b+d]}(E)$ where $fg$ is defined by $(fg)(z):=f(z)\cdot g(z)$. 
\item[b)] We define the convolution
\[
\ast\:\colon bv_{[a,b]}(E_{1})\times bv_{[c,d]}(E_{2})\to bv_{[a+c,b+d]}(E),\;
f\ast g:=\mathcal{F}^{-1}(\mathcal{F}(f)\mathcal{F}(g)).
\]
The convolution is well-defined, bilinear and continuous.
\item[c)] If $E_{1}=E_{2}$ and $\cdot\:\colon E_{1}\times E_{2}\to E$ is commutative, 
then the convolution is commutative as well. 
If $E=E_{1}=E_{2}=\C$ and $\cdot\:\colon E_{1}\times E_{2}\to E$ is the multiplication, 
then the convolution is associative.
\end{enumerate}
\end{thm}
\begin{proof}
a) Let $a,c\neq -\infty$ and $b,d=\infty$. Let $f\in\mathcal{FO}_{[a,\infty]}(E_{1})$ 
and $g\in\mathcal{FO}_{[c,\infty]}(E_{2})$. Due to the bilinearity of the map $\cdot\:\colon E_{1}\times E_{2}\to E$ 
we get for $z\in-\h$ and $h\in\C$, $h\neq 0$,
\begin{flalign*}
&\quad\;\frac{f(z+h)\cdot g(z+h)-f(z)\cdot g(z)}{h}-f'(z)\cdot g(z)-f(z)\cdot g'(z)\\
&= \bigl(\frac{f(z+h)-f(z)}{h}-f'(z)\bigr)\cdot g(z+h)+f'(z)\cdot (g(z+h)-g(z)) \\
&\phantom{=}+f(z)\cdot\bigl(\frac{g(z+h)-g(z)}{h}-g'(z)\bigr).
\end{flalign*}
This implies that $fg\in\mathcal{O}(-\h,E)$ by \eqref{P.0.0.11}. 
First, we consider the cases $a,c\in\R$ or $a=c=\infty$. Let $\alpha\in\mathfrak{A}$ and $k\in\N$. 
Due to \eqref{P.0.0.11} there are $\beta\in\mathfrak{B}$, $\omega\in\Omega$ and $D>0$ such that
\[
 \|fg\|_{k,\alpha,[a+c,\infty]}\leq D\|f\|_{2k,\beta,[a,\infty]}\|g\|_{2k,\omega,[c,\infty]}.
\]
Now, let $a:=\infty$ or $c:=\infty$ but $a\neq c$. W.l.o.g.\ $a=\infty$. 
Then there is $n\in\N$ such that $k+|c|\leq n$ and 
\[
 \|fg\|_{k,\alpha,\{\infty\}}\leq D\|f\|_{2n,\beta,\{\infty\}}\|g\|_{2n,\omega,[c,\infty]}.
\]
Hence $fg\in\mathcal{FO}_{[a+c,\infty]}(E)$. The proof in the other cases is analogous.

b) The convolution is well-defined by part a) and \prettyref{thm:fourier_unbounded_interval} resp.\
\prettyref{thm:fourier_bounded_interval}. The bilinearity is due to the bilinearity of the map 
$\cdot\:\colon E_{1}\times E_{2}\to E$. The continuity follows directly from \eqref{P.0.0.11} 
and \prettyref{thm:fourier_unbounded_interval} resp.\ \prettyref{thm:fourier_bounded_interval}.

c) This is obviously true.
\end{proof}

\begin{rem}
In the following cases condition \eqref{P.0.0.11} is fulfilled.
\begin{enumerate}
\item [a)] Let $E_{1}=E_{2}$ and $E$ be complex Banach spaces and $a\colon E_{1}\times E_{1}\to E$ 
a continuous bilinear form. Define $x\cdot y:= a(x,y)$ for $x,y\in E_{1}$. If $a$ is symmetric as well, 
then $\cdot$ is commutative.
\item [b)] Let $E_{2},E$ be complex Banach spaces and $E_{1}:=L(E_{2},E)$. Define $x\cdot y:= x(y)$ 
for $x\in E_{1}$, $y\in E_{2}$. 
\item [c)] Let $X,Y,Z$ be complex Banach spaces, $E_{1}:=L(X,Y)$, $E_{2}:=L(Y,X)$ and $E:=L(Y,Z)$. 
Define $x \cdot y:= x\circ y$ for $x\in E_{1}$, $y\in E_{2}$. 
\end{enumerate}
\end{rem}

If one of the Fourier hyperfunctions has real compact support, we may define the convolution in a more common way.

\begin{cor}\label{cor:convolution_standard}
Let $(E,(p_{\alpha})_{\alpha\in\mathfrak{A}})$, $(E_{1},(p_{\beta})_{\beta\in\mathfrak{B}})$ 
and $(E_{2},(p_{\omega})_{\omega\in\Omega})$ be sequentially complete $\C$-lcHs such that 
a canonical bilinear map $\cdot\:\colon E_{1}\times E_{2}\to E$ is defined with the property
\[
\forall\; \alpha\in\mathfrak{A}\;\exists\;\beta\in\mathfrak{B},\,\omega\in\Omega,\,D>0\;\forall\;x\in E_{1},\,y\in E_{2}:
\;p_{\alpha}(x\cdot y)\leq D p_{\beta}(x)p_{\omega}(y).
\]
Let $[a,b]\subset\R$ and $[c,d]\varsubsetneq\overline{\R}$. We set 
\begin{gather*}
 \circledast\:\colon bv_{[a,b]}(E_{1})\times bv_{[c,d]}(E_{2})\to bv_{[a+c,b+d]}(E),\\
[F] \circledast [G]:=\bigl[z\mapsto \int_{\gamma_{[a,b],2n}}F(w)G(z-w)\d w\bigr],
\end{gather*}
where for $z\in S_{n}([a+c,b+d])$, $n\in\N$, the path $\gamma_{[a,b],2n}$ goes along the boundary of $U_{2n}([a,b])$ 
with clockwise orientation. Then we have
\[
[F] \circledast [G]=[F]\ast[G]
\]
for $[F]\in  bv_{[a,b]}(E_{1})$ and $[G]\in bv_{[c,d]}(E_{2})$. 
\end{cor}
\begin{proof}
We restrict our considerations to the case that $[c,d]\neq \pm\{\infty\}$. The other case is similar.

(i) Let $[F]\in  bv_{[a,b]}(E_{1})$, $[G]\in bv_{[c,d]}(E_{2})$, $n\in\N$ and $\tfrac{1}{2n}<\widetilde{c}<\tfrac{1}{n}$.
Then 
\[
f\colon \Bigl(U_{\nicefrac{1}{\widetilde{c}}}([a,b])\setminus \overline{U_{3n}([a,b])}\Bigr)\times S_{n}([a+c,b+d])\to E,\; 
f(w,z):=F(w)G(z-w),
\] 
is well-defined, i.e.\ $z-w\in\C\setminus [c,d]$ for all
$w\in V:=U_{\nicefrac{1}{\widetilde{c}}}([a,b])\setminus \overline{U_{3n}([a,b])}$ and $z\in U:=S_{n}([a+c,b+d])$. 
Let $w\in V$ and $z\in U$ such that $y:=\im(z)=\im(w)$. This can only happen if $|y|=|\im(z)|<\widetilde{c}$. 
Then $z=a+c-\sqrt{r^{2}-y^2}+iy$ or $z=b+d+\sqrt{r^{2}-y^2}+iy$ for some $r>\tfrac{1}{n}$. 
If $w=t+iy$ with some $t\in[a,b]$, then 
\[
z-w=c+a-t-\sqrt{r^{2}-y^2}<c
\;\text{or}\;
z-w=d+b-t+\sqrt{r^{2}-y^2}>d.
\] 
If $w=a-\sqrt{R^{2}-y^2}+iy$ for some $\tfrac{1}{3n}<R<\widetilde{c}$, then $R<r$ and 
\[
z-w=c+\sqrt{R^{2}-y^2}-\sqrt{r^{2}-y^2}<c
\;\text{or}\;
z-w=d+(b-a)+\sqrt{r^{2}-y^2}+\sqrt{R^{2}-y^2}>d.
\] 
If $w=b+\sqrt{R^{2}-y^2}+iy$ for some $\tfrac{1}{3n}<R<\widetilde{c}$, then 
\[
z-w=c+(a-b)+\sqrt{R^{2}-y^2}-\sqrt{r^{2}-y^2}<c
\;\text{or}\;
z-w=d+\sqrt{r^{2}-y^2}-\sqrt{R^{2}-y^2}>d.
\] 
Hence we have $z-w\in\C\setminus\R$ or $z-w\in\R\setminus [c,d]$ for all $w\in V$ and $z\in U$, i.e.\ 
$z-w\in\C\setminus [c,d]$. 

(ii) It follows from \eqref{P.0.0.11} that $f(w,\cdot)$ is holomorphic on $U$ for all $w\in V$, 
and that $f(\cdot,z)$ is holomorphic on $V$ for all $z\in U$. Therefore 
\[
z\mapsto \int_{\gamma_{[a,b],2n}}F(w)G(z-w)\d w
\]
is holomorphic on $U=S_{n}([a+c,b+d])$ by differentiation under the Pettis integral w.r.t.\ the parameter $z$
(see \cite[Lemma 4.8 b), p.\ 259]{kruse2019_4}). Due to Cauchy's integral theorem 
\[
\int_{\gamma_{[a,b],2n}}F(w)G(z-w)\d w =\int_{\gamma_{[a,b],2k}}F(w)G(z-w)\d w
\] 
for $z\in S_{n}([a+c,b+d])$ and $k\in\N$ with $k>n$. Thus we get a well-defined holomorphic function 
$F\circledast G \colon\C\setminus [a+c,b+d]\to E$ by setting
\[
(F\circledast G)(z):=\int_{\gamma_{[a,b],2n}}F(w)G(z-w)\d w,\quad z\in  S_{n}([a+c,b+d]),
\]
for $n\in\N$.

(iii) Next, we show that $F\circledast G\in\mathcal{O}^{exp}(\overline{\C}\setminus [a+c,b+d],E)$. We claim that 
$z-\gamma_{[a,b],2n}\subset S_{2n}([c,d])$ for every $z\in S_{n}([a+c,b+d])$. Let the path $\gamma_{[a,b],2n}$ 
be parametrised by $[0,1]$. 

(iii.1) We note that 
\[
|\im(z-\gamma_{[a,b],2n}(t))|< n+\frac{1}{2n}<2n
\]
for all $z\in S_{n}([a+c,b+d])$ and $t\in [0,1]$. If $z\in S_{n}([a+c,b+d])$ with $|\im(z)|>\tfrac{1}{n}$, then 
\[
|\im(z-\gamma_{[a,b],2n}(t))|> \frac{1}{n}-\frac{1}{2n}=\frac{1}{2n}
\]
for all $t\in[0,1]$. 

(iii.2) Let $z\in S_{n}([a+c,b+d])$ with $|y|:=|\im(z)|\leq\tfrac{1}{n}$. Then we have $\re(z)<a+c$ or 
$\re(z)>b+d$. In particular, there is $r>\tfrac{1}{n}$ such that 
\[
z=a+c-\sqrt{r^{2}-y^2}+iy
\;\text{or}\;
z=b+d+\sqrt{r^{2}-y^2}+iy.
\] 

(iii.2.1) Let $t\in [0,1]$ such that $\gamma_{[a,b],2n}(t)= t_{0}\pm i\tfrac{1}{2n}$ for some $t_{0}\in[a,b]$. 
If $\re(z)<a+c$, then 
\[
z-\gamma_{[a,b],2n}(t)=c+a-t_{0}-\sqrt{r^{2}-y^2}+i\bigl(y\mp\frac{1}{2n}\bigr)
\]
and thus $\re(z-\gamma_{[a,b],2n}(t))<c$. This implies 
\begin{flalign*}
&\hspace{0.35cm} \d (z-\gamma_{[a,b],2n}(t),[c,d]\cap\C)\\
&=|z-\gamma_{[a,b],2n}(t)-c|
 =\sqrt{(a-t_{0})^{2}-2(a-t_{0})\sqrt{r^{2}-y^{2}}+r^{2}-y^{2}+y^{2}\mp y\frac{1}{n}+\frac{1}{4n^{2}}}\\
&\geq \sqrt{r^{2}\mp y\frac{1}{n}+\frac{1}{4n^{2}}}
 \geq \sqrt{r^{2}-\frac{1}{n^{2}}+\frac{1}{4n^{2}}}
 > \frac{1}{2n}.
\end{flalign*}
If $\re(z)>b+d$, then 
\[
z-\gamma_{[a,b],2n}(t)=d+b-t_{0}+\sqrt{r^{2}-y^2}+i\bigl(y\mp\frac{1}{2n}\bigr)
\]
and thus $\re(z-\gamma_{[a,b],2n}(t))>d$. This implies 
\begin{flalign*}
&\hspace{0.35cm} \d (z-\gamma_{[a,b],2n}(t),[c,d]\cap\C)\\
&=|z-\gamma_{[a,b],2n}(t)-d|
 =\sqrt{(b-t_{0})^{2}+2(b-t_{0})\sqrt{r^{2}-y^{2}}+r^{2}-y^{2}+y^{2}\mp y\frac{1}{n}+\frac{1}{4n^{2}}}\\
&\geq \sqrt{r^{2}\mp y\frac{1}{n}+\frac{1}{4n^{2}}}
 > \frac{1}{2n}.
\end{flalign*}

(iii.2.2) Let $t\in [0,1]$ such that $\gamma_{[a,b],2n}(t)=a-\sqrt{\frac{1}{4n^{2}}-y_{1}^2}+iy_{1}$ for some 
$|y_{1}|<\tfrac{1}{2n}$. If $\re(z)<a+c$, then 
\[
z-\gamma_{[a,b],2n}(t)=c+\sqrt{\frac{1}{4n^{2}}-y_{1}^2}-\sqrt{r^{2}-y^2}+i(y-y_{1}).
\]
If $|y-y_{1}|\leq\tfrac{1}{2n}$, then 
\[
|y_{1}^{2}-y^{2}|=|y_{1}-y||y_{1}+y|<\frac{1}{2n}\bigl(\frac{1}{2n}+\frac{1}{n}\bigr)=\frac{3}{4n^{2}},
\]
which yields
\[
\frac{1}{4n^{2}}-r^2<\frac{1}{4n^{2}}-\frac{1}{n^{2}}=-\frac{3}{4n^{2}}<y_{1}^{2}-y^{2}
\]
and thus 
\[
\sqrt{\frac{1}{4n^{2}}-y_{1}^2}-\sqrt{r^{2}-y^2}<0.
\]
Hence we have $\re(z-\gamma_{[a,b],2n}(t))<c$ and 
\begin{flalign*}
&\hspace{0.35cm} \d (z-\gamma_{[a,b],2n}(t),[c,d]\cap\C)\\
&=|z-\gamma_{[a,b],2n}(t)-c|
 =\sqrt{\Bigl(\sqrt{\frac{1}{4n^{2}}-y_{1}^2}-\sqrt{r^{2}-y^2}\Bigr)^{2}+(y-y_{1})^{2}}\\
&=\sqrt{\frac{1}{4n^{2}}-2\sqrt{\frac{1}{4n^{2}}-y_{1}^2}\sqrt{r^{2}-y^2}+r^{2}-2yy_{1}}
\end{flalign*}
if $\re(z)<a+c$ and $|y-y_{1}|\leq\tfrac{1}{2n}$. We claim that 
\begin{equation}\label{eq:conv_standard}
-2\sqrt{\frac{1}{4n^{2}}-y_{1}^2}\sqrt{r^{2}-y^2}+r^{2}-2yy_{1}> 0,
\end{equation}
which then implies $\d (z-\gamma_{[a,b],2n}(t),[c,d]\cap\C)>\tfrac{1}{2n}$. The inequality \eqref{eq:conv_standard} 
is equivalent to 
\[
 r^{4}-4yy_{1}r^{2}+4y_{1}^{2}y^{2}
> 4\bigl(\frac{r^{2}}{4n^{2}}-\frac{y^{2}}{4n^{2}}-y_{1}^{2}r^{2}+y_{1}^{2}y^{2}\bigr),
\]
which is again equivalent to
\[
0< r^{4}-\frac{r^{2}}{n^{2}}-4yy_{1}r^{2}+\frac{y^{2}}{n^{2}}+4y_{1}^{2}r^{2}=:r^{4}-\frac{r^{2}}{n^{2}}+h(y_{1}).
\]
We note that $h'(y_{1})=8r^{2}y_{1}-4r^{2}y$, $h'(\tfrac{y}{2})=0$ and $h''(y_{1})=8r^{2}>0$, yielding 
\[
 r^{4}-\frac{r^{2}}{n^{2}}+h(y_{1})
\geq r^{4}-\frac{r^{2}}{n^{2}}+h(\tfrac{y}{2})
=r^{4}-\frac{r^{2}}{n^{2}}+\bigl(\frac{1}{n^{2}}-r^{2}\bigr)y^{2}
=\bigl(r^{2}-\frac{1}{n^{2}}\bigr)\bigl(r^{2}-y^{2}\bigr)> 0
\]
as $r>\tfrac{1}{n}$ and $|y|\leq\tfrac{1}{n}$. Thus our claim \eqref{eq:conv_standard} is true.

If $\re(z)>b+d$, then 
\[
z-\gamma_{[a,b],2n}(t)=d+b-a+\sqrt{r^{2}-y^2}+\sqrt{\frac{1}{4n^{2}}+y_{1}^2}+i(y-y_{1}),
\]
which implies $\re(z-\gamma_{[a,b],2n}(t))>d$ and 
\begin{flalign*}
&\hspace{0.35cm} \d (z-\gamma_{[a,b],2n}(t),[c,d]\cap\C)\\
&=|z-\gamma_{[a,b],2n}(t)-d|
 =\sqrt{\Bigl(b-a+\sqrt{r^{2}-y^2}+\sqrt{\frac{1}{4n^{2}}-y_{1}^2}\Bigr)^{2}+(y-y_{1})^{2}}\\
&\geq\sqrt{r^{2}-y^2+\frac{1}{4n^{2}}-y_{1}^2+y^{2}-2yy_{1}+y_{1}^{2}}\\
&=\sqrt{\frac{1}{4n^{2}}+r^{2}-2yy_{1}}> \frac{1}{2n}
\end{flalign*}
because 
\[
r^{2}-2yy_{1}> r^{2}-2\frac{1}{n}\frac{1}{2n}=r^{2}-\frac{1}{n^{2}}>0
\]
as $r>\tfrac{1}{n}$. 

Similarly, we can handle the case that $\gamma_{[a,b],2n}(t)=b+\sqrt{\frac{1}{4n^{2}}-y_{1}^2}+iy_{1}$ for some 
$|y_{1}|<\tfrac{1}{2n}$ and $t\in [0,1]$. Altogether we obtain $z-\gamma_{[a,b],2n}\subset S_{2n}([c,d])$ 
for every $z\in S_{n}([a+c,b+d])$.

Let $\alpha\in\mathfrak{A}$. We observe that 
\begin{align*}
 -\frac{1}{n}|\re(z)|
&\leq -\frac{1}{2n}|\re(z-\gamma_{[a,b],2n}(t))|
 +\frac{1}{2n}\max\Bigl(\bigl|a-\frac{1}{2n}\bigr|,\bigl|b+\frac{1}{2n}\bigr|\Bigr) \\
&=: -\frac{1}{2n}|\re(z-\gamma_{[a,b],2n}(t))|+C_{n,a,b} .
\end{align*}
Denoting by  by 
$\ell(\gamma_{[a,b],2n}):=\int_{0}^{1}|\gamma_{[a,b],2n}'(t)|\d t$ the length of the path $\gamma_{[a,b],2n}$,
we get for $z\in S_{n}([a+c,b+d])$
\begin{flalign*}
&\hspace{0.35cm} p_{\alpha}((F\circledast G)(z))e^{-\frac{1}{n}|\re(z)|}\\
&=p_{\alpha}\Bigl(\int_{\gamma_{[a,b],2n}}F(w)G(z-w)\d w\Bigr)e^{-\frac{1}{n}|\re(z)|}\\
&\underset{\mathclap{\eqref{P.0.0.11}}}{\leq}D\ell(\gamma_{[a,b],2n})
 \sup_{t\in[0,1]}p_{\beta}\bigl(F(\gamma_{[a,b],2n}(t)\bigr)
 \sup_{t\in[0,1]}p_{\omega}\bigl(G(z-\gamma_{[a,b],2n}(t))\bigr)e^{-\frac{1}{n}|\re(z)|}\\
&\leq D\ell(\gamma_{[a,b],2n})e^{C_{n,a,b}} 
 \sup_{t\in[0,1]}p_{\beta}\bigl(F(\gamma_{[a,b],2n}(t)\bigr)
 \vertiii{G}_{2n,\omega,[c,d]}.
\end{flalign*}
Hence we obtain $F\circledast G\in\mathcal{O}^{exp}(\overline{\C}\setminus [a+c,b+d],E)$.

(iv) Let $\zeta\in\h$ if $-\infty\in [a+c,b+d]$, or $\zeta\in -\h$ if $\infty\in [a+c,b+d]$, 
or $\zeta\in\C$ if $[a+c,b+d]\subset\R$, respectively.
We choose $\tfrac{1}{n}<\widetilde{c}<n$ and let $w\in\gamma_{[a,b],2n}$. We claim that 
\begin{equation}\label{eq:conv_standard_1}
\int_{\gamma_{[a+c,b+d],\widetilde{c}}}G(z-w)e^{-i(z-w)\zeta}\d z=\mathcal{F}_{[c,d]}([G])(\zeta).
\end{equation}
Due to \prettyref{thm:fourier_unbounded_interval}  
resp.\ \prettyref{thm:fourier_bounded_interval} and Cauchy's integral theorem we only need to prove that 
there are $z_{1},z_{2}\in\gamma_{[a+c,b+d],\widetilde{c}}$ such that $z_{1}-w<c$ if $c>-\infty$, 
and $z_{2}-w>d$ if $d<\infty$, which implies that $\gamma_{[a+c,b+d],\widetilde{c}}-w$ encircles $[c,d]$.
We only consider the case $c\in\R$ and $d=\infty$. The other cases are similar. 
If $w=t\pm i\tfrac{1}{2n}$ for some $t\in[a,b]$, we choose $z_{1}:=a+c+\tfrac{1}{\widetilde{c}}e^{i\theta}$ 
where $\theta\in ]\tfrac{\pi}{2},\tfrac{3\pi}{2}[$ is the unique solution of 
$\sin(\theta)=\pm\tfrac{\widetilde{c}}{2n}$, and get $z_{1}-w=c+a-t+\tfrac{1}{\widetilde{c}}\cos(\theta)<c$. 
If $w=a-\sqrt{\tfrac{1}{4n^{2}}-y^2}+iy$ or $w=b+\sqrt{\tfrac{1}{4n^{2}}-y^2}+iy$ for some $|y|<\tfrac{1}{2n}$, 
then we choose $z_{1}:=a+c-\sqrt{\tfrac{1}{\widetilde{c}^{2}}-y^2}+iy$ and get 
$z_{1}-w=c+\sqrt{\tfrac{1}{4n^{2}}-y^2}-\sqrt{\tfrac{1}{\widetilde{c}^{2}}-y^2}<c$ or 
$z_{1}-w=c+a-b-\sqrt{\tfrac{1}{4n^{2}}-y^2}-\sqrt{\tfrac{1}{\widetilde{c}^{2}}-y^2}<c$.

Moreover, we have for every $e'\in E'$ that $e'(F(w)\,\cdot)\in E_{2}'$ by \eqref{P.0.0.11}, which implies 
\begin{flalign*}
&\hspace{0.35cm}\langle e', \int_{\gamma_{[a+c,b+d],\widetilde{c}}}F(w)G(z-w)e^{-i(z-w)\zeta}\d z\rangle\\
&=\int_{\gamma_{[a+c,b+d],\widetilde{c}}}\langle e',F(w)G(z-w)e^{-i(z-w)\zeta}\rangle\d z\\
&=\langle e'(F(w)\,\cdot),\int_{\gamma_{[a+c,b+d],\widetilde{c}}}G(z-w)e^{-i(z-w)\zeta}\d z\rangle\\
&=\langle e',F(w)\int_{\gamma_{[a+c,b+d],\widetilde{c}}}G(z-w)e^{-i(z-w)\zeta}\d z\rangle
\end{flalign*}
by the definition of Pettis integrability. It follows from the Hahn-Banach theorem that  
\begin{equation}\label{eq:conv_standard_2}
 \int_{\gamma_{[a+c,b+d],\widetilde{c}}}F(w)G(z-w)e^{-i(z-w)\zeta}\d z
=F(w)\int_{\gamma_{[a+c,b+d],\widetilde{c}}}G(z-w)e^{-i(z-w)\zeta}\d z.
\end{equation}
Analogously it follows that 
\begin{equation}\label{eq:conv_standard_3}
 \int_{\gamma_{[a,b],2n}}F(w)\mathcal{F}_{[c,d]}([G])(\zeta) e^{-iw\zeta}\d w 
=\mathcal{F}_{[a,b]}([F])(\zeta)\mathcal{F}_{[c,d]}([G])(\zeta),
\end{equation}
using that $e'(\cdot\,\mathcal{F}_{[c,d]}([G])(\zeta))\in E_{1}'$ by \eqref{P.0.0.11} for every $e'\in E'$ and 
\prettyref{thm:fourier_unbounded_interval} resp.\ \prettyref{thm:fourier_bounded_interval}.

We obtain by the Fubini-Tonelli theorem and \prettyref{thm:fourier_unbounded_interval}  
resp.\ \prettyref{thm:fourier_bounded_interval} again that 
\begin{flalign*}
&\hspace{0.35cm} \mathcal{F}_{[a+c,b+d]}([F\circledast G])(\zeta)\\
&=\int_{\gamma_{[a+c,b+d],\widetilde{c}}}(F\circledast G)(z)e^{-iz\zeta}\d z\\
&=\int_{\gamma_{[a+c,b+d],\widetilde{c}}}\int_{\gamma_{[a,b],2n}}F(w)G(z-w)\d w e^{-iz\zeta}\d z\\
&\underset{\mathclap{\eqref{eq:conv_standard_2}}}{=}\int_{\gamma_{[a,b],2n}}F(w)
 \int_{\gamma_{[a+c,b+d],\widetilde{c}}}G(z-w)e^{-i(z-w)\zeta}\d z e^{-iw\zeta}\d w \\
&\underset{\mathclap{\eqref{eq:conv_standard_1}}}{=}\int_{\gamma_{[a,b],2n}}F(w)\mathcal{F}_{[c,d]}([G])(\zeta) 
 e^{-iw\zeta}\d w 
 \underset{\mathclap{\eqref{eq:conv_standard_3}}}{=}\mathcal{F}_{[a,b]}([F])(\zeta)\mathcal{F}_{[c,d]}([G])(\zeta),
\end{flalign*}
implying $[F] \circledast [G]=[F]\ast[G]$ by \prettyref{thm:convolution}. Moreover, it follows from this equation 
that $[F] \circledast [G]$ does not depend on the representing functions $F$ and $G$.
\end{proof}

\begin{rem}
Let $E$, $E_{1}$ and $E_{2}$ be sequentially complete $\C$-lcHs, $E_{2}$ strictly admissible, 
$\cdot\:\colon E_{1}\times E_{2}\to E$ a canonical bilinear map with \eqref{P.0.0.11} and $[a,b]\subset\R$. 
Then the sheaf of $E_{2}$-valued Fourier hyperfunctions is flabby by 
\cite[Theorem 5.9 b), p.\ 33]{kruse2019_5} and for $g\in bv_{\overline{\R}}(E_{2})$ there are 
$g_{1}\in bv_{[-\infty,0]}(E_{2})$ and $g_{2}\in bv_{[0,\infty]}(E_{2})$ such that $g=g_{1}+g_{2}$ 
by \cite[Lemma 1.4.4, p.\ 36]{Kan}. Hence we may define the convolution 
\[
f\circledast g:= f\circledast g_{1} + f\circledast g_{2}\in bv_{[-\infty,b]}(E)+bv_{[a,\infty]}(E).
\]
for $f\in bv_{[a,b]}(E_{1})$ by \prettyref{cor:convolution_standard}.
\end{rem}

\begin{exa}\label{ex:dirac_hyp}
Let $F(z):=-\tfrac{1}{2\pi i z}$ for $z\neq 0$. Then $F\in\mathcal{O}^{exp}(\overline{\C}\setminus\{0\})$ 
and we call $\delta_{0}:=[F]\in bv_{\{0\}}$ \emph{Dirac hyperfunction} (see \cite[Examples 1.1.5 b), p.\ 15]{Kan}). 
Due to \eqref{P.0.0.10a} and Cauchy's integral formula we get $\mathcal{F}_{\{0\}}(\delta_{0})=1$ on $\C$,
$\mathcal{F}_{[-\infty,0]}(\delta_{0})=1$ on the upper halfplane and 
$\mathcal{F}_{[0,\infty]}(\delta_{0})=1$ on the lower halfplane. 
Hence $\delta_{0}$ is the neutral element of the convolution.
\end{exa}
\section{Asymptotic Fourier transform}
\label{sect:Asymptotic_Fourier_transform}
The results of the preceding sections allow us to define an asymptotic Fourier transform on the space 
$\mathcal{B}(A,E)$ of $E$-valued hyperfunctions with support in a closed interval $A\subset\R$.

\begin{thm}\label{thm:extension}
Let $E$ be an admissible $\C$-lcHs and $a\in\R$. Then the canonical 
(restriction) maps 
\[
\mathcal{R}_{[a,\infty[}\colon \mathcal{O}^{exp}(\overline{\C}\setminus [a,\infty],E)/
\mathcal{O}^{exp}(\overline{\C}\setminus\{\infty\},E)
\to \mathcal{B}([a,\infty[,E),\;[F]\mapsto [F],
\]
and 
\[
\mathcal{R}_{]-\infty,a]}\colon \mathcal{O}^{exp}(\overline{\C}\setminus [-\infty,a],E)/
\mathcal{O}^{exp}(\overline{\C}\setminus\{-\infty\},E)
\to \mathcal{B}(]-\infty,a],E),\;[F]\mapsto [F],
\]
are linear isomorphisms.
\end{thm}
\begin{proof}
Let $F\in \mathcal{O}^{exp}(\overline{\C}\setminus [a,\infty],E)$ such that $\mathcal{R}_{[a,\infty[}([F])=0$. This means 
that $F\in\mathcal{O}(\C,E)$ and thus $F\in \mathcal{O}^{exp}(\overline{\C}\setminus \{\infty\},E)$, yielding 
$[F]=0$. Hence $\mathcal{R}_{[a,\infty[}$ is injective.

Let $[F]\in \mathcal{B}([a,\infty[,E)$ and fix $\varepsilon>0$. 
Then $F\in \mathcal{O}(\C\setminus ]a-\varepsilon,\infty[,E)$ and by \cite[5.7 Lemma, p.\ 25]{kruse2019_5} there 
is $\widetilde{F}\in \mathcal{O}^{exp}(\overline{\C}\setminus [a-\varepsilon,\infty],E)$ such that 
$\widetilde{F}-F\in\mathcal{O}(\C,E)$, more precisely, $\widetilde{F}-F$ extends to a function $G\in\mathcal{O}(\C,E)$. 
Now, we extend $\widetilde{F}$ by $\widetilde{F}(z):=G(z)+F(z)$ on $[a-\varepsilon,a[$ and obtain 
$\widetilde{F}\in\mathcal{O}^{exp}(\overline{\C}\setminus [a,\infty],E)$ with 
$\mathcal{R}_{[a,\infty[}([\widetilde{F}])=[F]$ by this extension. Thus $\mathcal{R}_{[a,\infty[}$ is surjective. 
The proof for $\mathcal{R}_{]-\infty,a]}$ is analogous.
\end{proof}

\begin{thm}\label{thm:asymptotic_one_sided_fourier}
Let $E$ be an admissible sequentially complete $\C$-lcHs and $A:=]-\infty,a]$ or $A:=[a,\infty[$ for some $a\in\R$. 
Then the asymptotic one-sided Fourier transform
\[
\mathcal{F}_{A}^{\mathcal{B}}\colon 
\mathcal{B}(A,E)\to 
\mathcal{FO}_{\overline{A}}(E)/\mathcal{FO}_{\overline{A}\setminus A}(E),\;
\mathcal{F}_{A}^{\mathcal{B}}(f):=[(\mathcal{F}\circ \mathcal{R}_{A}^{-1})(f)],
\]
is a linear isomorphism, where the closure $\overline{A}$ is taken in $\overline{\R}$.
\end{thm}
\begin{proof}
This follows directly from \prettyref{thm:extension} and \prettyref{thm:fourier_unbounded_interval}. 
\end{proof}

\begin{thm}\label{thm:extension_real_comp}
Let $E$ be a locally complete $\C$-lcHs and $K\subset\R$ a non-empty compact set. Then the canonical 
(restriction) map
\[
\mathcal{R}_{K}\colon bv_{K}(E)\to \mathcal{B}(K,E),\;[F]\mapsto [F],
\]
is a topological isomorphism and its inverse $\mathcal{R}_{K}^{-1}\colon \mathcal{B}(K,E)\to bv_{K}(E)$ 
is given by $\mathcal{R}_{K}^{-1}([F])=[\Psi_{K}([F])]$ where 
\[
\Psi_{K}([F])\colon \C\setminus K \to E,\; 
\Psi_{K}([F])(z):=\frac{i}{2\pi}\Bigl\langle \mathscr{H}_{K}([F]),\frac{e^{-(z-\cdot)^2}}{z-\cdot}\Bigr\rangle .
\]
\end{thm}
\begin{proof}
The statement is a consequence of \prettyref{thm:duality} and \prettyref{cor:duality_real_comp}.
\end{proof}

\begin{cor}\label{cor:fourier_hyp_real_comp}
Let $E$ be a locally complete $\C$-lcHs and $-\infty<a\leq b<\infty$. Then the Fourier transform
\[
\mathcal{F}_{[a,b]}^{\mathcal{B}}\colon \mathcal{B}([a,b],E)\to \mathcal{FO}_{[a,b]}(E),\;
\mathcal{F}_{[a,b]}^{\mathcal{B}}([F]):=(\mathcal{F}_{[a,b]}\circ \mathcal{R}_{[a,b]}^{-1})([F]),
\]
is a topological isomorphism and 
\begin{equation}\label{eq:fourier_hyp_real_comp}
 \mathcal{F}_{[a,b]}^{\mathcal{B}}([F])(\zeta)
=\mathcal{F}_{[a,b]}([\Psi([F])])(\zeta)=\langle \mathscr{H}_{[a,b]}([F]),e^{-i(\cdot)\zeta}\rangle,\;\zeta\in\C .
\end{equation}
\end{cor}
\begin{proof}
This follows from \prettyref{thm:fourier_bounded_interval} and \prettyref{thm:extension_real_comp}.
\end{proof}

\begin{prop}\label{prop:fourier_decomposition}
Let $E$ be a strictly admissible sequentially complete $\C$-lcHs and $-\infty<a<b<\infty$. Then the following holds:
\begin{enumerate}
\item[a)] For every 
$f\in \mathcal{B}([a,\infty[,E)$ there are $f_{1}\in \mathcal{B}([a,b],E)$ and $f_{2}\in \mathcal{B}([b,\infty[,E)$ 
such that $f=f_{1}+f_{2}$ and 
\[
 \mathcal{F}_{[a,\infty[}^{\mathcal{B}}(f)
=[\mathcal{F}_{[a,b]}^{\mathcal{B}}(f_{1})_{\mid\im <0}]+\mathcal{F}_{[b,\infty[}^{\mathcal{B}}(f_{2}).
\]
\item[b)] For every 
$f\in \mathcal{B}(]-\infty,b],E)$ there are $f_{1}\in \mathcal{B}([a,b],E)$ and $f_{2}\in \mathcal{B}(]-\infty,a],E)$ 
such that $f=f_{1}+f_{2}$ and 
\[
 \mathcal{F}_{]-\infty,b]}^{\mathcal{B}}(f)
=[\mathcal{F}_{[a,b]}^{\mathcal{B}}(f_{1})_{\mid\im >0}]+\mathcal{F}_{]-\infty,a]}^{\mathcal{B}}(f_{2}).
\]
\end{enumerate}
\end{prop}
\begin{proof}
Let us start with part a). Since $E$ is strictly admissible, the sheaf of $E$-valued hyperfunctions is flabby 
by \cite[Theorem 5.9 b), p.\ 33]{kruse2019_5} and thus for every 
$f\in \mathcal{B}([a,\infty[,E)$ there are $f_{1}\in \mathcal{B}([a,b],E)$ and $f_{2}\in \mathcal{B}([b,\infty[,E)$ 
such that $f=f_{1}+f_{2}$ by \cite[Lemma 1.4.4, p.\ 36]{Kan}. Moreover, we have 
$\mathscr{H}_{[a,\infty]}(f_{1})=\mathscr{H}_{[a,b]}(f_{1})$ and 
$\mathscr{H}_{[a,\infty]}(f_{2})=\mathscr{H}_{[b,\infty]}(f_{2})$ due to \cite[Eq.\ (6), p.\ 14]{kruse2019_2}, 
which implies 
\[
\mathcal{F}_{[a,\infty[}^{\mathcal{B}}(f_{1})=[\mathcal{F}_{[a,b]}^{\mathcal{B}}(f_{1})_{\mid\im <0}]
 \quad\text{and}\quad
 \mathcal{F}_{[a,\infty[}^{\mathcal{B}}(f_{2})=\mathcal{F}_{[b,\infty[}^{\mathcal{B}}(f_{2})
\]
by \prettyref{eq:fourier_hyp_real_comp} and the definition of the asymptotic Fourier transforms. We deduce that
\[
 \mathcal{F}_{[a,\infty[}^{\mathcal{B}}(f)
=\mathcal{F}_{[a,\infty[}^{\mathcal{B}}(f_{1})+\mathcal{F}_{[a,\infty[}^{\mathcal{B}}(f_{2})
=[\mathcal{F}_{[a,b]}^{\mathcal{B}}(f_{1})_{\mid\im <0}]+\mathcal{F}_{[b,\infty[}^{\mathcal{B}}(f_{2}).
\]
The proof of part b) is analogous.
\end{proof}

\begin{thm}\label{thm:asymptotic_fourier}
Let $E$ be an admissible sequentially complete $\C$-lcHs. Then the canonical (restriction) map 
\[
\mathcal{R}_{\R}\colon\mathcal{O}^{exp}(\overline{\C}\setminus\overline{\R},E)/
\mathcal{O}^{exp}(\overline{\C}\setminus\{\pm\infty\},E)\to \mathcal{B}(\R,E),\;[F]\mapsto [F],
\]
and the asymptotic Fourier transform 
\begin{gather*}
\mathcal{F}^{\mathcal{B}}_{\R}\colon\mathcal{B}(\R,E)\to\mathcal{O}^{exp}(\overline{\C}\setminus\overline{\R},E)/
\bigl(\mathcal{O}^{exp}(\overline{\C},E)\oplus \mathcal{FO}_{\{\pm\infty\}}(E)\bigr),\\
\mathcal{F}^{\mathcal{B}}_{\R}(f):=[(\mathcal{F}_{\overline{\R}}\circ \mathcal{R}_{\R}^{-1})(f)],
\end{gather*}
are linear isomorphisms.
\end{thm}
\begin{proof}
(i) We consider $\mathcal{R}_{\R}$. Let $F\in\mathcal{O}^{exp}(\overline{\C}\setminus\overline{\R},E)$ 
such that $\mathcal{R}_{\R}([F])=0$, i.e.\ $F\in\mathcal{O}(\C,E)$. 
We deduce that $F\in \mathcal{O}^{exp}(\overline{\C}\setminus \{\pm\infty\},E)$, resulting in 
$[F]=0$. Thus $\mathcal{R}_{\R}$ is injective.

Let $[F]\in \mathcal{B}(\R,E)$. 
Then $F\in \mathcal{O}(\C\setminus\R,E)$ and by \cite[5.7 Lemma, p.\ 25]{kruse2019_5} there 
is $\widetilde{F}\in \mathcal{O}^{exp}(\overline{\C}\setminus\overline{\R},E)$ such that 
$\widetilde{F}-F\in\mathcal{O}(\C,E)$. It follows that $\mathcal{R}_{\R}([\widetilde{F}])=[F]$, meaning that 
$\mathcal{R}_{\R}$ is surjective.

(ii) Next, we prove $\mathcal{O}^{exp}(\overline{\C},E)\cap\mathcal{FO}_{\{\pm\infty\}}(E)=\{0\}$. 
The proof is similar to the one of part (a) of the proof of \cite[Proposition 3.3, p.\ 50]{langenbruch2011}.
Let $f\in\mathcal{O}^{exp}(\overline{\C},E)\cap\mathcal{FO}_{\{\pm\infty\}}(E)$. This implies 
\begin{flalign*}
&\hspace{0.35cm}\sup_{z\in\C}p_{\alpha}(f(z))e^{-\frac{1}{k}|z|+k|\im(z)|}\\
&\leq |f|_{k,\alpha,\{-\infty\}}+ |f|_{k,\alpha,\{\infty\}}
 +\sup_{|\im(z)|<\frac{1}{k}}p_{\alpha}(f(z))e^{-\frac{1}{k}|z|+k|\im(z)|}\\
&\leq |f|_{k,\alpha,\{-\infty\}}+ |f|_{k,\alpha,\{\infty\}}+ e \vertiii{f}_{k,\alpha,\emptyset}
\end{flalign*}
for every $\alpha\in\mathfrak{A}$ and $k\in\N$. Hence we have 
\[
\forall\;\alpha\in\mathfrak{A},\,k\in\N\;\exists\;C>0\;\forall\;z\in\C:\;
 p_{\alpha}(f(iz))\leq C e^{\frac{1}{k}|z|-k|\re(z)|}.
\] 
Let $e'\in E'$. Due to the Paley-Wiener theorem (see \cite[19.3 Theorem, p.\ 375]{rudin1986}) 
applied to $e'(f(i\cdot))$ there is $F_{e',k}\in L^{2}(-\tfrac{1}{k},\tfrac{1}{k})$ such that 
\[
(e'\circ f)(iz)=\int_{-\frac{1}{k}}^{\frac{1}{k}}F_{e',k}(t)e^{itz}\d t,\;z\in\C,
\]
for every $k\in\N$. We derive that $e'\circ f=0$ and the Hahn-Banach theorem yields $f=0$.

(iii) Let $f\in\mathcal{B}(\R,E)$ such that $\mathcal{F}^{\mathcal{B}}_{\R}(f)=0
=\mathcal{O}^{exp}(\overline{\C},E)\oplus \mathcal{FO}_{\{\pm\infty\}}(E)$. 
Let $F\in\mathcal{F}^{\mathcal{B}}_{\R}(f)$. Then there are $F_{1}\in\mathcal{O}^{exp}(\overline{\C},E)$ and 
$F_{2}\in\mathcal{FO}_{\{\pm\infty\}}(E)$ such that $F=F_{1}+F_{2}$. 
Due to \prettyref{thm:fourier_union_unbounded_intervals} there is $g\in bv_{\{\pm\infty\}}(E)$ such 
that $F_{2}=\mathcal{F}_{\{\pm\infty\}}(g)$. Taking equivalence classes in $bv_{\overline{\R}}(E)$, we get 
\begin{align*}
 (\mathscr{H}_{\overline{\R}}^{-1}\circ\mathcal{F}_{\star}\circ \mathscr{H}_{\overline{\R}})(\mathcal{R}_{\R}^{-1}(f))
&=\mathcal{F}_{\overline{\R}}(\mathcal{R}_{\R}^{-1}(f))
 =[F_{1}+F_{2}]
 =[F_{2}]
 =[\mathcal{F}_{\{\pm\infty\}}(g)]\\
&\underset{\mathclap{\eqref{eq:fourier_union_unbounded_intervals}}}{=}
(\mathscr{H}_{\overline{\R}}^{-1}\circ\mathcal{F}_{\star})\bigl(\mathscr{H}_{\{-\infty\}}(g)
 -\mathscr{H}_{\{\infty\}}(g)\bigr),
\end{align*}
which implies 
\[
 \mathscr{H}_{\overline{\R}}(\mathcal{R}_{\R}^{-1}(f))
=\mathscr{H}_{\{-\infty\}}(g)-\mathscr{H}_{\{\infty\}}(g)\in L(\mathcal{P}(\{\pm\infty\}),E).
\]
By \prettyref{thm:duality} and \cite[Eq.\ (6), p.\ 14]{kruse2019_2} there is $h\in bv_{\{\pm\infty\}}(E)$
such that  
\[
\mathscr{H}_{\{\pm\infty\}}(h)=\mathscr{H}_{\{-\infty\}}(g)-\mathscr{H}_{\{\infty\}}(g)
 \quad\text{and}\quad 
\mathscr{H}_{\{\pm\infty\}}(h)
=\mathscr{H}_{\overline{\R}}(h),
\]
yielding $\mathscr{H}_{\overline{\R}}(\mathcal{R}_{\R}^{-1}(f))=\mathscr{H}_{\overline{\R}}(h)$.
From the injectivity of $\mathscr{H}_{\overline{\R}}$ on $bv_{\overline{\R}}(E)$ it follows 
that $\mathcal{R}_{\R}^{-1}(f)=h$. We derive from part (i) that $f=0$, i.e.\ $\mathcal{F}^{\mathcal{B}}_{\R}$ 
is injective. 

Let $F\in\mathcal{O}^{exp}(\overline{\C}\setminus\overline{\R},E)$. 
Since $\mathcal{F}_{\overline{\R}}\colon bv_{\overline{\R}}(E)\to bv_{\overline{\R}}(E)$ is 
surjective by \prettyref{cor:fourier_isom_extended_reals}, there is $g=[G]\in bv_{\overline{\R}}(E)$ such that 
$\mathcal{F}_{\overline{\R}}(g)=F+\mathcal{O}^{exp}(\overline{\C},E)$. 
From the surjectivity of $\mathcal{R}_{\R}^{-1}$ we deduce that there is $h\in\mathcal{B}(\R,E)$ such that 
$H-G\in\mathcal{O}^{exp}(\overline{\C}\setminus\{\pm\infty\},E)$ for $[H]=\mathcal{R}^{-1}_{\R}(h)$. 
We note that by \cite[Eq.\ (6), p.\ 14]{kruse2019_2}
\[
 \mathscr{H}_{\overline{\R}}([H-G])
=\mathscr{H}_{\{\pm\infty\}}([H-G])
=\mathscr{H}_{\{-\infty\}}([H-G])+\mathscr{H}_{\{\infty\}}([H-G])
\]
and so
\begin{align*}
 \mathcal{F}_{\overline{\R}}([H-G])
&=(\mathscr{H}_{\overline{\R}}^{-1}\circ\mathcal{F}_{\star}\circ \mathscr{H}_{\overline{\R}})([H-G])\\
&=(\mathscr{H}_{\overline{\R}}^{-1}\circ\mathcal{F}_{\star}\circ \mathscr{H}_{\{-\infty\}})([H-G])
  +(\mathscr{H}_{\overline{\R}}^{-1}\circ\mathcal{F}_{\star}\circ \mathscr{H}_{\{\infty\}})([H-G])\\
&\underset{\mathclap{\eqref{eq:compat_fourier_union}}}{=}
 [\mathcal{F}_{\{-\infty\}}([H-G])-\mathcal{F}_{\{\infty\}}([H-G])] 
 \underset{\eqref{eq:fourier_union_unbounded_intervals}}{=}[\mathcal{F}_{\{\pm\infty\}}([H-G])].
\end{align*}
Thus there is $P\in\mathcal{FO}_{\{\pm\infty\}}(E)$ such that 
$\mathcal{F}_{\overline{\R}}([H-G])=P+\mathcal{O}^{exp}(\overline{\C},E)$, which guarantees that our map is 
well-defined as well. We conclude that
\[
 \mathcal{F}_{\overline{\R}}(\mathcal{R}_{\R}^{-1}(h))
=\mathcal{F}_{\overline{\R}}([G]+[H-G])
 =\mathcal{F}_{\overline{\R}}([G])+\mathcal{F}_{\overline{\R}}([H-G])\\
=F+P+\mathcal{O}^{exp}(\overline{\C},E), 
\]
which means 
\[
\mathcal{F}_{\R}^{\mathcal{B}}(h)=[\mathcal{F}_{\overline{\R}}(\mathcal{R}_{\R}^{-1}(h))]=[F]
\]
where the equivalence class is taken in $\mathcal{O}^{exp}(\overline{\C}\setminus\overline{\R},E)/
\bigl(\mathcal{O}^{exp}(\overline{\C},E)\oplus \mathcal{FO}_{\{\pm\infty\}}(E)\bigr)$.
This proves that $\mathcal{F}_{\R}^{\mathcal{B}}$ is surjective.
\end{proof}

Now, we may transfer the connections of some standard operations and the asymptotic Fourier transform 
from the preceding section. 

\begin{prop}\label{prop:asymp_fourier_transform_of_shift}
Let $E$ be an admissible sequentially complete $\C$-lcHs, $A:=]-\infty,a]$ or $A:=[a,\infty[$ for some $a\in\R$. 
Let $\tau_{h}([F]):=[F(\cdot -h)]$, $h\in\R$, be the shift operator from 
$\mathcal{B}(A,E)$ to $\mathcal{B}(h+A,E)$. 
Then we have for for $[F]\in\mathcal{B}(A,E)$
\[
\mathcal{F}_{h+A}^{\mathcal{B}}(\tau_{h}([F]))=e^{-ih(\cdot)}\mathcal{F}_{A}^{\mathcal{B}}([F]).
\]
\end{prop}
\begin{proof}
This is a consequence of \prettyref{prop:fourier_transform_of_shift} and 
$\mathcal{R}_{A}(\tau_{h}\mathcal{R}_{A}^{-1}([F]))=\tau_{h}([F])$, which implies 
$\tau_{h}(\mathcal{R}_{A}^{-1}([F]))=\mathcal{R}_{A}^{-1}(\tau_{h}([F]))$.
\end{proof}

\begin{prop}\label{prop:asymp_fourier_transform_of_derivative}
Let $E$ be an admissible sequentially complete $\C$-lcHs, $A:=]-\infty,a]$ or $A:=[a,\infty[$ for some $a\in\R$ and 
$P(-i\partial):=\sum_{k=0}^{\infty}\tfrac{c_{k}}{k!}(-i\partial)^{k}$ where $(c_{k})\subset\C$ and 
$P$ is of exponential type $0$.
Then we have for $f=[F]\in\mathcal{B}(A,E)$
\[
\mathcal{F}_{A}^{\mathcal{B}}(P(-i\partial)f)=P\mathcal{F}_{A}^{\mathcal{B}}(f)
 \quad\text{and}\quad
\mathcal{F}_{A}^{\mathcal{B}}(Pf)=P(i\partial)\mathcal{F}_{A}^{\mathcal{B}}(f)
\]
where $P(-i\partial)f:=[P(-i\partial)F]$ and $Pf:=[PF]$.
\end{prop}
\begin{proof}
First, we remark that $F\mapsto P(-i\partial)F$ and $F\mapsto PF$ are well-defined continuous linear operators on 
$\mathcal{O}(\C\setminus A,E)$ and $\mathcal{O}(\C,E)$ as in \prettyref{prop:fourier_transform_of_derivative} 
when these spaces are equipped with the topology of uniform convergence on compact subsets. 
This implies that $P(-i\partial)f$ and $Pf$ are well-defined for $f\in\mathcal{B}(A,E)$.

The first equation follows from \eqref{eq:fourier_transform_of_derivative} because 
$\mathcal{R}_{A}(P(-i\partial)\mathcal{R}_{A}^{-1}f)=P(-i\partial)f$ and hence
$P(-i\partial)\mathcal{R}_{A}^{-1}f=\mathcal{R}_{A}^{-1}P(-i\partial)f$ for $f\in\mathcal{B}(A,E)$ 
(like in \cite[Example 3.8 (b), p.\ 52]{langenbruch2011}).

The second equation follows from \eqref{eq:fourier_transform_of_multiplication} since 
$\mathcal{R}_{A}(P\mathcal{R}_{A}^{-1}f)=Pf$ and thus $P\mathcal{R}_{A}^{-1}f=\mathcal{R}_{A}^{-1}(Pf)$ 
(like in \cite[Proposition 3.10, p.\ 53]{langenbruch2011}). 
\end{proof}

\begin{thm}\label{thm:asymp_fourier_convolution}
Let $(E,(p_{\alpha})_{\alpha\in\mathfrak{A}})$, $(E_{1},(p_{\beta})_{\beta\in\mathfrak{B}})$ 
and $(E_{2},(p_{\omega})_{\omega\in\Omega})$ be admissible sequentially complete $\C$-lcHs such that 
a canonical bilinear map $\cdot\:\colon E_{1}\times E_{2}\to E$ is defined with the property
\[
\forall\; \alpha\in\mathfrak{A}\;\exists\;\beta\in\mathfrak{B},\,\omega\in\Omega,\,D>0\;\forall\;x\in E_{1},\,y\in E_{2}:
\;p_{\alpha}(x\cdot y)\leq D p_{\beta}(x)p_{\omega}(y).
\]
Let $a,b,c,d\in\R$.
\begin{enumerate}
\item[a)] If $f\in\mathcal{FO}_{\{\infty\}}(E_{1})$ and $g\in\mathcal{FO}_{[c,\infty]}(E_{2})$, or if
$f\in\mathcal{FO}_{[a,\infty]}(E_{1})$ and $g\in\mathcal{FO}_{\{\infty\}}(E_{2})$, 
then $fg\in\mathcal{FO}_{\{\infty\}}(E)$.
\item[b)] If $f\in\mathcal{FO}_{\{-\infty\}}(E_{1})$ and $g\in\mathcal{FO}_{[-\infty,d]}(E_{2})$, or if
$f\in\mathcal{FO}_{[-\infty,b]}(E_{1})$ and $g\in\mathcal{FO}_{\{-\infty\}}(E_{2})$, 
then $fg\in\mathcal{FO}_{\{-\infty\}}(E)$.
\item[c)] We define the convolution
\begin{gather*}
\ast^{\mathcal{B}}\:\colon \mathcal{B}([a,\infty[,E_{1})\times \mathcal{B}([c,\infty[,E_{2})
\to\mathcal{B}([a+c,\infty[,E),\\
f\ast^{\mathcal{B}} g:=(\mathcal{F}^{\mathcal{B}})^{-1}
\bigl(\mathcal{F}^{\mathcal{B}}(f)\mathcal{F}^{\mathcal{B}}(g)\bigr),
\end{gather*}
and 
\begin{gather*}
\ast^{\mathcal{B}}\:\colon \mathcal{B}(]-\infty,b],E_{1})\times \mathcal{B}(]-\infty,d],E_{2})
\to\mathcal{B}(]-\infty,b+d],E),\\
f\ast^{\mathcal{B}} g:=(\mathcal{F}^{\mathcal{B}})^{-1}
\bigl(\mathcal{F}^{\mathcal{B}}(f)\mathcal{F}^{\mathcal{B}}(g)\bigr).
\end{gather*}
The convolution is well-defined, bilinear and continuous.
\item[d)] If $E_{1}=E_{2}$ and $\cdot\:\colon E_{1}\times E_{2}\to E$ is commutative, 
then the convolution is commutative as well. 
If $E=E_{1}=E_{2}=\C$ and $\cdot\:\colon E_{1}\times E_{2}\to E$ is the multiplication, 
then the convolution is associative.
\end{enumerate}
\end{thm}
\begin{proof}
The parts a) and b) are just special cases of \prettyref{thm:convolution} a). In particular, this implies 
that for $\mathcal{F}^{\mathcal{B}}(f)=[\widetilde{F}]$ and $\mathcal{F}^{\mathcal{B}}(g)=[\widetilde{G}]$ the product 
$\mathcal{F}^{\mathcal{B}}(f)\mathcal{F}^{\mathcal{B}}(g):=[\widetilde{F}\cdot\widetilde{G}]$ is well-defined.
The convolutions in c) are well-defined due to a) and b) and \prettyref{thm:convolution} a). The rest follows 
from \prettyref{thm:convolution} b) and c).
\end{proof}

Let $E$, $E_{1}$ and $E_{2}$ and $\cdot\:\colon E_{1}\times E_{2}\to E$ be as above and $a,b,c,d\in\R$.
We note that the convolution $f\circledast g$ from \prettyref{cor:convolution_standard} is even defined 
and an element of $\mathcal{B}([a+c,\infty[,E)$ resp.\ $\mathcal{B}(]-\infty,b+d],E)$ by part (i) and (ii) of the 
proof of \prettyref{cor:convolution_standard} if $f\in \mathcal{B}([a,b],E_{1})$ and 
$g\in \mathcal{B}([c,\infty[,E_{2})$ resp.\ $g\in \mathcal{B}(]-\infty,d],E_{2})$. If $E=E_{1}=E_{2}=\C$, then 
this is the usual definition of the convolution of two hyperfunctions where one of them has real compact support 
(see \cite[Proposition 2.53, p.\ 141]{graf2010}). 

\begin{cor}\label{cor:asymp_fourier_convolution}
Let $(E,(p_{\alpha})_{\alpha\in\mathfrak{A}})$, $(E_{1},(p_{\beta})_{\beta\in\mathfrak{B}})$ 
and $(E_{2},(p_{\omega})_{\omega\in\Omega})$ be admissible sequentially complete $\C$-lcHs, 
$E$ and $E_{1}$ strictly admissible, such that
a canonical bilinear map $\cdot\:\colon E_{1}\times E_{2}\to E$ is defined with the property
\[
\forall\; \alpha\in\mathfrak{A}\;\exists\;\beta\in\mathfrak{B},\,\omega\in\Omega,\,D>0\;\forall\;x\in E_{1},\,y\in E_{2}:
\;p_{\alpha}(x\cdot y)\leq D p_{\beta}(x)p_{\omega}(y).
\]
Let $a,b,c,d\in\R$. 
\begin{enumerate}
\item[a)] For every $f\in\mathcal{B}([a,\infty[,E_{1})$ and $j>a$ there is $f_{j}\in\mathcal{B}([a,j],E_{1})$ 
such that $f-f_{j}\in \mathcal{B}([j,\infty[,E_{1})$.
\item[b)] For every $f\in\mathcal{B}(]-\infty,b],E_{1})$ and $j<b$ there is 
$\widetilde{f}_{j}\in\mathcal{B}([j,b],E_{1})$ such that $f-\widetilde{f}_{j}\in \mathcal{B}(]-\infty,j],E_{1})$.
\item[c)] We define the convolutions
\[
\circledast^{\mathcal{B}}\:\colon \mathcal{B}([a,\infty[,E_{1})\times \mathcal{B}([c,\infty[,E_{2})
\to\mathcal{B}([a+c,\infty[,E)
\]
by $(f\circledast^{\mathcal{B}} g)_{\big{|}]-\infty,j[}:=(f_{j}\circledast g)_{\big{|} ]-\infty,j[}$ for 
every $j>a$, and 
\[
\circledast^{\mathcal{B}}\:\colon \mathcal{B}(]-\infty,b],E_{1})\times \mathcal{B}(]-\infty,d],E_{2})
\to\mathcal{B}(]-\infty,b+d],E)
\]
by $(f\circledast^{\mathcal{B}} g)_{\big{|}]j,\infty[}
:=(\widetilde{f}_{j}\circledast g)_{\big{|} ]j,\infty[}$
for every $j<b$. Then we have
\[
f\circledast^{\mathcal{B}} g=f\ast^{\mathcal{B}} g 
 \quad\text{and}\quad
\mathcal{F}^{\mathcal{B}}(f\circledast^{\mathcal{B}} g)=\mathcal{F}^{\mathcal{B}}(f)\mathcal{F}^{\mathcal{B}}(g).
\]
\end{enumerate}
\end{cor}
\begin{proof}
a) Since $E_{1}$ is strictly admissible, the sheaf of $E_{1}$-valued hyperfunctions is flabby by 
\cite[Theorem 5.9 b), p.\ 33]{kruse2019_5} and for $f\in\mathcal{B}([a,\infty[,E_{1})$ and $j>a$ there is
$f_{j}\in\mathcal{B}([a,j],E_{1})$ such that $f-f_{j}\in \mathcal{B}([j,\infty[,E_{1})$
due to \cite[Lemma 1.4.4, p.\ 36]{Kan}. The proof of part b) is analogous.

c) Let $f\in\mathcal{B}([a,\infty[,E_{1})$ and $g\in\mathcal{B}([c,\infty[,E_{2})$.
It follows from the glueing property of a sheaf (see \cite[Property (S2), p.\ 6]{bre}) that by setting 
$(f\circledast^{\mathcal{B}} g)_{\big{|}]-\infty,j[}:=(f_{j}\circledast g)_{\big{|} ]-\infty,j[}$ for 
every $j>a$, we get a hyperfunction $f\circledast^{\mathcal{B}} g\in \mathcal{B}([a+c,\infty[,E)$. 
We deduce from \prettyref{thm:asymptotic_one_sided_fourier}, \prettyref{thm:convolution} 
and \prettyref{cor:convolution_standard} that 
\begin{flalign*}
&\hspace{0.35cm} \mathcal{F}^{\mathcal{B}}(f\circledast^{\mathcal{B}} g)
 -\mathcal{F}^{\mathcal{B}}(f)\mathcal{F}^{\mathcal{B}}(g)\\
&=\mathcal{F}^{\mathcal{B}}((f-f_{j})\circledast^{\mathcal{B}} g)
 +\mathcal{F}^{\mathcal{B}}(f_{j}\circledast^{\mathcal{B}} g)
 -\mathcal{F}^{\mathcal{B}}(f)\mathcal{F}^{\mathcal{B}}(g)\\
&=\mathcal{F}^{\mathcal{B}}((f-f_{j})\circledast^{\mathcal{B}} g)
 +\mathcal{F}(\mathcal{R}_{[a,j]}^{-1}(f_{j})\circledast \mathcal{R}_{[c,\infty[}^{-1}(g))
 -\mathcal{F}^{\mathcal{B}}(f)\mathcal{F}^{\mathcal{B}}(g)\\
&=\mathcal{F}^{\mathcal{B}}((f-f_{j})\circledast^{\mathcal{B}} g)
 +\mathcal{F}^{\mathcal{B}}(f_{j})\mathcal{F}^{\mathcal{B}}(g)
 -\mathcal{F}^{\mathcal{B}}(f)\mathcal{F}^{\mathcal{B}}(g)\\
&=\mathcal{F}^{\mathcal{B}}((f-f_{j})\circledast^{\mathcal{B}} g)
 +\mathcal{F}^{\mathcal{B}}(f_{j}-f)\mathcal{F}^{\mathcal{B}}(g).
\end{flalign*}
We note that $(f-f_{j})\circledast^{\mathcal{B}} g\in\mathcal{B}([c+j,\infty[,E)$, 
$f_{j}-f\in \mathcal{B}([j,\infty[,E_{1})$ by a) and $g\in\mathcal{B}([c,\infty[,E_{2})$, 
which implies that 
$\mathcal{F}^{\mathcal{B}}((f-f_{j})\circledast^{\mathcal{B}} g)\in 
\mathcal{FO}_{[c+j,\infty]}(E)/\mathcal{FO}_{\{\infty\}}(E)$, 
$\mathcal{F}^{\mathcal{B}}(f_{j}-f)\in \mathcal{FO}_{[j,\infty]}(E_{1})/\mathcal{FO}_{\{\infty\}}(E_{1})$ 
and $\mathcal{F}^{\mathcal{B}}(g)\in \mathcal{FO}_{[c,\infty]}(E_{2})/\mathcal{FO}_{\{\infty\}}(E_{2})$ 
by \prettyref{thm:asymptotic_one_sided_fourier}. We derive from \prettyref{thm:convolution} a) that 
\[
\mathcal{F}^{\mathcal{B}}(f_{j}-f)\mathcal{F}^{\mathcal{B}}(g)
\in \mathcal{FO}_{[c+j,\infty]}(E)/\mathcal{FO}_{\{\infty\}}(E).
\]
We conclude that $\mathcal{F}^{\mathcal{B}}(f\circledast^{\mathcal{B}} g)
-\mathcal{F}^{\mathcal{B}}(f)\mathcal{F}^{\mathcal{B}}(g)\in\mathcal{FO}_{[c+j,\infty]}(E)/\mathcal{FO}_{\{\infty\}}(E)$
for every $j>a$, yielding
\[
 \mathcal{F}^{\mathcal{B}}(f\circledast^{\mathcal{B}} g)
=\mathcal{F}^{\mathcal{B}}(f)\mathcal{F}^{\mathcal{B}}(g).
\]
It follows from \prettyref{thm:asymp_fourier_convolution} that $f\circledast^{\mathcal{B}} g=f\ast^{\mathcal{B}} g $. 
The other case is analogous.
\end{proof}

The idea of the proof above is a modification of the proof of \cite[Theorem 5.5, p.\ 56]{langenbruch2011}.
\section{Langenbruch's asymptotic Fourier transform}
\label{sect:Asymptotic_Fourier_transform_Lan}
In this section we study the relation of our asymptotic Fourier transform on $\mathcal{B}(\R)$ 
to the one of Langenbruch. We briefly recall the relevant notions and results from \cite{langenbruch2011} 
(and the introduction).

\begin{defn}[{\cite[p.\ 44, 49, 53, 54, 55, 61]{langenbruch2011}}]\label{def:test_functions_langenbruch}
Let $E$ be a $\C$-lcHs and $K\subset\overline{\R}$ be compact. We define the space
\[
\mathcal{H}_{-\infty}(\overline{\C}\setminus K,E):= \{f\in\mathcal{O}(\C\setminus K,E)
\; | \; \forall\;n\in\N,\,\alpha\in\mathfrak{A}:\;\|f\|_{n,\alpha,K}^{\mathcal{H}_{-\infty}} < \infty\}
\]
where
\[
\|f\|_{n,\alpha,K}^{\mathcal{H}_{-\infty}}:=\sup_{z\in S_{n}(K)}
p_{\alpha}(f(z))e^{n|\re(z)|},
\]
and the quotient space
\[
\mathcal{G}(K,E):=\mathcal{H}_{-\infty}(\overline{\C}\setminus K,E)/
\mathcal{H}_{-\infty}(\overline{\C},E).
\]
\end{defn}

We remark that our definition of the spaces above coincides with Langenbruch's original definition because 
$S_{n}(K)=W_{n}(K)$ for $K=\emptyset$, $[0,\infty]$, $S_{n}(\overline{\R})=F_{n}$ with $F_{n}$ from 
\cite[p.\ 44]{langenbruch2011} and 
$\mathcal{H}_{-\infty}(\overline{\C}\setminus\{\pm\infty\},E)=
\mathcal{H}_{-\infty}(\overline{\C}\setminus\overline{\R},E)\cap\mathcal{O}(\C,E)$
and
$\mathcal{H}_{-\infty}(\overline{\C}\setminus\{\infty\},E)=
\mathcal{H}_{-\infty}(\overline{\C}\setminus [0,\infty],E)\cap\mathcal{O}(\C,E)$.

\begin{defn}[{\cite[p.\ 46, 50]{langenbruch2011}}]
Let $E$ be a $\C$-lcHs. We define the spaces 
\[
\mathcal{H}_{1}^{\pm}(E):=\{f\in\mathcal{O}(\C,E)\;|\;\forall\;j\in\N,\,\alpha\in\mathfrak{A}:\;
|f|_{j,\alpha}^{\pm}<\infty\}
\]
where 
\[
|f|_{j,\alpha}^{\pm}:=\sup_{\pm\im(z)\leq j}p_{\alpha}(f(z))e^{-\frac{1}{j}|\re(z)|+j|\im(z)|},
\]
and
\[
\mathfrak{F}\mathcal{G}(\overline{\R},E):=\mathcal{O}^{exp}(\overline{\C},E)
\quad\text{and}\quad
\mathfrak{F}\mathcal{G}(\{\pm\infty\},E):=\mathcal{H}_{1}^{-}(E)\oplus\mathcal{H}_{1}^{+}(E).
\]
\end{defn}

Let $K=\overline{\R}$ or $K=\{\pm\infty\}$. By 
\cite[Theorem 2.3, p.\ 47, Proposition 3.3, p.\ 50]{langenbruch2011} the Fourier transform 
\[
\mathfrak{F}\colon \mathcal{G}(K)\to \mathfrak{F}\mathcal{G}(K),\;
\mathfrak{F}([g])(z):=\int_{\gamma_{\overline{\R}}}e^{-iz\zeta}g(\zeta)\d\zeta,\;z\in\C,
\]
where $\gamma_{\overline{\R}}$ is the path along the boundary of $U_{\nicefrac{1}{c}}(\overline{\R})$ for $c>0$ 
with clockwise orientation, is a topological isomorphism (for $K=\{\pm\infty\}$ the path may be deformed to the 
boundary of $U_{\nicefrac{1}{c}}(\{\pm\infty\})$). 

The canonical (restriction) map 
\[
R\colon \mathcal{H}_{-\infty}(\overline{\C}\setminus\overline{\R})/
\mathcal{H}_{-\infty}(\overline{\C}\setminus\{\pm\infty\})
\to \mathcal{B}(\R),\;[F]\mapsto [F],
\]
is a linear isomorphism by \cite[Theorem 3.1, p.\ 49]{langenbruch2011_1}. 
The combination of both results gives the following theorem. 

\begin{thm}[{\cite[Theorem 3.4, p.\ 51]{langenbruch2011}}]\label{thm:fourier_asymp_iso_langenbruch}
The asymptotic Fourier transform 
\[
\mathfrak{F}_{\mathcal{B}}\colon \mathcal{B}(\R)\to 
\mathfrak{F}\mathcal{G}(\overline{\R})/\mathfrak{F}\mathcal{G}(\{\pm\infty\}),\;
\mathfrak{F}_{\mathcal{B}}(f):=[(\mathfrak{F}\circ R^{-1})(f)],
\]
is a linear isomorphism. 
\end{thm}

\begin{prop}\label{prop:fourier_decomposition_langenbruch}
Let $j\in\N$. 
\begin{enumerate}
\item[a)] Then for every $f\in \mathcal{B}([0,\infty[)$ there are $f_{1}\in \mathcal{B}([0,j])$ 
and $f_{2}\in \mathcal{B}([j,\infty[)$ such that $f=f_{1}+f_{2}$. 
For $h\in\mathfrak{F}_{\mathcal{B}}(f)$ we have for every $k\in\N$ 
\[
\sup_{\im(z)\leq k}|h(z)-\mathcal{F}_{[0,j]}^{\mathcal{B}}(f_{1})(z)|e^{-\frac{1}{k}|z|+j|\im(z)|}<\infty .
\]
\item[b)] Then for every $f\in \mathcal{B}(]-\infty,0])$ there are $f_{1}\in \mathcal{B}([-j,0])$ 
and $f_{2}\in \mathcal{B}(]-\infty,-j])$ such that $f=f_{1}+f_{2}$. 
For $h\in\mathfrak{F}_{\mathcal{B}}(f)$ we have for every $k\in\N$ 
\[
\sup_{\im(z)\geq -k}|h(z)-\mathcal{F}_{[-j,0]}^{\mathcal{B}}(f_{1})(z)|e^{-\frac{1}{k}|z|+j|\im(z)|}<\infty .
\]
\end{enumerate}
\end{prop}
\begin{proof}
a) The sheaf of $\C$-valued hyperfunctions is flabby 
by \cite[p.\ 392]{Sato2} and thus for every 
$f\in \mathcal{B}([0,\infty[)$ there are $f_{1}\in \mathcal{B}([0,j])$ and $f_{2}\in \mathcal{B}([j,\infty[)$ 
such that $f=f_{1}+f_{2}$ by \cite[Lemma 1.4.4, p.\ 36]{Kan}. Let $[g]= R^{-1}(f)$. Then we even have 
$g\in\mathcal{H}_{-\infty}(\overline{\C}\setminus[0,\infty])$ and 
$g-g_{1}\in\mathcal{H}_{-\infty}(\overline{\C}\setminus\{\infty\})$ for any $g_{1}$ with $[g_{1}]=[g]$ 
by \cite[Corollary 4.2, p.\ 41]{langenbruch2011_1} as well as
\[
 \mathfrak{F}(g)(z)
=\int_{\gamma_{\overline{\R}}}e^{-iz\zeta}g(\zeta)\d\zeta 
=\int_{\gamma_{[0,\infty]}}e^{-iz\zeta}g(\zeta)\d\zeta
=\mathfrak{L}(g)(iz),\;z\in\C,
\]
by Cauchy's integral theorem with the Laplace transform $\mathfrak{L}$ from \eqref{eq:laplace_def_lan}. 
Let $h\in\mathfrak{F}_{\mathcal{B}}(f)$. It follows that there are 
$g\in\mathcal{H}_{-\infty}(\overline{\C}\setminus[0,\infty])$ and 
$h_{1}\in\mathcal{H}_{-\infty}(\overline{\C}\setminus\{\infty\})$ such that $h=\mathfrak{F}(g+h_{1})$. 
Hence we have 
\begin{flalign*}
&\hspace{0.35cm}\sup_{\im(z)\leq k}|h(z)-\mathcal{F}_{[0,j]}^{\mathcal{B}}(f_{1})(z)|e^{-\frac{1}{k}|z|+j|\im(z)|}\\
&=\sup_{\im(z)\leq k}|\mathfrak{L}(g+h_{1})(iz)-\mathcal{L}_{[0,j]}^{\mathcal{B}}(f_{1})(iz)|
  e^{-\frac{1}{k}|z|+j|\im(z)|}\\
&=\sup_{\re(z)\geq -k}|\mathfrak{L}(g+h_{1})(z)-\mathcal{L}_{[0,j]}^{\mathcal{B}}(f_{1})(z)|
  e^{-\frac{1}{k}|z|+j|\re(z)|}<\infty
\end{flalign*}
for every $k\in\N$ by \prettyref{prop:laplace_decomposition_langenbruch} because 
$\mathfrak{L}(g+h_{1})\in\mathfrak{L}_{\mathcal{B}}(f)$.

b) As above for every $f\in \mathcal{B}(]-\infty,0])$ there are $f_{1}=[F_{1}]\in \mathcal{B}([-j,0])$ and 
$f_{2}=[F_{2}]\in \mathcal{B}(]-\infty,-j])$ such that $f=f_{1}+f_{2}$. 
The map $[F]\mapsto[F(-\cdot)]$ is a linear isomorphism $\mathcal{B}(]-\infty,0])\to\mathcal{B}([0,\infty[)$ resp.\ 
$\mathcal{H}_{-\infty}(\overline{\C}\setminus[-\infty,0])/\mathcal{H}_{-\infty}(\overline{\C}\setminus\{-\infty\})\to
\mathcal{H}_{-\infty}(\overline{\C}\setminus[0,\infty])/\mathcal{H}_{-\infty}(\overline{\C}\setminus\{\infty\})$.
Due to part a) we have for $[g]= R^{-1}(f)$ that
$g\in\mathcal{H}_{-\infty}(\overline{\C}\setminus[-\infty,0])$ and 
$g-g_{1}\in\mathcal{H}_{-\infty}(\overline{\C}\setminus\{-\infty\})$ for any $g_{1}$ with $[g_{1}]=[g]$ as well as
\begin{align*}
 \mathfrak{F}(g)(z)
&=\int_{\gamma_{\overline{\R}}}e^{-iz\zeta}g(\zeta)\d\zeta 
 =\int_{\gamma_{[-\infty,0]}}e^{-iz\zeta}g(\zeta)\d\zeta
 =\int_{\gamma_{[0,\infty]}}e^{iz\zeta}g(-\zeta)\d\zeta\\
&=\mathfrak{L}(g(-\cdot))(-iz),\;z\in\C,
\end{align*}
by Cauchy's integral theorem. The rest follows from part a).
\end{proof}

\begin{thm}\label{thm:asymptotic_fourier_equiv}
The restriction map 
\[
I\colon\mathfrak{F}\mathcal{G}(\overline{\R})/\mathfrak{F}\mathcal{G}(\{\pm\infty\})\to 
\mathcal{O}^{exp}(\overline{\C}\setminus\overline{\R})/
\bigl(\mathcal{O}^{exp}(\overline{\C})\oplus \mathcal{FO}_{\{\pm\infty\}}\bigr),\;[F]\mapsto [F_{+}],
\]
where $F_{+}:=F$ on $\h$ and $F_{+}:=0$ on $-\h$, is a linear isomorphism 
and $I\circ\mathfrak{F}_{\mathcal{B}}=\mathcal{F}^{\mathcal{B}}$.
\end{thm}
\begin{proof}
The map $I$ is well-defined because $F_{+}\in\mathcal{O}^{exp}(\overline{\C}\setminus\overline{\R})$ for 
$F\in\mathfrak{F}\mathcal{G}(\overline{\R})=\mathcal{O}^{exp}(\overline{\C})$ 
and $F_{+}\in\mathcal{FO}_{\{\infty\}}$ for $F\in\mathfrak{F}\mathcal{G}(\{\pm\infty\})$.  

Now, we prove $I\circ\mathfrak{F}_{\mathcal{B}}=\mathcal{F}^{\mathcal{B}}$, which also implies that $I$ is 
a linear isomorphism since $\mathfrak{F}_{\mathcal{B}}$ and $\mathcal{F}^{\mathcal{B}}$ are linear isomorphism 
by \prettyref{thm:fourier_asymp_iso_langenbruch} and \prettyref{thm:asymptotic_fourier}. 
Let $f\in\mathcal{B}(\R)$ and $j\in\N$. Then there are $f_{1}\in\mathcal{B}(]-\infty,0])$ and 
$f_{2}\in\mathcal{B}([0,\infty[)$ such that $f=f_{1}+f_{2}$ due to the flabbiness of the sheaf of $\C$-valued 
hyperfunctions by \cite[p.\ 392]{Sato2} and \cite[Lemma 1.4.4, p.\ 36]{Kan}. Furthermore, there are
$f_{1,1}\in\mathcal{B}([-j,0])$ and $f_{1,2}\in\mathcal{B}(]-\infty,-j])$ 
such that $f_{1}=f_{1,1}+f_{1,2}$ and
\[
(\mathcal{F}\circ \mathcal{R}_{]-\infty,0]}^{-1})(f_{1})-\mathcal{F}_{[-j,0]}^{\mathcal{B}}(f_{1,1})_{\mid\im >0}
\in\mathcal{FO}_{[-\infty,-j]}
\]
as well as $f_{2,1}\in\mathcal{B}([0,j])$ and $f_{2,2}\in\mathcal{B}([j,\infty[)$ 
such that $f_{2}=f_{2,1}+f_{2,2}$ and
\[
(\mathcal{F}\circ \mathcal{R}_{[0,\infty[}^{-1})(f_{2})-\mathcal{F}_{[0,j]}^{\mathcal{B}}(f_{2,1})_{\mid\im <0}
\in\mathcal{FO}_{[j,\infty]}
\]
by \prettyref{prop:fourier_decomposition}. In addition, we obtain 
\[
\sup_{\im(z)\geq -k}|(\mathfrak{F}\circ R^{-1})(f_{1})(z)-\mathcal{F}_{[-j,0]}^{\mathcal{B}}(f_{1,1})(z)|
 e^{-\frac{1}{k}|z|+j|\im(z)|}<\infty 
\]
and
\[
\sup_{\im(z)\leq k}|(\mathfrak{F}\circ R^{-1})(f_{2})(z)-\mathcal{F}_{[0,j]}^{\mathcal{B}}(f_{2,1})(z)|
 e^{-\frac{1}{k}|z|+j|\im(z)|}<\infty 
\]
for every $k\in\N$ from \prettyref{prop:fourier_decomposition_langenbruch}. We remark that 
\begin{flalign*}
&\hspace{0.35cm} \mathfrak{F}(R^{-1}(f))_{+}-\mathcal{F}_{\overline{\R}}(\mathcal{R}_{\R}^{-1}(f))\\
&=\mathfrak{F}(R^{-1}(f_{1}))_{+}+\mathfrak{F}(R^{-1}(f_{2}))_{+}
 -\mathcal{F}_{\overline{\R}}(\mathcal{R}_{\R}^{-1}(f_{1}))-\mathcal{F}_{\overline{\R}}(\mathcal{R}_{\R}^{-1}(f_{2}))\\
&=\mathfrak{F}(R^{-1}(f_{1}))_{+}-\mathcal{F}(\mathcal{R}_{]-\infty,0]}^{-1}(f_{1}))
 +\mathfrak{F}(R^{-1}(f_{2}))_{+}-(-\mathcal{F}(\mathcal{R}_{[0,\infty[}^{-1}(f_{2})))\\
&=\bigl(\mathfrak{F}(R^{-1}(f_{1}))_{+}-\mathcal{F}(\mathcal{R}_{]-\infty,0]}^{-1}(f_{1}))\bigr)
 +\mathfrak{F}(R^{-1}(f_{2}))\\
&\phantom{=}+\bigl(-\mathfrak{F}(R^{-1}(f_{2}))_{-}+\mathcal{F}(\mathcal{R}_{[0,\infty[}^{-1}(f_{2}))\bigr)
\end{flalign*}
where $\mathfrak{F}(R^{-1}(f_{2}))_{-}:=\mathfrak{F}(R^{-1}(f_{2}))$ on $-\h$ and 
$\mathfrak{F}(R^{-1}(f_{2}))_{-}:=0$ on $\h$. We deduce with $k=j$ that
\begin{flalign*}
&\hspace{0.35cm}|\mathfrak{F}(R^{-1}(f_{1}))_{+}-\mathcal{F}(\mathcal{R}_{]-\infty,0]}^{-1}(f_{1}))|_{j,\{-\infty\}}\\
&\leq|\mathfrak{F}(R^{-1}(f_{1}))_{+}-\mathcal{F}_{[-j,0]}^{\mathcal{B}}(f_{1,1})|_{j,\{-\infty\}}
 +|\mathcal{F}_{[-j,0]}^{\mathcal{B}}(f_{1,1})-\mathcal{F}(\mathcal{R}_{]-\infty,0]}^{-1}(f_{1}))|_{j,\{-\infty\}}\\
&=\phantom{+}\sup_{\im(z)\geq \frac{1}{j}}|\mathfrak{F}(R^{-1}(f_{1}))(z)-\mathcal{F}_{[-j,0]}^{\mathcal{B}}(f_{1,1})(z)|
  e^{-\frac{1}{j}|z|+j|\im(z)|}\\
&\phantom{=}+\sup_{\im(z)\geq \frac{1}{j}}|\mathcal{F}_{[-j,0]}^{\mathcal{B}}(f_{1,1})(z)
  -\mathcal{F}(\mathcal{R}_{]-\infty,0]}^{-1}(f_{1}))(z)|e^{-\frac{1}{j}|z|+j|\im(z)|}
 <\infty
\end{flalign*}
and 
\begin{flalign*}
&\hspace{0.35cm}|-\mathfrak{F}(R^{-1}(f_{2}))_{-}+\mathcal{F}(\mathcal{R}_{[0,\infty[}^{-1}(f_{2}))|_{j,\{\infty\}}\\
&\leq|-\mathfrak{F}(R^{-1}(f_{2}))_{-}+\mathcal{F}_{[0,j]}^{\mathcal{B}}(f_{2,1})|_{j,\{\infty\}}
 +|-\mathcal{F}_{[0,j]}^{\mathcal{B}}(f_{2,1})+\mathcal{F}(\mathcal{R}_{[0,\infty[}^{-1}(f_{2}))|_{j,\{\infty\}}\\
&=\phantom{+}\sup_{\im(z)\leq -\frac{1}{j}}|-\mathfrak{F}(R^{-1}(f_{2}))(z)
 +\mathcal{F}_{[0,j]}^{\mathcal{B}}(f_{2,1})(z)|e^{-\frac{1}{j}|z|+j|\im(z)|}\\
&\phantom{=}+\sup_{\im(z)\leq -\frac{1}{j}}|-\mathcal{F}_{[0,j]}^{\mathcal{B}}(f_{2,1})(z)
  +\mathcal{F}(\mathcal{R}_{[0,\infty[}^{-1}(f_{2}))(z)|e^{-\frac{1}{j}|z|+j|\im(z)|}
 <\infty
\end{flalign*}
plus $\mathfrak{F}(R^{-1}(f_{2}))\in\mathfrak{F}\mathcal{G}(\overline{\R})=\mathcal{O}^{exp}(\overline{\C})$.
Hence we conclude 
\[
 \mathfrak{F}(R^{-1}(f))_{+}-\mathcal{F}_{\overline{\R}}(\mathcal{R}_{\R}^{-1}(f))
\in(\mathcal{O}^{exp}(\overline{\C})\oplus \mathcal{FO}_{\{\pm\infty\}}\bigr),
\]
resulting in $I\circ\mathfrak{F}_{\mathcal{B}}=\mathcal{F}^{\mathcal{B}}$.
\end{proof}

\section{Laplace and asymptotic Laplace transform}
\label{sect:Laplace_and_Asymptotic_Laplace}
In this short section we phrase our results on the (asymptotic) Fourier transform in terms 
of the Laplace transform. We start with the definition of the range spaces.

\begin{defn} 
Let $E$ be a $\C$-lcHs.
\begin{enumerate}
\item [a)] For $-\infty<a\leq\infty$ we define the space
\[
\mathcal{LO}_{[a,\infty]}(E):=\{f\in\mathcal{O}(-i\h,E)\;|\;\forall\;k\in\N,\,\alpha\in\mathfrak{A}:\;
\|f\|_{k,\alpha,[a,\infty]}<\infty\}
\]
where
\[
 \|f\|_{k,\alpha,[a,\infty]}:=\sup_{\re(z)\geq \frac{1}{k}}p_{\alpha}(f(z))e^{-\frac{1}{k}|z|-w^{+}_{a}(-\re(z))}.
\]
\item [b)] For $-\infty\leq a<\infty$ we define the space
\[
\mathcal{LO}_{[-\infty,a]}(E):=\{f\in\mathcal{O}(i\h,E)\;|\;\forall\;k\in\N,\,\alpha\in\mathfrak{A}:\;
\|f\|_{k,\alpha,[-\infty,a]}<\infty\}
\]
where
\[
 \|f\|_{k,\alpha,[-\infty,a]}:=\sup_{\re(z)\leq -\frac{1}{k}}p_{\alpha}(f(z))e^{-\frac{1}{k}|z|-w^{-}_{a}(-\re(z))}.
\]
\item [c)] For $-\infty<a\leq b<\infty$ we define the space
\[
\mathcal{LO}_{[a,b]}(E):=\{f\in\mathcal{O}(\C,E)\;|\;\forall\;k\in\N,\,\alpha\in\mathfrak{A}:\;
\|f\|_{k,\alpha,[a,b]}<\infty\}
\]
where
\[
 \|f\|_{k,\alpha,[a,b]}:=\sup_{z\in\C}p_{\alpha}(f(z))e^{-\frac{1}{k}|z|-H_{[a,b]}(-\re(z))}.
\]
\end{enumerate}
\end{defn}

\begin{thm}\label{thm:laplace_top_iso}
Let $E$ be a $\C$-lcHs and $\varnothing\neq K\varsubsetneq\overline{\R}$ a compact interval. If 
\begin{enumerate}
\item[(i)] $K\subset\R$ and $E$ is locally complete, or if
\item[(ii)] $E$ is sequentially complete,
\end{enumerate}
then the map
\[
\mathcal{L}\colon bv_{K}(E) \to \mathcal{LO}_{K}(E),\;\;
\mathcal{L}([F])(\zeta):=\mathcal{F}([F])(-i\zeta),
\]
is a topological isomorphism.
\end{thm}
\begin{proof}
This follows directly from \prettyref{thm:fourier_unbounded_interval} resp.\
\prettyref{thm:fourier_bounded_interval} and $\im(-i\zeta)=-\re(\zeta)$. 
\end{proof}

\begin{prop}\label{prop:bv_LO_nuclear}
Let $\varnothing\neq K\varsubsetneq\overline{\R}$ be a compact interval. 
Then the spaces $bv_{K}$ and $\mathcal{LO}_{K}$ are nuclear Fr\'echet spaces.
\end{prop}
\begin{proof}
For $bv_{K}$ this statement follows from \cite[Remark 3.4 a), p.\ 8]{kruse2019_5} and 
\cite[Proposition 28.6, p.\ 347]{meisevogt1997}. This implies that $\mathcal{LO}_{K}$ 
is a nuclear Fr\'echet space by \prettyref{thm:laplace_top_iso} too.
\end{proof}

\begin{thm}\label{thm:asymptotic_laplace}
Let $E$ be an admissible sequentially complete $\C$-lcHs and $A:=]-\infty,a]$ or $A:=[a,\infty[$ for some $a\in\R$. 
Then the asymptotic Laplace transform 
\[
\mathcal{L}_{A}^{\mathcal{B}}\colon 
\mathcal{B}(A,E)\to 
\mathcal{LO}_{\overline{A}}(E)/\mathcal{LO}_{\overline{A}\setminus A}(E),\;
\mathcal{L}_{A}^{\mathcal{B}}(f):=[(\mathcal{L}\circ \mathcal{R}_{A}^{-1})(f)],
\]
is a linear isomorphism, where the closure $\overline{A}$ is taken in $\overline{\R}$.
\end{thm}
\begin{proof}
It follows directly from \prettyref{thm:extension} and \prettyref{thm:laplace_top_iso} (ii). 
\end{proof}

Due to the theorem above, \prettyref{prop:bv_LO_nuclear} and \prettyref{thm:examples_strictly_admiss} 
our asymptotic Laplace transform fulfils the conditions (I) and (IV) from the introduction. 

\begin{cor}\label{cor:laplace_hyp_real_comp}
Let $E$ be a locally complete $\C$-lcHs and $-\infty<a\leq b<\infty$. Then the Laplace transform
\[
\mathcal{L}_{[a,b]}^{\mathcal{B}}\colon \mathcal{B}([a,b],E)\to \mathcal{LO}_{[a,b]}(E),\;
\mathcal{L}_{[a,b]}^{\mathcal{B}}([F]):=(\mathcal{L}_{[a,b]}\circ \mathcal{R}_{[a,b]}^{-1})([F]),
\]
is a topological isomorphism and 
\begin{equation}\label{eq:laplace_hyp_real_comp}
 \mathcal{L}_{[a,b]}^{\mathcal{B}}([F])(\zeta)
=\mathcal{L}_{[a,b]}([\Psi([F])])(\zeta)=\langle \mathscr{H}_{[a,b]}([F]),e^{-(\cdot)\zeta}\rangle,\;\zeta\in\C .
\end{equation}
\end{cor}
\begin{proof}
This holds due to \prettyref{thm:extension_real_comp} and \prettyref{thm:laplace_top_iso} (i).
\end{proof}

\begin{prop}\label{prop:laplace_decomposition}
Let $E$ be a strictly admissible sequentially complete $\C$-lcHs and $-\infty<a<b<\infty$. Then the following holds:
\begin{enumerate}
\item[a)] For every 
$f\in \mathcal{B}([a,\infty[,E)$ there are $f_{1}\in \mathcal{B}([a,b],E)$ and $f_{2}\in \mathcal{B}([b,\infty[,E)$ 
such that $f=f_{1}+f_{2}$ and 
\[
 \mathcal{L}_{[a,\infty[}^{\mathcal{B}}(f)
=[\mathcal{L}_{[a,b]}^{\mathcal{B}}(f_{1})_{\mid\re >0}]+\mathcal{L}_{[b,\infty[}^{\mathcal{B}}(f_{2}).
\]
\item[b)] For every 
$f\in \mathcal{B}(]-\infty,b],E)$ there are $f_{1}\in \mathcal{B}([a,b],E)$ and $f_{2}\in \mathcal{B}(]-\infty,a],E)$ 
such that $f=f_{1}+f_{2}$ and 
\[
 \mathcal{L}_{]-\infty,b]}^{\mathcal{B}}(f)
=[\mathcal{L}_{[a,b]}^{\mathcal{B}}(f_{1})_{\mid\re <0}]+\mathcal{L}_{]-\infty,a]}^{\mathcal{B}}(f_{2}).
\]
\end{enumerate}
\end{prop}
\begin{proof}
The statement is a consequence of \prettyref{prop:fourier_decomposition} and $\im(-i\zeta)=-\re(\zeta)$.
\end{proof}

\begin{prop}\label{prop:laplace_transform_of_shift}
Let $E$ be a $\C$-lcHs and $\varnothing\neq K\varsubsetneq\overline{\R}$ a compact interval. Let 
$\tau_{h}([F]):=[F(\cdot -h)]$, $h\in\R$, be the shift operator from $bv_{K}(E)$ to $bv_{h+K}(E)$. If 
\begin{enumerate}
\item[(i)] $K\subset\R$ and $E$ is locally complete, or if
\item[(ii)] $E$ is sequentially complete,
\end{enumerate}
then we have for $[F]\in bv_{K}(E)$
\[
\mathcal{L}_{h+K}(\tau_{h}([F]))=e^{-h(\cdot)}\mathcal{L}_{K}([F]).
\]
\end{prop}
\begin{proof}
This is a consequence of \prettyref{prop:fourier_transform_of_shift} and the definition of $\mathcal{L}_{K}$.
\end{proof}

\begin{prop}\label{prop:asymp_laplace_transform_of_shift}
Let $E$ be a sequentially complete, $\C$-lcHs and $A:=]-\infty,a]$ or $A:=[a,\infty[$ for some $a\in\R$. Let 
$\tau_{h}([F]):=[F(\cdot -h)]$, $h\in\R$, be the shift operator from $\mathcal{B}(A,E)$ to $\mathcal{B}(h+A,E)$. 
Then we have for $[F]\in\mathcal{B}(A,E)$
\[
\mathcal{L}_{h+A}^{\mathcal{B}}(\tau_{h}([F]))=e^{-h(\cdot)}\mathcal{L}_{A}^{\mathcal{B}}([F]).
\]
\end{prop}
\begin{proof}
This is a consequence of \prettyref{prop:asymp_fourier_transform_of_shift} and the definition of 
$\mathcal{L}_{A}^{\mathcal{B}}$.
\end{proof}

\begin{prop}\label{prop:laplace_transform_of_derivative}
Let $E$ be a sequentially complete $\C$-lcHs, $\varnothing\neq K\varsubsetneq\overline{\R}$ a compact interval and 
$P(\partial):=\sum_{k=0}^{\infty}\tfrac{c_{k}}{k!}\partial^{k}$ where $(c_{k})\subset\C$ and 
$P$ is of exponential type $0$.
Then we have for $[F]\in bv_{K}(E)$
\[
\mathcal{L}_{K}(P(\partial)[F])(\zeta)=P(\zeta)\mathcal{L}_{K}([F])(\zeta)
 \quad\text{and}\quad
\mathcal{L}_{K}(P[F])(\zeta)=P(-\partial)\mathcal{L}_{K}([F])(\zeta)
\]
for $\re(\zeta)>0$ if $K=[a,\infty]$, for $\re(\zeta)<0$ if $K=[-\infty,a]$, and for $\zeta\in\C$ if 
$K\subset\R$, respectively. 
\end{prop}
\begin{proof}
This follows from \prettyref{prop:fourier_transform_of_derivative} and the definition of $\mathcal{L}_{K}$.
\end{proof}

\begin{prop}\label{prop:asymp_laplace_transform_of_derivative}
Let $E$ be an admissible sequentially complete $\C$-lcHs, $A:=]-\infty,a]$ or $A:=[a,\infty[$ for some $a\in\R$ and 
$P(\partial):=\sum_{k=0}^{\infty}\tfrac{c_{k}}{k!}\partial^{k}$ where $(c_{k})\subset\C$ and 
$P$ is of exponential type $0$.
Then we have for $f\in\mathcal{B}(A,E)$
\[
\mathcal{L}_{A}^{\mathcal{B}}(P(\partial)f)=P\mathcal{L}_{A}^{\mathcal{B}}(f)
 \quad\text{and}\quad
\mathcal{L}_{A}^{\mathcal{B}}(Pf)=P(-\partial)\mathcal{L}_{A}^{\mathcal{B}}(f).
\]
\end{prop}
\begin{proof}
This follows from \prettyref{prop:asymp_fourier_transform_of_derivative} and the definition of 
$\mathcal{L}_{A}^{\mathcal{B}}$.
\end{proof}

\begin{prop}\label{prop:asymp_laplace_convolution}
Let $(E,(p_{\alpha})_{\alpha\in\mathfrak{A}})$, $(E_{1},(p_{\beta})_{\beta\in\mathfrak{B}})$ 
and $(E_{2},(p_{\omega})_{\omega\in\Omega})$ be admissible sequentially complete $\C$-lcHs, 
$E$ and $E_{1}$ strictly admissible, such that 
a canonical bilinear map $\cdot\:\colon E_{1}\times E_{2}\to E$ is defined with the property
\[
\forall\; \alpha\in\mathfrak{A}\;\exists\;\beta\in\mathfrak{B},\,\omega\in\Omega,\,D>0\;\forall\;x\in E_{1},\,y\in E_{2}:
\;p_{\alpha}(x\cdot y)\leq D p_{\beta}(x)p_{\omega}(y).
\]
Let $a,b,c,d\in\R$.
Then we have for $f\in\mathcal{B}([a,\infty[,E_{1})$ and $g\in\mathcal{B}([c,\infty[,E_{2})$, or for 
$f\in\mathcal{B}(]-\infty,b],E_{1})$ and $g\in\mathcal{B}(]-\infty,d],E_{2})$ that
\[
 \mathcal{L}^{\mathcal{B}}(f\ast^{\mathcal{B}}g)
=\mathcal{L}^{\mathcal{B}}(f\circledast^{\mathcal{B}}g)
=\mathcal{L}^{\mathcal{B}}(f)\mathcal{L}^{\mathcal{B}}(g).
\]
\end{prop}
\begin{proof}
This follows from \prettyref{cor:asymp_fourier_convolution} c) and the definition of 
$\mathcal{L}_{A}^{\mathcal{B}}$.
\end{proof}

Due to \prettyref{eq:laplace_hyp_real_comp}, \prettyref{prop:asymp_laplace_transform_of_derivative}
and \prettyref{prop:asymp_laplace_convolution} combined with \prettyref{thm:asymp_fourier_convolution} 
and \prettyref{cor:asymp_fourier_convolution} our asymptotic Laplace transform fulfils condition (II) 
for a satisfactory theory of Laplace transforms as well. 
\section{Langenbruch's asymptotic Laplace transform}
\label{sect:Asymptotic_Laplace_transform_Lan}
In \prettyref{sect:Asymptotic_Fourier_transform_Lan} we already discussed the relation between our asymptotic 
Fourier transform and the one Langenbruch on $\mathcal{B}(\R)$. This section is dedicated to 
the connection of our asymptotic Laplace transform on $\mathcal{B}([0,\infty[,E)$ with the one of Langenbruch 
in the case that $E$ is a $\C$-Fr\'echet space and with the ones of Komatsu, Lumer and Neubrander in the case that $E$ 
is a $\C$-Banach space as well. We recall the relevant notions and results from \cite{langenbruch2011} 
(and the introduction). 

\begin{defn}[{\cite[p.\ 53, 55, 61]{langenbruch2011}}]\label{def:range_laplace_lan}
Let $E$ be a $\C$-lcHs. We define the spaces 
\[
\mathfrak{L}\mathcal{G}_{[0,\infty]}(E):=\{f\in\mathcal{O}(\C,E)\;|\;\forall\;k\in\N,\,\alpha\in\mathfrak{A}:\; 
\|f\|_{k,\alpha,[0,\infty]}^{\mathfrak{L}\mathcal{G}}<\infty\}
\]
where 
\[
\|f\|_{k,\alpha,[0,\infty]}^{\mathfrak{L}\mathcal{G}}:=\sup_{\re(z)\geq -k}
p_{\alpha}(f(z))e^{-\frac{1}{k}|z|},
\]
and 
\[
\mathfrak{L}\mathcal{G}_{\{\infty\}}(E):=\{f\in\mathcal{O}(\C,E)\;|\;\forall\;k\in\N,\,\alpha\in\mathfrak{A}:\; 
\|f\|_{k,\alpha,\{\infty\}}^{\mathfrak{L}\mathcal{G}}<\infty\}
\]
where 
\[
\|f\|_{k,\alpha,\{\infty\}}^{\mathfrak{L}\mathcal{G}}:=\sup_{\re(z)\geq -k}
p_{\alpha}(f(z))e^{k|\re(z)|-\frac{1}{k}|z|}.
\]
\end{defn}

Let $E$ be a $\C$-Fr\'echet space and $K=[0,\infty]$ or $K=\{\infty\}$. By 
\cite[Theorem 4.1, p.\ 53, Proposition 5.2, p.\ 55, 61]{langenbruch2011} the Laplace transform 
\begin{equation}\label{eq:laplace_def_lan}
\mathfrak{L}\colon \mathcal{G}_{K}(E)\to \mathfrak{L}\mathcal{G}_{K}(E),\;
\mathfrak{L}([g])(z):=\int_{\gamma_{K}}g(\zeta)e^{-z\zeta}\d\zeta,\;z\in\C,
\end{equation}
where $\gamma_{K}$ is the path along the boundary of $U_{\nicefrac{1}{c}}(K)$ for $c>0$ with clockwise orientation, 
is a topological isomorphism. 

The canonical (restriction) map 
\[
R_{+}\colon \mathcal{H}_{-\infty}(\overline{\C}\setminus [0,\infty],E)/
\mathcal{H}_{-\infty}(\overline{\C}\setminus\{\infty\},E)
\to \mathcal{B}([0,\infty[,E)
\]
is a linear isomorphism for any $\C$-Fr\'echet space $E$ by \cite[Corollary 5.2, p.\ 42]{langenbruch2011_1}. 
The combination of both results gives the following theorem. 

\begin{thm}[{\cite[Theorem 5.3, p.\ 55, 61]{langenbruch2011}}]\label{thm:asymp_laplace_iso_langenbruch}
Let $E$ be a $\C$-Fr\'echet space. Then the asymptotic Laplace transform 
\[
\mathfrak{L}_{\mathcal{B}}\colon \mathcal{B}([0,\infty[,E)\to 
\mathfrak{L}\mathcal{G}_{[0,\infty]}(E)/\mathfrak{L}\mathcal{G}_{\{\infty\}}(E),\;
\mathfrak{L}_{\mathcal{B}}(f):=[(\mathfrak{L}\circ R_{+}^{-1})(f)],
\]
is a linear isomorphism. 
\end{thm}

\begin{prop}\label{prop:laplace_decomposition_langenbruch}
Let $E$ be a $\C$-Fr\'echet space and $j\in\N$. Then for every 
$f\in \mathcal{B}([0,\infty[,E)$ there are $f_{1}\in \mathcal{B}([0,j],E)$ and $f_{2}\in \mathcal{B}([j,\infty[,E)$ 
such that $f=f_{1}+f_{2}$. For $h\in\mathfrak{L}_{\mathcal{B}}(f)$ we have for every $k\in\N$ and $\alpha\in\mathfrak{A}$ 
\[
\sup_{\re(z)\geq -k}p_{\alpha}(h(z)-\mathcal{L}_{[0,j]}^{\mathcal{B}}(f_{1})(z))e^{-\frac{1}{k}|z|+j|\re(z)|}<\infty .
\]
\end{prop}
\begin{proof}
First, we observe that the operation $e'(g):=[e'\circ G]$ is well-defined for 
$g=[G]\in\mathcal{B}([0,\infty[,E)$ resp.\ $g=[G]\in\mathcal{B}([0,j],E)$ resp.\
$g=[G]\in\mathfrak{L}\mathcal{G}_{[0,\infty]}(E)/\mathfrak{L}\mathcal{G}_{\{\infty\}}(E)$ and any $e'\in E'$.

The decomposition $f=f_{1}+f_{2}$ follows from the flabbiness of the sheaf of Fr\'echet-valued hyperfunctions 
by \cite[Theorem 2.6, p.\ 14]{Ion/Ka} and \cite[Lemma 1.4.4, p.\ 36]{Kan}. 
Let $h\in\mathfrak{L}_{\mathcal{B}}(f)$, $e'\in E'$, $f=[F]$,
$f_{1}=[F_{1}]$ and $f_{2}=[F_{2}]$. Then $e'(f)=[e'\circ F_{1}]+[e'\circ F_{2}]$ and 
\[
(e'\circ \mathcal{L}_{[0,j]}^{\mathcal{B}}(f_{1}))(z)=\mathcal{L}_{[0,j]}^{\mathcal{B}}([e'\circ F_{1}])(z),\;z\in\C,
\]
as well as $e'\circ h\in\mathfrak{L}_{\mathcal{B}}([e'\circ F])$ because 
$\mathfrak{L}(e'\circ g)=e'\circ \mathfrak{L}(g)$ for any 
$g\in\mathcal{H}_{-\infty}(\overline{\C}\setminus [0,\infty],E)$.
By \cite[Lemma 5.4 (a), p.\ 56]{langenbruch2011} ($\mathcal{L}_{[0,j]}^{\mathcal{B}}([e'\circ F_{1}])$ corresponds to 
$\widetilde{\nu}_{j}$ in this lemma due to \prettyref{cor:duality_real_comp} and \eqref{eq:laplace_hyp_real_comp}) 
we obtain 
\[
\sup_{\re(z)\geq -k}|(e'\circ h)(z)-\mathcal{L}_{[0,j]}^{\mathcal{B}}([e'\circ F_{1}])(z)|
e^{-\frac{1}{k}|z|+j|\re(z)|}<\infty 
\]
for every $k\in\N$. Since this holds for any $e'\in E'$, the Mackey theorem implies that 
\[
\sup_{\re(z)\geq -k}p_{\alpha}(h(z)-\mathcal{L}_{[0,j]}^{\mathcal{B}}([f_{1}])(z))
e^{-\frac{1}{k}|z|+j|\re(z)|}<\infty 
\]
every $k\in\N$ and $\alpha\in\mathfrak{A}$.
\end{proof}

We come to our main result of this section.

\begin{thm}\label{thm:asymptotic_laplace_equiv}
Let $E$ be a $\C$-Fr\'echet space. Then the canonical map 
\[
I_{0}\colon \mathfrak{L}\mathcal{G}_{[0,\infty]}(E)/\mathfrak{L}\mathcal{G}_{\{\infty\}}(E)\to 
\mathcal{LO}_{[0,\infty]}(E)/\mathcal{LO}_{\{\infty\}}(E),\;[F]\mapsto [F],\;
\]
is a linear isomorphism and $I_{0}\circ\mathfrak{L}_{\mathcal{B}}=\mathcal{L}^{\mathcal{B}}$.
\end{thm}
\begin{proof}
The map $I_{0}$ is well-defined as $\mathfrak{L}\mathcal{G}_{[0,\infty]}(E)\subset \mathcal{LO}_{[0,\infty]}(E)$ and 
$\mathfrak{L}\mathcal{G}_{\{\infty\}}(E)\subset\mathcal{LO}_{\{\infty\}}(E)$.

We only need to prove that $I_{0}\circ\mathfrak{L}_{\mathcal{B}}=\mathcal{L}^{\mathcal{B}}$ since 
$\mathfrak{L}_{\mathcal{B}}$ and $\mathcal{L}^{\mathcal{B}}$ are linear isomorphisms by 
\prettyref{thm:asymp_laplace_iso_langenbruch}, \prettyref{thm:asymptotic_laplace} and
\prettyref{thm:examples_strictly_admiss} a).
Let $f\in\mathcal{B}([0,\infty[,E)$ and $j\in\N$. Then there are
$f_{1}\in\mathcal{B}([0,j],E)$ and $f_{2}\in\mathcal{B}([j,\infty[,E)$ 
such that $f=f_{1}+f_{2}$ and
\[
(\mathcal{L}\circ \mathcal{R}_{[0,\infty[}^{-1})(f)-\mathcal{L}_{[0,j]}^{\mathcal{B}}(f_{1})_{\mid\re >0}
\in\mathcal{LO}_{[j,\infty]}(E)
\]
as well as 
\[
\sup_{\re(z)\geq -k}p_{\alpha}((\mathfrak{L}\circ R_{+}^{-1})(f)(z)-\mathcal{L}_{[0,j]}^{\mathcal{B}}(f_{1})(z))
e^{-\frac{1}{k}|z|+j|\re(z)|}<\infty 
\]
for all $k\in\N$ and $\alpha\in\mathfrak{A}$ by \prettyref{prop:laplace_decomposition} a) and 
\prettyref{prop:laplace_decomposition_langenbruch}. Using 
\begin{flalign*}
&\hspace{0.35cm}(\mathcal{L}\circ \mathcal{R}_{[0,\infty[}^{-1})(f)(z)-(\mathfrak{L}\circ R_{+}^{-1})(f)(z)\\
&=(\mathcal{L}\circ \mathcal{R}_{[0,\infty[}^{-1})(f)(z)-\mathcal{L}_{[0,j]}^{\mathcal{B}}(f_{1})(z)
-\bigl((\mathfrak{L}\circ R_{+}^{-1})(f)(z)-\mathcal{L}_{[0,j]}^{\mathcal{B}}(f_{1})(z)\bigr)
\end{flalign*} 
for $\re(z)>0$, this implies with $k=j$ that
\begin{flalign*}
&\hspace{0.35cm} \|(\mathcal{L}\circ \mathcal{R}_{[0,\infty[}^{-1})(f)
 -(\mathfrak{L}\circ R_{+}^{-1})(f)\|_{j,\alpha,\{\infty\}}\\
&=\sup_{\re(z)\geq\frac{1}{j}}p_{\alpha}((\mathcal{L}\circ \mathcal{R}_{[0,\infty[}^{-1})(f)(z)
 -(\mathfrak{L}\circ R_{+}^{-1})(f)(z))e^{-\frac{1}{j}|z|+j|\re(z)|}<\infty .
\end{flalign*}
We deduce that $(\mathcal{L}\circ \mathcal{R}_{[0,\infty[}^{-1})(f)-(\mathfrak{L}\circ R_{+}^{-1})(f)
\in\mathcal{LO}_{\{\infty\}}(E)$ and so $I_{0}\circ\mathfrak{L}_{\mathcal{B}}=\mathcal{L}^{\mathcal{B}}$.
\end{proof}

\begin{rem}
Due to \prettyref{thm:asymptotic_laplace_equiv}, \prettyref{prop:asymp_laplace_convolution} and 
\cite[Theorem 5.5, p.\ 56]{langenbruch2011} our definition of the convolution of two elements 
from $\mathcal{B}([0,\infty[)$ is consistent with the one given by Langenbruch.
\end{rem}

Let us turn to Komatsu's asymptotic Laplace transform. Again, we briefly recall the relevant notions and results.
For $0<\varphi<\tfrac{\pi}{2}$ and $r\geq 0$ we set 
\[
\Gamma_{r,\varphi}:=\{\rho e^{i\psi}\;|\; \rho\geq r,\,|\psi|\leq \varphi\}.
\]
An open set $U\subset\C$ is called \emph{postsectorial} (see \cite[p.\ 37]{lumer1999}, \cite[p.\ 150]{lumer2001}) if 
\[
\forall\;0<\varphi<\frac{\pi}{2}\;\exists\;r>0:\;\Gamma_{r,\varphi}\subset U.
\]

\begin{center}
\begin{minipage}{\linewidth}
\centering
\begin{tikzpicture}
\def\mypath{
 (6,2) -- (2,0.66) arc(18.43:-18.43:2.106) -- (6,-2)}
\fill[fill=black!10,draw=black,thick] \mypath;
\draw[dashed,thick]  (0,0) -- (2,0.66);
\draw[dashed,thick]  (0,0) -- (2,-0.66);
\node[anchor=north east] (A) at (5.5,1) {$\Gamma_{r,\varphi}$};
\node[anchor=north east] (B) at (1.7,0.45) {$\varphi$};
\node[anchor=north east] (C) at (1,-0.425) {$r$};
\draw[->] (-2,0) -- (6,0) node[right] {$\re(z)$} coordinate (x axis);
\draw[->] (0,-2) -- (0,2) node[above] {$\im(z)$} coordinate (y axis);
\end{tikzpicture}
\end{minipage}
\captionsetup{type=figure}
\caption{$\Gamma_{r,\varphi}$ for $0<\varphi<\tfrac{\pi}{2}$ and $r\geq 0$}
\end{center}

\begin{defn}[{\cite[p.\ 57]{langenbruch2011}}]
Let $(E,\|\cdot\|_{E})$ be a $\C$-Banach space, $a\in\R$ and $U\subset\C$ open and postsectorial. We define 
\[
\mathfrak{L}\mathcal{G}_{[a,\infty]}(U,E):=\{f\in\mathcal{O}(U,E)\,|\,
\forall\,j\in\N,0<\varphi<\tfrac{\pi}{2},r>0,\Gamma_{r,\varphi}\subset U:\:
\|f\|_{j,r,\varphi,[a,\infty]}^{\operatorname{Kom}}<\infty\}
\]
where
\[
\|f\|_{j,r,\varphi,[a,\infty]}^{\operatorname{Kom}}:=\sup_{z\in\Gamma_{r,\varphi}}\|f(z)\|_{E}
e^{-\frac{1}{j}|z|+a|\re(z)|},
\]
and 
\[
\mathfrak{L}\mathcal{G}_{\{\infty\}}(U,E):=\{f\in\mathcal{O}(U,E)\,|\,
\forall\,j\in\N,0<\varphi<\tfrac{\pi}{2},r>0,\Gamma_{r,\varphi}\subset U:\:
\|f\|_{j,r,\varphi,\{\infty\}}^{\operatorname{Kom}}<\infty\}
\]
where
\[
\|f\|_{j,r,\varphi,\{\infty\}}^{\operatorname{Kom}}:=\sup_{z\in\Gamma_{r,\varphi}}\|f(z)\|_{E}
e^{-\frac{1}{j}|z|+j|\re(z)|},
\]
as well as 
\[
\mathfrak{L}_{\operatorname{Kom}}\mathcal{B}^{\operatorname{exp}}_{[a,\infty]}(E)
:= \lim\limits_{\substack{\longrightarrow\\U}}\,\mathfrak{L}\mathcal{G}_{[a,\infty]}(U,E)
\quad\text{and}\quad 
\mathfrak{L}_{\operatorname{Kom}}\mathcal{B}^{\operatorname{exp}}_{\{\infty\}}(E)
:= \lim\limits_{\substack{\longrightarrow\\U}}\,\mathfrak{L}\mathcal{G}_{\{\infty\}}(U,E)
\]
where the inductive limits run over all open postsectorial sets $U\subset\C$. 
\end{defn}

Komatsu's asymptotic Laplace transform $\mathfrak{L}_{\operatorname{Kom}}$ is a linear isomorphism 
since the canonical (restriction) map 
\[
\rho_{a}\colon \mathcal{B}_{[a,\infty]}^{\operatorname{exp}}(E)/\mathcal{B}_{\{\infty\}}^{\operatorname{exp}}(E)
\to \mathcal{B}([a,\infty[,E),\; [F]\mapsto [F],
\]
is a linear isomorphism for $a\in\R$ by \cite[Theorem 1, p.\ 361]{Kom4} (cf.\ \cite[Theorem 3.5, p.\ 816]{Kom3}, 
\cite[Theorem 2, p.\ 61]{komatsu1988} with different proofs) and his Laplace transform 
\[
\mathfrak{L}_{\operatorname{Kom},a}\colon  \mathcal{B}_{[a,\infty]}^{\operatorname{exp}}(E)\to 
\mathfrak{L}_{\operatorname{Kom}}\mathcal{B}^{\operatorname{exp}}_{[a,\infty]}(E)
\]
for $a\in\R\cup\{\infty\}$ as well due to
\cite[Theorem 3.3, 3.4, p.\ 815-816]{Kom3} for $E=\C$ and \cite[p.\ 218]{Kom5} for $\C$-Banach spaces $E$. 

\begin{thm}[{\cite[Eq.\ (24), p.\ 217]{Kom5}}]\label{thm:asymp_laplace_iso_komatsu}
Let $(E,\|\cdot\|_{E})$ be a $\C$-Banach space and $a\in\R$. Then  
\begin{gather*}
\mathfrak{L}_{\operatorname{Kom}}\colon \mathcal{B}([a,\infty[,E)\to 
\mathfrak{L}_{\operatorname{Kom}}\mathcal{B}^{\operatorname{exp}}_{[a,\infty]}(E)/
\mathfrak{L}_{\operatorname{Kom}}\mathcal{B}^{\operatorname{exp}}_{\{\infty\}}(E),\\
\mathfrak{L}_{\operatorname{Kom}}(f):=[(\mathfrak{L}_{\operatorname{Kom},a}\circ \rho_{a}^{-1})(f)],
\end{gather*}
is a linear isomorphism.
\end{thm}

We note the following generalisation of \cite[Theorem 6.3, p.\ 59]{langenbruch2011} from $E=\C$ to general 
$\C$-Banach spaces $E$. 

\begin{thm}\label{thm:asymptotic_laplace_equiv_kom_lan}
Let $(E,\|\cdot\|_{E})$ be a $\C$-Banach space. Then the canonical map 
\[
I_{1}\colon \mathfrak{L}\mathcal{G}_{[0,\infty]}(E)/\mathfrak{L}\mathcal{G}_{\{\infty\}}(E)\to 
\mathfrak{L}_{\operatorname{Kom}}\mathcal{B}^{\operatorname{exp}}_{[0,\infty]}(E)/
\mathfrak{L}_{\operatorname{Kom}}\mathcal{B}^{\operatorname{exp}}_{\{\infty\}}(E),\; [F]\mapsto [F],
\]
is a linear isomorphism such that $I_{1}\circ\mathfrak{L}_{\mathcal{B}}=\mathfrak{L}_{\operatorname{Kom}}$.
\end{thm}
\begin{proof}
The map $I_{1}$ is well-defined due to the canonical inclusions $\mathfrak{L}\mathcal{G}_{[0,\infty]}(E)\subset 
\mathfrak{L}_{\operatorname{Kom}}\mathcal{B}^{\operatorname{exp}}_{[0,\infty]}(E)$ and 
$\mathfrak{L}\mathcal{G}_{\{\infty\}}(E)\subset 
\mathfrak{L}_{\operatorname{Kom}}\mathcal{B}^{\operatorname{exp}}_{\{\infty\}}(E)$. 

We only need to show that $I_{1}\circ\mathfrak{L}_{\mathcal{B}}=\mathfrak{L}_{\operatorname{Kom}}$ because 
$\mathfrak{L}_{\mathcal{B}}$ and $\mathfrak{L}_{\operatorname{Kom}}$ are linear isomorphisms by 
\prettyref{thm:asymp_laplace_iso_langenbruch} and \prettyref{thm:asymp_laplace_iso_komatsu}. 
We repeat the proof from \cite[Theorem 6.3, p.\ 59]{langenbruch2011} with small modifications. 
Let $f\in \mathcal{B}([0,\infty[,E)$ and $(\mathfrak{L}_{\operatorname{Kom},0}\circ \rho_{0}^{-1})(f)$ be defined on 
some postsectorial open set $U\subset\C$. For $j\in\N$ there are
$f_{1}\in\mathcal{B}([0,j],E)$ and $f_{2}\in\mathcal{B}([j,\infty[,E)$ 
such that $f=f_{1}+f_{2}$ and
\[
(\mathfrak{L}_{\operatorname{Kom},0}\circ \rho_{0}^{-1})(f)-\mathcal{L}_{[0,j]}^{\mathcal{B}}(f_{1})_{\mid U}
\in\mathfrak{L}_{\operatorname{Kom}}\mathcal{B}^{\operatorname{exp}}_{[j,\infty]}(E)
\]
as well as 
\[
\sup_{\re(z)\geq -k}\|(\mathfrak{L}\circ R_{+}^{-1})(f)(z)-\mathcal{L}_{[0,j]}^{\mathcal{B}}(f_{1})(z)\|_{E}
e^{-\frac{1}{k}|z|+j|\re(z)|}<\infty 
\]
for all $k\in\N$ by \cite[Theorem 3.9, p.\ 818]{Kom3} combined with \cite[p.\ 218]{Kom5} 
and \prettyref{prop:laplace_decomposition_langenbruch}. We note that
\begin{flalign*}
&\hspace{0.35cm}(\mathfrak{L}_{\operatorname{Kom},0}\circ \rho_{0}^{-1})(f)(z)-(\mathfrak{L}\circ R_{+}^{-1})(f)(z)\\
&=(\mathfrak{L}_{\operatorname{Kom},0}\circ \rho_{0}^{-1})(f)(z)-\mathcal{L}_{[0,j]}^{\mathcal{B}}(f_{1})(z)
-\bigl((\mathfrak{L}\circ R_{+}^{-1})(f)(z)-\mathcal{L}_{[0,j]}^{\mathcal{B}}(f_{1})(z)\bigr)
\end{flalign*} 
for $z\in U$. Let $0<\varphi<\tfrac{\pi}{2}$. Then there is $r>0$ such that $\Gamma_{r,\varphi}\subset U$ and 
with $k=j$ 
\begin{flalign*}
&\hspace{0.35cm} \|(\mathfrak{L}_{\operatorname{Kom},0}\circ \rho_{0}^{-1})(f)
 -(\mathfrak{L}\circ R_{+}^{-1})(f)\|_{j,r,\varphi,\{\infty\}}^{\operatorname{Kom}}\\
&=\sup_{z\in \Gamma_{r,\varphi}}\|(\mathfrak{L}_{\operatorname{Kom},0}\circ \rho_{0}^{-1})(f)(z)
 -(\mathfrak{L}\circ R_{+}^{-1})(f)(z)\|_{E}e^{-\frac{1}{j}|z|+j|\re(z)|}<\infty .
\end{flalign*}
We deduce that $(\mathfrak{L}_{\operatorname{Kom},0}\circ \rho_{0}^{-1})(f)-(\mathfrak{L}\circ R_{+}^{-1})(f)
\in\mathfrak{L}_{\operatorname{Kom}}\mathcal{B}^{\operatorname{exp}}_{\{\infty\}}(E)$ and so 
$I_{1}\circ\mathfrak{L}_{\mathcal{B}}=\mathfrak{L}_{\operatorname{Kom}}$.
\end{proof}

\begin{cor}\label{cor:asymptotic_laplace_equiv_kom}
Let $E$ be a $\C$-Banach space. Then the canonical map 
\[
I_{2}\colon \mathcal{LO}_{[0,\infty]}(E)/\mathcal{LO}_{\{\infty\}}(E)\to 
\mathfrak{L}_{\operatorname{Kom}}\mathcal{B}^{\operatorname{exp}}_{[0,\infty]}(E)/
\mathfrak{L}_{\operatorname{Kom}}\mathcal{B}^{\operatorname{exp}}_{\{\infty\}}(E),\; [F]\mapsto [F],
\]
is a linear isomorphism such that $I_{2}\circ I_{0}= I_{1}$ and 
$I_{2}\circ\mathcal{L}^{\mathcal{B}}=\mathfrak{L}_{\operatorname{Kom}}$.
\end{cor}
\begin{proof}
The map $I_{2}$ is well-defined due to the canonical inclusions $\mathcal{LO}_{[0,\infty]}(E)\subset 
\mathfrak{L}_{\operatorname{Kom}}\mathcal{B}^{\operatorname{exp}}_{[0,\infty]}(E)$ and 
$\mathcal{LO}_{\{\infty\}}(E)\subset 
\mathfrak{L}_{\operatorname{Kom}}\mathcal{B}^{\operatorname{exp}}_{\{\infty\}}(E)$. Clearly, 
$I_{2}\circ I_{0}= I_{1}$ and the rest of the statement follows from
\prettyref{thm:asymptotic_laplace_equiv} and \prettyref{thm:asymptotic_laplace_equiv_kom_lan}.
\end{proof}

It follows from the corollary above that our asymptotic Laplace transform satisfies condition (III) 
from the introduction as well.
In \cite{lumer1999} Lumer and Neubrander (cf.\ B\"aumer \cite{baeumer1997}) introduced an asymptotic Laplace transform 
$\mathfrak{L}_{0,\operatorname{LN}}$ on the space $L_{\operatorname{loc}}^{1}([0,\infty[,E)$ of (equivalence classes) 
of locally integrable functions on $[0,\infty[$ with values in a $\C$-Banach space $E$. 
They modified their Laplace transform in \cite[2.5 Definition, p.\ 156]{lumer2001} to a Laplace 
transform $\mathfrak{L}_{\operatorname{LN}}$ because the unmodified Laplace transform does not fulfil condition (II) 
from the introduction by \cite[Example 2.1, p.\ 153]{lumer2001}. Their modified Laplace transform 
$\mathfrak{L}_{\operatorname{LN}}$ coincides with Komatsu's Laplace transform on 
$L_{\operatorname{loc}}^{1}([0,\infty[,E)$ by \cite[Eq.\ (15), p.\ 157]{lumer2001}. 
Here we may regard $L_{\operatorname{loc}}^{1}([0,\infty[,E)$ as a linear subspace of $\mathcal{B}([0,\infty[,E)$ 
as in \cite[Theorem 1.3.10, p.\ 25]{Kan}.
Further, they clarified the relation between their unmodified Laplace transform and the one of Komatsu, namely 
that $\mathfrak{L}_{\operatorname{Kom}}(f)\subset\mathfrak{L}_{0,\operatorname{LN}}(f)$ 
for all $f\in L_{\operatorname{loc}}^{1}([0,\infty[,E)$ by \cite[3.1 Theorem, p.\ 157]{lumer2001}. 
In combination with \prettyref{thm:asymptotic_laplace_equiv} and \prettyref{thm:asymptotic_laplace_equiv_kom_lan} 
we note the following implication, which generalises \cite[Corollary 6.4, p.\ 60]{langenbruch2011}.

\begin{cor}\label{cor:asymptotic_laplace_lumer_neubrander}
Let $E$ be a $\C$-Banach space. Then we have 
\[
\mathcal{L}^{\mathcal{B}}(f)=\mathfrak{L}_{\mathcal{B}}(f)=\mathfrak{L}_{\operatorname{Kom}}(f)
=\mathfrak{L}_{\operatorname{LN}}(f)\subset \mathfrak{L}_{0,\operatorname{LN}}(f),
\quad f\in L_{\operatorname{loc}}^{1}([0,\infty[,E).
\]
\end{cor}

We omitted the linear isomorphisms $I_{j}$ in the equations above. 
\appendix
\renewcommand{\thesection}{\Alph{section}}  
\section{An alternative proof of \prettyref{thm:asymptotic_laplace_equiv}}
In this section we give an alternative proof that the map 
\[
I_{0}\colon \mathfrak{L}\mathcal{G}_{[0,\infty]}(E)/\mathfrak{L}\mathcal{G}_{\{\infty\}}(E)\to 
\mathcal{LO}_{[0,\infty]}(E)/\mathcal{LO}_{\{\infty\}}(E),\;[F]\mapsto [F],\;
\]
in our main result \prettyref{thm:asymptotic_laplace_equiv} of \prettyref{sect:Asymptotic_Laplace_transform_Lan} 
is a linear isomorphism if $E$ is a $\C$-Fr\'echet space.

\begin{defn}
Let $E$ be a $\C$-lcHs and $V_{0}:=-i\h=\C_{\re >0}$. We define the space
\[
\mathcal{E}_{\infty}(E):= \{f\in\mathcal{C}^{\infty}(\C,E)
\; | \; \forall\;n\in\N,\,m\in\N_{0},\,\alpha\in\mathfrak{A}:\;|f|_{n,m,\alpha}^{\mathcal{E}_{\infty}} < \infty\}
\]
where
\[
|f|_{n,m,\alpha}^{\mathcal{E}_{\infty}}:=\sup_{\substack{\re(z)>-n\\ \beta\in\N_{0}^{2}, |\beta|\leq m}}
p_{\alpha}(\partial^{\beta}f(z))e^{-\frac{1}{n}|z|+n|\re(z)|},
\]
and its topological subspaces
\[
\mathcal{E}_{\overline{V}_{0}}(E):= \{f\in\mathcal{E}_{\infty}(E)\;|\;
\operatorname{supp}f\subset\overline{V}_{0}\}
\]
and
\[
\mathcal{E}_{\overline{\partial}\overline{V}_{0}}(E):= \{f\in\mathcal{E}_{\infty}(E)\;|\;
\operatorname{supp}\overline{\partial}f\subset\overline{V}_{0}\}.
\]
\end{defn}

An alternative way to prove the surjectivity of $I_{0}$ is to find a solution 
$g\in\mathcal{E}_{\overline{\partial}\overline{V}_{0}}(E)$ of the Cauchy-Riemann equation 
$\overline{\partial}g=f$ for a given $f\in\mathcal{E}_{\overline{V}_{0}}(E)$ in the case that 
$E$ is a $\C$-Fr\'echet space. We start with a minimal modification of \cite[Lemma 6.1, p.\ 58]{langenbruch2011}.
We set $V_{n}:=\{z\in\C\;|\;\re(z)>-n\}$ for $n\in\N$ and 
\[
L^{2}(\mathfrak{L}\mathcal{G}_{\{\infty\}}):=\{f\in L^{2}_{\operatorname{loc}}(\C)\;|\;\forall\;n\in\N:\; 
|f|_{n}^{2}:=\int_{V_{n}}|f(z)|^{2}e^{-\frac{2}{n}|z|+2n|\re(z)|}\d z<\infty\}.
\]

\begin{lem}\label{lem:surj_dbar_L2}
For any $f\in L^{2}(\mathfrak{L}\mathcal{G}_{\{\infty\}})$ with $\operatorname{supp} f\subset \overline{V}_{0}$ 
there is $g\in L^{2}(\mathfrak{L}\mathcal{G}_{\{\infty\}})$ such that $\overline{\partial}g=f$.
\end{lem}
\begin{proof}
This is just \cite[Lemma 6.1, p.\ 58]{langenbruch2011} with the closure $\overline{V}_{0}$ instead of $V_{0}$. 
The proof of \cite[Lemma 6.1, p.\ 58]{langenbruch2011} does not change since the condition on the support is only used 
in part (a) of the proof in the estimate  
\begin{align*}
 \int_{\C}|g_{n}(z)|^{2}e^{-\frac{2|z|}{n}}(1+|z|^{2})^{-2}\d z 
&\leq\int_{\C}|f(z)|^{2}e^{-\frac{2|z|}{n}+2j|\re(z)|}\d z \\
&=\int_{V_{0}}|f(z)|^{2}e^{-\frac{2|z|}{n}+2j|\re(z)|}\d z 
\end{align*}
for $n\in\N$ which does not change if $V_{0}$ is replaced by $\overline{V}_{0}$, 
where $g_{n}\in L^{2}_{\operatorname{loc}}(\C)$ is a solution of $\overline{\partial}g_{n}(z)=f(z)e^{jz}$ 
for a fixed $j\in\N$. 
\end{proof}

\begin{cor}\label{cor:surj_dbar_E_scalar}
The Cauchy-Riemann operator
\[
\overline{\partial}\colon \mathcal{E}_{\overline{\partial}\overline{V}_{0}}(\C)\to \mathcal{E}_{\overline{V}_{0}}(\C)
\]
is surjective.
\end{cor}
\begin{proof}
(i) We note that the system of seminorms 
$(\sup_{\beta\in\N_{0}^{2},|\beta|\leq m}|\partial^{\beta}\cdot|_{n})_{n\in\N,m\in\N_{0}}$ with 
the seminorms $(|\cdot|_{n})_{n\in\N}$ from $L^{2}(\mathfrak{L}\mathcal{G}_{\{\infty\}})$ induces the
same topology on $\mathcal{E}_{\infty}(\C)$ as the system $(|\cdot|_{n,m}^{\mathcal{E}_{\infty}})_{n\in\N,m\in\N_{0}}$ 
by \cite[Lemma 3.6 (b), p.\ 9]{kruse2018_5} with $\Omega_{n}:=V_{n}$ and $\nu_{n}:=\widehat{\nu}_{n}\widetilde{\nu}_{n}$ 
since \cite[Condition (PN), p.\ 8]{kruse2018_5} is fulfilled for every $q\in\N$ 
by \cite[Remark 2.3 (a), p.\ 177]{kruse2018_4} with $\widehat{\nu}_{n}(z):=e^{-\frac{1}{n}|z|}$ 
and $\widetilde{\nu}_{n}(z):=e^{n|\re(z)|}$ due to \cite[Example 3.7 (b)(ii), p.\ 10-11]{kruse2018_5} 
for $\widehat{\nu}_{n}$ and \cite[Example 2.8 (i), p.\ 179]{kruse2018_4} for $\widetilde{\nu}_{n}$.
 
(ii) Let $f\in\mathcal{E}_{\overline{V}_{0}}(\C)$. It follows from part (i) 
that $f\in L^{2}(\mathfrak{L}\mathcal{G}_{\{\infty\}})$. Due to \prettyref{lem:surj_dbar_L2} there 
is $g\in L^{2}(\mathfrak{L}\mathcal{G}_{\{\infty\}})$ such that $\overline{\partial}g=f$. Since 
$\overline{\partial}$ is hypoelliptic and $f\in\mathcal{E}_{\overline{V}_{0}}(\C)$, 
we obtain that $g\in\mathcal{C}^{\infty}(\C)$ and $\operatorname{supp}\overline{\partial}g\subset\overline{V}_{0}$. 
From $|g|_{n}<\infty$ and $\sup_{\beta\in\N_{0}^{2},|\beta|\leq m}|\partial^{\beta}\overline{\partial}g|_{n}
=\sup_{\beta\in\N_{0}^{2},|\beta|\leq m}|\partial^{\beta}f|_{n}<\infty$ for all $n\in\N$ and $m\in\N_{0}$ we deduce 
that $g\in\mathcal{E}_{\infty}(\C)$ by \cite[Lemma 3.6 (a), p.\ 9]{kruse2018_5}. 
Hence we have $g\in\mathcal{E}_{\overline{\partial}\overline{V}_{0}}(\C)$, 
implying the surjectivity of $\overline{\partial}$.
\end{proof}

\begin{cor}\label{cor:surj_dbar_E_frechet}
Let $E$ be a $\C$-Fr\'echet space. Then the Cauchy-Riemann operator
\[
\overline{\partial}\colon \mathcal{E}_{\overline{\partial}\overline{V}_{0}}(E)\to \mathcal{E}_{\overline{V}_{0}}(E)
\]
is surjective.
\end{cor}
\begin{proof}
$\mathcal{E}_{\infty}(\C)$ is a Fr\'echet space by \cite[Proposition 3.7, p.\ 240]{kruse2018_2} and thus its 
closed subspaces $\mathcal{E}_{\overline{\partial}\overline{V}_{0}}(\C)$ and  $\mathcal{E}_{\overline{V}_{0}}(\C)$ 
as well (this is the reason why we phrased \prettyref{lem:surj_dbar_L2} with the closure $\overline{V}_{0}$ 
instead of $V_{0}$). $\mathcal{E}_{\infty}(\C)$ is nuclear by \cite[Theorem 3.1, p.\ 188]{kruse2018_4} and part (i) 
of the proof of \prettyref{cor:surj_dbar_E_scalar} and thus its subspaces 
$\mathcal{E}_{\overline{\partial}\overline{V}_{0}}(\C)$ and  $\mathcal{E}_{\overline{V}_{0}}(\C)$ too. 

By \cite[Theorem 14 (i), p.\ 1524]{kruse2017} the maps
\[
S\colon \mathcal{E}_{\overline{\partial}\overline{V}_{0}}(\C)\varepsilon E
\to\mathcal{E}_{\overline{\partial}\overline{V}_{0}}(E),\;u\longmapsto \bigl(z\mapsto u(\delta_{z})\bigr),
\]
and 
\[
S\colon\mathcal{E}_{\overline{V}_{0}}(\C)\varepsilon E
\to\mathcal{E}_{\overline{V}_{0}}(E),\;u\longmapsto \bigl(z\mapsto u(\delta_{z})\bigr),
\]
are topological isomorphisms like in \cite[Example 16 c), p.\ 1526]{kruse2017} 
where we only have to observe in addition that 
$\operatorname{supp}\overline{\partial}S(u)=\operatorname{supp}u(\delta_{(\cdot)}\circ\overline{\partial})
\subset\overline{V}_{0}$ for 
$u\in\mathcal{E}_{\overline{\partial}\overline{V}_{0}}(\C)\varepsilon E$ resp.\ 
$\operatorname{supp}S(u)\subset\overline{V}_{0}$ for $u\in\mathcal{E}_{\overline{V}_{0}}(\C)\varepsilon E$ and 
that $\operatorname{supp}\overline{\partial}(e'\circ f)
=\operatorname{supp}(e'\circ \overline{\partial}f)\subset\overline{V}_{0}$ for 
$f\in\mathcal{E}_{\overline{\partial}\overline{V}_{0}}(E)$ resp.\ 
$\operatorname{supp}(e'\circ f)\subset\overline{V}_{0}$ for $f\in\mathcal{E}_{\overline{V}_{0}}(E)$ and $e'\in E'$. 
This observation follows from $\delta_{z}\circ\overline{\partial}=0$ 
on $\mathcal{E}_{\overline{\partial}\overline{V}_{0}}(\C)$ resp.\ 
$\delta_{z}=0$ on $\mathcal{E}_{\overline{V}_{0}}(\C)$ for $z\notin\overline{V}_{0}$. 

Now, using the nuclearity, we obtain 
$\mathcal{E}_{\overline{\partial}\overline{V}_{0}}(\C)\widehat{\otimes}_{\pi} E\cong 
\mathcal{E}_{\overline{\partial}\overline{V}_{0}}(E)$ and 
$\mathcal{E}_{\overline{V}_{0}}(\C)\widehat{\otimes}_{\pi} E\cong \mathcal{E}_{\overline{V}_{0}}(E)$ where 
$\widehat{\otimes}_{\pi}$ denotes the completion of the projective tensor product. 
Further, the Cauchy-Riemann operator 
$\overline{\partial}\colon \mathcal{E}_{\overline{\partial}\overline{V}_{0}}(\C)\to \mathcal{E}_{\overline{V}_{0}}(\C)$
is a linear, continuous and surjective map between Fr\'echet spaces in the $\C$-valued case 
by \prettyref{cor:surj_dbar_E_scalar}. It follows like in the proof of \cite[Corollary 4.3, p.\ 14]{kruse2018_5} 
that the Cauchy-Riemann operator in the $E$-valued case is surjective as well.
\end{proof}

\begin{thm}\label{thm:alt_asymptotic_laplace_equiv}
Let $E$ be a $\C$-Fr\'echet space. Then the canonical map 
\[
I_{0}\colon \mathfrak{L}\mathcal{G}_{[0,\infty]}(E)/\mathfrak{L}\mathcal{G}_{\{\infty\}}(E)\to 
\mathcal{LO}_{[0,\infty]}(E)/\mathcal{LO}_{\{\infty\}}(E),\;[F]\mapsto [F],\;
\]
is a linear isomorphism.
\end{thm}
\begin{proof}
The map $I_{0}$ is well-defined as $\mathfrak{L}\mathcal{G}_{[0,\infty]}(E)\subset \mathcal{LO}_{[0,\infty]}(E)$ and 
$\mathfrak{L}\mathcal{G}_{\{\infty\}}(E)\subset\mathcal{LO}_{\{\infty\}}(E)$.

(i) Let $[F]\in\mathfrak{L}\mathcal{G}_{[0,\infty]}(E)/\mathfrak{L}\mathcal{G}_{\{\infty\}}(E)$ such that 
$I_{0}([F])=0$, which means that $F\in\mathfrak{L}\mathcal{G}_{[0,\infty]}(E)\cap\mathcal{LO}_{\{\infty\}}(E)$. 
Let $k\in\N$, $\alpha\in\mathfrak{A}$ and 
set $n:=k+2$. Then $n-\tfrac{1}{n}\geq k+\tfrac{1}{k}$ and we obtain 
\begin{flalign*}
&\hspace{0.35cm} \|F\|_{k,\alpha,\{\infty\}}^{\mathfrak{L}\mathcal{G}}\\
&=\sup_{\re(z)\geq -k}p_{\alpha}(F(z))e^{k|\re(z)|-\frac{1}{k}|z|}\\
&\leq \sup_{-k\leq\re(z)\leq\frac{1}{n}}p_{\alpha}(F(z))e^{k|\re(z)|-\frac{1}{k}|z|}
 +\sup_{\re(z)\geq\frac{1}{n}}p_{\alpha}(F(z))e^{k|\re(z)|-\frac{1}{k}|z|}\\
&\leq e^{k^{2}}\sup_{-k\leq\re(z)\leq\frac{1}{n}}p_{\alpha}(F(z))e^{-\frac{1}{k}|z|}
 +\sup_{\re(z)\geq\frac{1}{n}}p_{\alpha}(F(z))e^{k|\re(z)|-\frac{1}{k}|\im(z)|+\frac{1}{k}|\re(z)|}\\
&\leq e^{k^{2}} \|F\|_{k,\alpha,[0,\infty]}^{\mathfrak{L}\mathcal{G}}
 +\sup_{\re(z)\geq\frac{1}{n}}p_{\alpha}(F(z))e^{(k+\frac{1}{k})|\re(z)|-\frac{1}{k}|\im(z)|}\\
&\leq e^{k^{2}} \|F\|_{k,\alpha,[0,\infty]}^{\mathfrak{L}\mathcal{G}}
 +\sup_{\re(z)\geq\frac{1}{n}}p_{\alpha}(F(z))e^{(n-\frac{1}{n})|\re(z)|-\frac{1}{n}|\im(z)|}\\
&\leq e^{k^{2}} \|F\|_{k,\alpha,[0,\infty]}^{\mathfrak{L}\mathcal{G}}
 +\sup_{\re(z)\geq\frac{1}{n}}p_{\alpha}(F(z))e^{n|\re(z)|-\frac{1}{n}|z|}\\
&=e^{k^{2}} \|F\|_{k,\alpha,[0,\infty]}^{\mathfrak{L}\mathcal{G}}+\|F\|_{n,\alpha,\{\infty\}}<\infty.
\end{flalign*}
This implies that $F\in\mathfrak{L}\mathcal{G}_{\{\infty\}}(E)$ and thus the injectivity of $I_{0}$.

(ii) Let $[F]\in\mathcal{LO}_{[0,\infty]}(E)/\mathcal{LO}_{\{\infty\}}(E)$. We choose a cut-off function 
$\varphi\in\mathcal{C}^{\infty}(\C)$ such that $0\leq\varphi\leq 1$ on $\C$, 
$\varphi=1$ near $\{z\in\C\;|\;\re(z)\geq 1\}$, $\varphi=0$ near $\{z\in\C\;|\;\re(z)\leq \tfrac{1}{2}\}$ 
and 
\[
|\partial^{\beta}\varphi(z)|\leq C_{\beta}, \; z\in\C,\,\beta\in\N_{0}^{2},
\]
where $C_{\beta}>0$ only depends on $\beta$ (see \cite[Theorem 1.4.1, Eq.\ (1.4.2), p.\ 25]{H1}). 
Then we may regard $\overline{\partial}(\varphi F)$ as an element of $\mathcal{C}^{\infty}(\C,E)$ with 
$\operatorname{supp}\overline{\partial}(\varphi F)\subset \overline{V}_{0}$ by 
setting $\overline{\partial}(\varphi F)(z):=0$ for $\re(z)\leq 0$. Further, we note that for $n\in\N$, $m\in\N_{0}$ 
and $\alpha\in\mathfrak{A}$ it follows from the Leibniz rule that
\begin{align}\label{eq:estimate_cut_off_dbar}
 |\overline{\partial}(\varphi F)|_{n,m,\alpha}^{\mathcal{E}_{\infty}}
&=\sup_{\substack{\re(z)>-n\\ \beta\in\N_{0}^{2}, |\beta|\leq m}}
 p_{\alpha}(\partial^{\beta}\overline{\partial}(\varphi F)(z))e^{-\frac{1}{n}|z|+n|\re(z)|}\notag\\
&=\sup_{\substack{\frac{1}{2}\leq\re(z)\leq 1\\ \beta\in\N_{0}^{2}, |\beta|\leq m}}
 p_{\alpha}(\partial^{\beta}(\overline{\partial}(\varphi)F)(z))e^{-\frac{1}{n}|z|+n|\re(z)|}\notag\\
&\leq e^{n}\sup_{\substack{\frac{1}{2}\leq\re(z)\leq 1\\ \beta\in\N_{0}^{2}, |\beta|\leq m}}
 \sum_{\gamma\leq\beta}\dbinom{\beta}{\gamma}|\partial^{\beta-\gamma}\overline{\partial}\varphi(z)|
 p_{\alpha}(\partial^{\gamma}F(z))e^{-\frac{1}{n}|z|}\notag\\
&\leq(m!)^{2}e^{n}\sum_{|\gamma|\leq m+1}\sup_{\tfrac{1}{2}\leq\re(z)\leq 1}|\partial^{\gamma}\varphi(z)|
 \sup_{\substack{\frac{1}{2}\leq\re(z)\leq 1\\ \beta\in\N^{2}_{0},|\beta|\leq m}}
 p_{\alpha}(\partial^{\beta}F(z))e^{-\frac{1}{n}|z|}\notag\\
&\leq (m!)^{2}e^{n}\sum_{|\gamma|\leq m+1}C_{\gamma}
 \sup_{\substack{\frac{1}{2}\leq\re(z)\leq 1\\ \beta\in\N^{2}_{0},|\beta|\leq m}}
 p_{\alpha}(\partial^{\beta}F(z))e^{-\frac{1}{n}|z|}. 
\end{align}
Next, we observe the relation 
$\partial^{\beta}F(z)=i^{\beta_{2}}\partial^{|\beta|}_{\C}F(z)$ for $\re(z)>0$ 
and $\beta=(\beta_{1},\beta_{2})\in\N_{0}^{2}$ between the real partial derivative 
$\partial^{\beta}F$ and the complex derivative $\partial^{|\beta|}_{\C}F$ 
(see \cite[Proposition 7.1, p.\ 270]{kruse2019_4}).
We fix $0<r<\tfrac{1}{2}$ and choose $k\in\N$ such that $\tfrac{1}{k}<\tfrac{1}{2}-r$ and $k\geq n$. 
We remark that
\[
-\frac{1}{n}|z|\leq -\frac{1}{n}|\zeta|+\frac{1}{n}|z-\zeta|\leq -\frac{1}{k}|\zeta|+\frac{r}{n},
\quad z,\zeta\in\C,\;|\zeta-z|=r.
\]
By Cauchy's inequality \cite[Proposition 2.5, p.\ 57]{dineen1981} 
we have for $\beta\in\N_{0}^{2}$
\[
 p_{\alpha}(\partial^{\beta}F(z))
=p_{\alpha}(\partial_{\C}^{|\beta|}F(z)) 
\leq\frac{|\beta|!}{r^{|\beta|}}\sup_{\zeta\in\C,|\zeta-z|=r}p_{\alpha}(F(\zeta)),\quad \frac{1}{2}\leq\re(z)\leq 1,
\]
which implies 
\begin{flalign*}
&\hspace{0.35cm}\sup_{\substack{\frac{1}{2}\leq\re(z)\leq 1\\ \beta\in\N^{2}_{0},|\beta|\leq m}}
 p_{\alpha}(\partial^{\beta}F(z))e^{-\frac{1}{n}|z|}
 \leq\sup_{\substack{\frac{1}{2}\leq\re(z)\leq 1\\ \beta\in\N^{2}_{0},|\beta|\leq m}}
 \frac{|\beta|!}{r^{|\beta|}}\sup_{\zeta\in\C,|\zeta-z|=r}p_{\alpha}(F(\zeta))e^{-\frac{1}{n}|z|}\\
&\leq e^{\frac{r}{n}}\frac{m!}{r^{m}}\sup_{\re(\zeta)\geq\frac{1}{k}}p_{\alpha}(F(\zeta))e^{-\frac{1}{k}|\zeta|}
 = e^{\frac{r}{n}}\frac{m!}{r^{m}}\|F\|_{k,\alpha,[0,\infty]}.
\end{flalign*}
In combination with \eqref{eq:estimate_cut_off_dbar} we derive that 
\[
 |\overline{\partial}(\varphi F)|_{n,m,\alpha}^{\mathcal{E}_{\infty}}
\leq \frac{(m!)^{3}}{r^{m}}e^{n+\frac{r}{n}}\sum_{|\gamma|\leq m+1}C_{\gamma}\|F\|_{k,\alpha,[0,\infty]}<\infty
\]
since $F\in\mathcal{LO}_{[0,\infty]}(E)$, yielding 
$\overline{\partial}(\varphi F)\in\mathcal{E}_{\overline{V}_{0}}(E)$. 
Hence there is $g\in\mathcal{E}_{\overline{\partial}\overline{V}_{0}}(E)$ such that 
$\overline{\partial}g=\overline{\partial}(\varphi F)$ on $\C$ by \prettyref{cor:surj_dbar_E_scalar}.

Now, we set $F_{1}:=\varphi F-g$ and $F_{2}:=(1-\varphi)F+g$ and show that 
$F_{1}$ may be regarded as an element of $\mathfrak{L}\mathcal{G}_{[0,\infty]}(E)$ by setting 
$(\varphi F)(z):=0$ for $\re(z)\leq 0$ and that $F_{2}\in\mathcal{LO}_{\{\infty\}}(E)$. 
We have $\overline{\partial}F_{1}=0$ on $\C$ and $\overline{\partial}F_{2}=0$ on $-i\h$, 
which implies $F_{1}\in\mathcal{O}(\C,E)$ and $F_{2}\in\mathcal{O}(-i\h,E)$. 
Let $k\in\N$ and $\alpha\in\mathfrak{A}$ and choose $n\in\N$ such that $n\geq\max(2,k)$. Then we get
\begin{align*}
 \|F_{1}\|_{k,\alpha,[0,\infty]}^{\mathfrak{L}\mathcal{G}}
&=\sup_{\re(z)\geq -k}p_{\alpha}(F_{1}(z))e^{-\frac{1}{k}|z|}\\
&\leq \sup_{\re(z)\geq \frac{1}{2}}|\varphi(z)| p_{\alpha}(F(z))e^{-\frac{1}{k}|z|}
 +\sup_{\re(z)\geq -k}p_{\alpha}(g(z))e^{-\frac{1}{k}|z|}\\
&\leq \sup_{\re(z)\geq \frac{1}{2}} p_{\alpha}(F(z))e^{-\frac{1}{k}|z|}
 +\sup_{\re(z)\geq -k}p_{\alpha}(g(z))e^{-\frac{1}{k}|z|}\\
&\leq \sup_{\re(z)\geq \frac{1}{n}} p_{\alpha}(F(z))e^{-\frac{1}{n}|z|}
 +\sup_{\re(z)> -n}p_{\alpha}(g(z))e^{-\frac{1}{n}|z|+n|\re(z)|}\\
&=\|F\|_{n,\alpha,[0,\infty]}+|g|_{n,0,\alpha}^{\mathcal{E}_{\infty}}<\infty
\end{align*}
because $F\in\mathcal{LO}_{[0,\infty]}(E)$ and $g\in\mathcal{E}_{\overline{\partial}\overline{V}_{0}}(E)$. 
Thus we get $F_{1}\in\mathfrak{L}\mathcal{G}_{[0,\infty]}(E)$. 

Let us turn to $F_{2}$. We have
\begin{flalign*}
&\hspace{0.35cm} \|F_{2}\|_{k,\alpha,\{\infty\}}\\
&=\sup_{\re(z)\geq \frac{1}{k}} p_{\alpha}(F_{2}(z))e^{-\frac{1}{k}|z|+k|\re(z)|}\\
&\leq \sup_{\re(z)\geq \frac{1}{k}} p_{\alpha}((1-\varphi(z))F(z))e^{-\frac{1}{k}|z|+k|\re(z)|}
 +\sup_{\re(z)\geq \frac{1}{k}} p_{\alpha}(g(z))e^{-\frac{1}{k}|z|+k|\re(z)|}\\
&\leq \sup_{\frac{1}{k}\leq\re(z)\leq 1}|1-\varphi(z)|p_{\alpha}(F(z))e^{-\frac{1}{k}|z|+k|\re(z)|}
 +\sup_{\re(z)> -k} p_{\alpha}(g(z))e^{-\frac{1}{k}|z|+k|\re(z)|}\\
&\leq e^{k}\sup_{\frac{1}{k}\leq\re(z)\leq 1}p_{\alpha}(F(z))e^{-\frac{1}{k}|z|}
  +|g|_{k,0,\alpha}^{\mathcal{E}_{\infty}}
 \leq e^{k}\sup_{\re(z)\geq \frac{1}{k}}p_{\alpha}(F(z))e^{-\frac{1}{k}|z|}
  +|g|_{k,0,\alpha}^{\mathcal{E}_{\infty}}\\
&= e^{k}\|F\|_{k,\alpha,[0,\infty]}+|g|_{k,0,\alpha}^{\mathcal{E}_{\infty}}<\infty
\end{flalign*}
because $F\in\mathcal{LO}_{[0,\infty]}(E)$ and $g\in\mathcal{E}_{\overline{\partial}\overline{V}_{0}}(E)$. 
Hence we have $F_{2}\in\mathcal{LO}_{\{\infty\}}(E)$. 
Altogether, we deduce that $I_{0}([F_{1}])=[F_{1}]=[F_{1}+F_{2}]=[F]$, which means that $I_{0}$ is surjective.
\end{proof}

\bibliography{biblio}
\bibliographystyle{plainnat}
\end{document}